\documentclass[10pt]{amsart}
\usepackage{amsthm}
\usepackage{times,amsmath,amsbsy,amssymb,amscd,mathrsfs}
\usepackage{graphicx,subfigure,epstopdf,wrapfig,chemarrow}

\usepackage{algorithm2e} 
\usepackage{multicol,multirow}
\usepackage{mathtools}
\usepackage[usenames,dvipsnames,svgnames,table]{xcolor}
\usepackage[all]{xy}
\usepackage{wrapfig}
\usepackage{tcolorbox}

\usepackage{tikz,tikz-cd}
\usepackage[utf8]{inputenc}
\usepackage{pgfplots} 
\usepackage{pgfgantt}
\usepackage{pdflscape}
\pgfplotsset{compat=newest} 
\pgfplotsset{plot coordinates/math parser=false}
\newlength\fwidth

\definecolor{myBlue}{rgb}{0.0,0.0,0.55}
\usepackage[pdftex,colorlinks=true,citecolor=myBlue,linkcolor=myBlue]{hyperref}

\usepackage[hyperpageref]{backref}

\usepackage{comment,enumerate,multicol,xspace}

  \newcounter{mnote}
  \let\oldmarginpar\marginpar
    \renewcommand\marginpar[1]{\-\oldmarginpar[\raggedleft\footnotesize #1]%
    {\raggedright\footnotesize #1}}

\newtheorem{theorem}{Theorem}[section]
\newtheorem{lemma}[theorem]{Lemma}
\newtheorem{corollary}[theorem]{Corollary}
\newtheorem{proposition}[theorem]{Proposition}

\newtheorem{example}[theorem]{Example}

\newtheorem{remark}[theorem]{Remark}

\newcommand{\dx}{\,{\rm d}x}
\newcommand{\dd}{\,{\rm d}}

\newcommand{\bs}{\boldsymbol}

\renewcommand{\div}{\operatorname{div}}
\newcommand{\grad}{{\rm grad\,}}

\newcommand{\tr}{\operatorname{tr}}
\newcommand{\dev}{\operatorname{dev}}
\newcommand{\sym}{\operatorname{sym}}

\newcommand{\alt}{\operatorname{Alt}}

\newcommand{\step}[1]{\noindent\raisebox{1.5pt}[10pt][0pt]{\tiny\framebox{$#1$}}\xspace}

\newcommand{\vertiii}[1]{{\left\vert\kern-0.25ex\left\vert\kern-0.25ex\left\vert #1 
    \right\vert\kern-0.25ex\right\vert\kern-0.25ex\right\vert}}

\usepackage{stmaryrd}
\usepackage{scalefnt}
\usepackage{graphicx} 

\newcommand{\Oplus}{\ensuremath{\vcenter{\hbox{\scalebox{1.5}{$\oplus$}}}}}
\newcommand{\prox}{\operatorname{Prox}}

\usepackage{geometry}
\begin{document}
\title{$H(\div)$-conforming Finite Element Tensors with Constraints}
\author{Long Chen}%
 \address{Department of Mathematics, University of California at Irvine, Irvine, CA 92697, USA}%
 \email{chenlong@math.uci.edu}%
 \author{Xuehai Huang}%
 \address{Corresponding author. School of Mathematics, Shanghai University of Finance and Economics, Shanghai 200433, China}%
 \email{huang.xuehai@sufe.edu.cn}%

\thanks{The first author was supported by DMS-2012465 and DMS-2309785.}
\thanks{The second author was supported by the National Natural Science Foundation of China Project 12171300, and the Natural Science Foundation of Shanghai 21ZR1480500.}

\makeatletter
\@namedef{subjclassname@2020}{\textup{2020} Mathematics Subject Classification}
\makeatother

\begin{abstract}
A unified construction of $H(\div)$-conforming finite element tensors, including vector  element, symmetric  matrix element, traceless  matrix element, and, in general, tensors with linear constraints, is developed in this work. It is based on the geometric decomposition of Lagrange elements into bubble functions on each sub-simplex. Each tensor at a sub-simplex is decomposed into tangential and normal components. The tangential component forms the bubble function space, while the normal component characterizes the trace. Some degrees of freedom can be redistributed to $(n-1)$-dimensional faces. The developed finite element spaces are $H(\div)$-conforming and satisfy the discrete inf-sup condition. Intrinsic bases of the constraint tensor space are also established. 
\end{abstract}
\maketitle


\section{Introduction}
Hilbert complexes play a fundamental role in the theoretical analysis and the design of stable numerical methods for partial differential equations~\cite{Arnold2006,Arnold2010,arnoldFiniteElementExterior2018,ChenHuang2018}. Recently, in~\cite{arnoldComplexesComplexes2021}, Arnold and Hu have developed a systematic approach to derive new Hilbert complexes from well-understood differential complexes, such as the de Rham complex. In space $\mathbb{R}^n$, for $0 \leq k \leq n$, let $\alt^{k,n-1}:=\alt^{k}\otimes\alt^{n-1}$ be the tensor product of alternating multilinear functional spaces, $H^s$ be the standard Sobolev space with real index $s$, and $\kappa_k$ be the Koszul operator for the de Rham complex. Below, we rotate the right end column of the Bernstein-Gelfand-Gelfand (BGG) diagram in~\cite{arnoldComplexesComplexes2021} and switch the ordering of the index in~\cite{arnoldComplexesComplexes2021} to match the row action of the operator $\div$.
\begin{figure}[htbp]
\begin{center}
\includegraphics[width=5.6in]{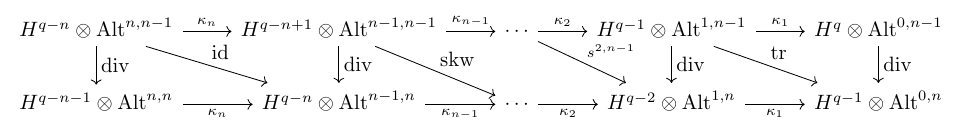}
\label{fig:divXdiagram}
\end{center}
\end{figure}

The algebraic operator  \( s^{k, n-1}: \alt^{k,n-1} \rightarrow \alt^{k-1, n} \), for \( 1 \leq k \leq n \) along the $\searrow$ direction, is defined as
\begin{align*}
s^{k, n-1} \omega \left(w_{1}, \cdots, w_{k-1}\right) \left(v_{1}, \cdots, v_{n}\right) 
& := \sum_{i=1}^{n} (-1)^{i-1} \omega \left(v_{i}, w_{1}, \cdots, w_{k-1}\right) \left(v_{1}, \cdots, \widehat{v}_{i}, \cdots, v_{n}\right) \\
& \forall~v_{1}, \cdots, v_{n}, w_{1}, \cdots, w_{k-1} \in \mathbb{R}^{n}.
\end{align*}
The tensor space \( \mathbb{X} \) is defined as
$$
\mathbb{X} := \ker(s^{k, n-1}) \cap \alt^{k,n-1}, \quad 1 \leq k \leq n-1.
$$
A tensor \( \omega \in \alt^{k,n-1} \) can be represented by a matrix \( \bs{A} = (a_{\sigma, i}) \in \mathbb{R}^{{n \choose k} \times n} \), which is called a proxy matrix, and then \( \mathbb{X} \) is a subset of matrices satisfying certain linear constraints. For simplicity, we will use matrix \( \bs{A} \) to represent an element in \( \mathbb{X} \). Given a domain \( \Omega \subset \mathbb{R}^n \), define \( H(\div, \Omega; \mathbb{X}) \) space as
\begin{equation*}
H(\div, \Omega; \mathbb{X}) = \big\{ \bs{A} \in L^2(\Omega; \mathbb{X}): \div \bs{A} \in L^2(\Omega; \mathbb{R}^{{n \choose k}}) \big\},
\end{equation*}
where the divergence operator is applied to each row of \( \bs{A} \).

Notable examples are \( H(\div; \mathbb{S}) \) with the symmetric matrix \( \mathbb{S} \), which plays an important role in the discretization of the elasticity equation in the mixed form~\cite{JohnsonMercier1978,ArnoldDouglasGupta1984,Arnold.D;Winther.R2002}, and \( H(\div; \mathbb{T}) \) with the traceless matrix \( \mathbb{T} \), which is used in the mixed form of the linearized Einstein-Bianchi system \cite{QuennevilleBelair2015,Chen;Huang:2020Discrete,HuLiang2021}.

Given a simplicial mesh \( \mathcal{T}_h \) of \( \Omega \), we shall develop a systematic construction of \( H(\div) \)-conforming finite element subspaces of \( H(\div, \Omega; \mathbb{X}) \) space. Take a proxy matrix \( \bs{A} \) for an elementwise smooth tensor \( \omega \in L^2(\Omega; \mathbb{X}) \). \( H(\div) \)-conformity means \( \bs{A} \boldsymbol{n}_F \) is continuous across each \( (n-1) \)-dimensional face \( F \) of \( \mathcal{T}_h \). Namely, each row of \( \bs{A} \) is an \( H(\div) \)-conforming vector function.

It is the constraint \( s^{k, n-1}(\bs{A}) = 0 \) that makes the finite element construction difficult, as the constraint and the normal continuity should be satisfied simultaneously. For example, the symmetry of the tensors makes it a challenging problem to construct \( H(\div; \mathbb{S}) \)-conforming finite elements. Arnold and Winther~\cite{Arnold.D;Winther.R2002} constructed such an element in two dimensions, and later it was extended to higher dimensions in~\cite{HuZhang2016,ArnoldAwanouWinther2008,AdamsCockburn2005}. Hu and Zhang~\cite{HuZhang2015,Hu2015} constructed \( H(\div; \mathbb{S}) \)-conforming finite elements with full polynomial spaces with matched order of approximation.

The approach we shall use is the geometric decomposition, which leads to explicit bases for finite elements. The geometric decomposition is an important tool for finite element analysis. For example, it is used in~\cite{Falk2014} to construct a local and bounded co-chain projection to the discrete de Rham complexes. The finite element system in~\cite{ChristiansenMunthe-KaasOwren2011} also originates from the geometric decomposition. The geometric decomposition of standard finite element de Rham complexes is well-studied in~\cite{ArnoldFalkWinther2009,Arnold2006,GopalakrishnanGarciaCastilloDemkowicz2005}, and in~\cite{christiansenNodalFiniteElement2018} for nodal finite element de Rham complexes. Recently geometric decomposition has been extended to smooth finite elements and smooth finite element de Rham and Stokes complexes~\cite{ChenHuang2021Cmgeodecomp,Chen;Huang:2022FEMcomplex3D}.

We shall integrate the geometric decomposition of the Lagrange element with tangential-normal ($t$-$n$) decompositions on subsimplices to produce geometric decompositions of \( H(\div) \)-conforming finite element vectors and tensors, exhibiting normal continuity across all $(n-1)$-dimensional faces. In a similar way, hierarchical geometric decomposition of \( H(\div) \)-conforming finite elements in two and three dimensions is discussed in~\cite{DeSiqueiraDevlooGomes2010,DeSiqueiraDevlooGomes2013,CastroDevlooFariasGomesEtAl2016}. A geometric decomposition of \( H(\div) \)-conforming finite element vectors with a different tangential-normal decomposition in three dimensions is also shown in~\cite{christiansenNodalFiniteElement2018}. While these prior studies offer similar insights, our methodology introduces a novel level of generality. A significant aspect of our contribution is the expansion of geometric decomposition techniques to effectively manage tensors subjected to specific constraints.

To satisfy the constraint while still keeping normal continuity, the crucial step is to get a \( t \)-\( n \) decomposition of \( \mathbb{X} \) with respect to the second component in \( \alt^{k, n-1} \), i.e.,
\begin{equation*}
\mathbb{X} =  \mathscr{T}^f(\mathbb{X}) \oplus \mathscr{N}^f(\mathbb{X}),
\end{equation*}
where \( \mathscr{T}^f(\mathbb{X}) = \left( \alt^k \otimes \mathscr{T}^f \right) \cap \ker (s^{k,n-1}) \) and \( \mathscr{N}^f(\mathbb{X}) = \pi_{\mathbb{X}} (\alt^k \otimes \mathscr{N}^f) \) with an oblique (non-orthogonal) projection operator \( \pi_{\mathbb{X}}: \alt^k \otimes \mathscr{N}^f \to \mathbb{X} \) so that one constraint is used only once either in \( \mathscr{T}^f(\mathbb{X}) \) or \( \mathscr{N}^f(\mathbb{X}) \).

This induces the geometric decomposition
\begin{equation}\label{intro:PrXdec}
\mathbb{P}_r(T; \mathbb{X}) = \Oplus_{\ell = 0}^n \Oplus_{f \in \Delta_{\ell}(T)} b_f \mathbb{P}_{r-(\ell+1)}(f) \otimes \left[ \mathscr{T}^f(\mathbb{X}) \oplus \mathscr{N}^f(\mathbb{X}) \right].
\end{equation}
The tangential component will contribute to the polynomial bubble space 
$$ \mathbb{B}_r(\div,T; \mathbb{X}) := \Oplus_{\ell = 1}^n \Oplus_{f \in \Delta_{\ell}(T)} \left[ b_f \mathbb{P}_{r-(\ell+1)}(f) \otimes \mathscr{T}^f(\mathbb{X}) \right], $$ 
and the normal component \( b_f \mathbb{P}_{r-(\ell+1)}(f) \otimes \mathscr{N}^f(\mathbb{X}) \) to the trace.

As a direct result of decomposition~\eqref{intro:PrXdec}, the following DoFs
\begin{subequations}\label{intro:XDof}
\begin{align}
\label{intro:bdXDof0}
\omega(\texttt{v}_i), 
& \quad \quad~i = 0,\ldots, n, \omega\in \mathbb X,\\
\label{intro:bdXDof1}
(\omega, \eta)_f, 
& \quad \quad~\eta \in \mathbb{P}_{r-(\ell +1)}(f) \otimes \mathscr{N}^f(\mathbb{X}), \quad f \in \Delta_{\ell}(\mathcal{T}_h), \quad \ell = 1,\ldots, n-1, \\
\label{intro:bubbleXDof} 
(\omega, \eta)_T, 
& \quad \quad~\eta \in \mathbb{B}_r(\div,T; \mathbb{X}), \quad T \in \mathcal{T}_h,
\end{align}
\end{subequations}
will determine a space \( V_h \subset H(\div, \Omega; \mathbb{X}) \). Here we single out the vertex DoFs to emphasize the finite element function is continuous on vertices.

Discrete inf-sup condition will be established with requirement $r\geq n+1$ and with modification of DoFs for $r\geq k+1$ for $1\leq k\leq n-2$. Variants can be constructed by further tuning DoFs~\eqref{intro:XDof}, which will recover the existing $H(\div;\mathbb S)$ elements~\cite{Chen;Huang:2021divFinite,Hu2015,HuZhang2015} and $H(\div;\mathbb T)$ elements~\cite{Chen;Huang:2020Discrete,HuLiang2021}. 

The geometric decomposition approach in this paper is not easy to extend to the case $\mathbb X=\mathbb S\cap\mathbb T$, which requires special care and super-smoothness of DoFs; see the recent work~\cite{HuLinShi2023}.

The rest of this paper is organized as follows. Section~\ref{sec:pre} covers foundational concepts, including simplex, barycentric coordinates, Bernstein polynomials, and a geometric decomposition of Lagrange elements.
Sections~\ref{sec:geodecompdivvec} and~\ref{sec:geodecompdivmatrix} explore the
geometric decompositions of vector face elements and matrix face elements,
respectively. Section~\ref{sec:divtensorspace} focuses on the constraint tensor space $\mathbb X$ and its bases. The geometric decomposition of $H(\div)$-conforming tensors is developed in Section~\ref{sec:geodecompdivtensor}. As the language of differential form is abstract, in the first few sections we present the results using vector and matrix language and then move to the differential forms in Sections \ref{sec:divtensorspace} and \ref{sec:geodecompdivtensor}.

\section{Notation and background}\label{sec:pre}
We summarize the most important notation and integer indices in the beginning:
\begin{itemize}
 \item $\mathbb R^n: n$ is the dimension of the ambient Euclidean space and $n\geq 2$;
\item $\mathbb P_r: r$ is the degree of the polynomial and $r\geq 0$;
 \item $\Lambda^k: k$ is the order of the differential form and $0\leq k\leq n$;
\item $\Delta_{\ell}(T): \ell$ is the dimension of a sub-simplex $f\in \Delta_{\ell}(T)$ and $0\leq \ell \leq n$.
\end{itemize}

\subsection{Simplex and sub-simplices}
Let $T \subset \mathbb{R}^{n}$ be an $n$-dimensional simplex with vertices $\texttt{v}_{0}, \texttt{v}_{1}, \ldots$, $\texttt{v}_{n}$ in general position. 
Following~\cite{ArnoldFalkWinther2009}, we let $\Delta(T)$ denote all the subsimplices of $T$, while $\Delta_{\ell}(T)$ denotes the set of subsimplices of dimension $\ell$, for $0\leq \ell \leq n$. 

For a sub-simplex $f\in \Delta_{\ell}(T)$, we will overload the notation $f$ for both the geometric simplex and the algebraic set of indices. Namely on one hand $f = \{f(0), \ldots, f(\ell)\}\subseteq \{0, 1, \ldots, n\}$, and on the other hand 
\[
f ={\rm Convex}(\texttt{v}_{f(0)}, \ldots, \texttt{v}_{f(\ell)}) \in \Delta_{\ell}(T)
\]
is the $\ell$-dimensional simplex spanned by the vertices $\texttt{v}_{f(0)}, \ldots, \texttt{v}_{f( \ell)}$. If $f \in \Delta_{\ell}(T)$, for $\ell = 0, \ldots, n-1$, then $f^{*} \in \Delta_{n- \ell-1}(T)$ denotes the sub-simplex of $T$ opposite to $f$. Algebraically treating $f$ as a subset of $\{0, 1, \ldots, n\}$, $f^*\subseteq \{0,1, \ldots, n\}$ so that $f\cup f^* = \{0, 1, \ldots, n\}$, i.e., $f^*$ is the complement of set $f$. Geometrically,
\[
f^* ={\rm Convex}(\texttt{v}_{f^*(1)}, \ldots, \texttt{v}_{f^*(n-\ell)}) \in \Delta_{n- \ell-1}(T)
\]
is the $(n- \ell-1)$-dimensional simplex spanned by vertices not contained in $f$. We refer to~\cite[Fig. 2]{Chen;Huang:2022FEMcomplex3D} for an illustration of $f$ and $f^*$.

Denote by $F_{i}$ the $(n-1)$-dimensional face opposite to vertex $\texttt{v}_i$, i.e., $F_i = \{ i\}^*$. Here capital $F$ is reserved for an $(n-1)$-dimensional face of $T$. For lower dimensional sub-simplices, we sometimes use more conventional notation. For example, the vertex will be denoted by $\texttt{v}_i$ and the edge formed by $\texttt{v}_i$ and $\texttt{v}_j$ will be denoted by $\texttt{e}_{ij}$. 


\subsection{$t$-$n$ bases}
For an $\ell$-dimensional sub-simplex $f\in \Delta_{\ell}(T)$, choose 
$\ell$ linearly independent tangential vectors $\{\bs t_1^f, \ldots, \bs t_{\ell}^f\}$ of $f$ and $n - \ell$ linearly independent normal vectors $\{\bs n_1^f, \ldots, \bs n_{n-\ell}^f\}$ of $f$.  
The set of $n$ vectors $\{\bs t_1^f, \ldots, \bs t_{\ell}^f$, $\bs n_1^f, \ldots, \bs n_{n-\ell}^f \}$ forms a basis of $\mathbb R^n$.
Notice that for $\ell = 0$, i.e., at vertices, there are no tangential vectors, and for $\ell = n$, there are no normal vectors. Define the tangent plane and normal plane of $f$ as 
\begin{equation*}
\mathscr T^f := {\rm span} \{\bs t_i^f, i=1, \ldots, \ell \},
\quad
\mathscr{N}^f :=  {\rm span} \{\bs n_i^f, i=1, \ldots, n - \ell \}.
\end{equation*}
All vectors are normalized but $\{\bs t_i^f\}$ or $\{ \bs n_i^f\}$ may not form an orthonormal basis. 

Inside the subspace $\mathscr T^f$, we can find a basis $\{\hat{\bs t}_1^f, \ldots, \hat{\bs t}_{\ell}^f\}$ dual to $\{\bs t_1^f, \ldots, \bs t_{\ell}^f\}$, i.e., $\hat{\bs t}_i\in \mathscr T^f$ and $(\hat{\bs t}_i, \bs t_j)=\delta_{i,j}$ for $i, j=1,\ldots,\ell$. Similarly we have a basis $\{\hat{\bs n}_1^f, \ldots, \hat{\bs n}_{n-\ell}^f\}$ of $\mathscr N^f$ and  $(\hat{\bs n}_i, \bs n_j)=\delta_{i,j}$ for $i, j=1,\ldots, n-\ell$. As $\mathscr T^f \perp \mathscr N^f$, the basis  $\{\hat{\bs t}_1^f, \ldots, \hat{\bs t}_{\ell}^f, \hat{\bs n}_{1}^f, \ldots, \hat{\bs n}_{n-\ell}^f\}$ is also dual to  $\{\bs t_1^f, \ldots, \bs t_{\ell}^f, \bs n_{1}^f, \ldots, \bs n_{n-\ell}^f\}$. 

Given a sub-simplex $f\in \Delta_{\ell}(T)$, we now present two bases for its normal plane $\mathscr N^f$ constructed in~\cite{ChenChenHuangWei2023}. 

Recall that we label $F_i$ as the $(n-1)$-dimensional face opposite to the $i$-th vertex. Then $f\subseteq F_i$ for $i\in f^*$. One basis is composed by unit normal vectors of all such $(n-1)$-dimensional faces:
\begin{equation*}
\{\bs n_{F_i}, i\in f^*\},
\end{equation*}
and will be called the {\it face normal basis}. 

We now give its dual basis in $\mathscr N^f$. 
For $f\in \Delta_{\ell}(T), \ell = 0,1,\ldots, n-1$ and $i\in f^*$, let $f\cup\{i\}$ denotes the $(\ell+1)$-dimensional face in $\Delta_{\ell+1}(T)$ with vertices $\{i, f(0),\dots, f(\ell)\}$. Let $\bs n_{f\cup\{i\}}^f$ be a unit normal vector of $f$ but tangential to $f\cup\{i\}$. The basis
\begin{equation*}
\{\bs n_{f\cup\{i\}}^f, i\in f^*\}
\end{equation*}
will be called the {\it tangential normal basis}. 

\begin{lemma}\label{lem:normaldual}
For $f\in \Delta_{\ell}(T)$, the rescaled tangential normal basis $\{\bs n_{f\cup\{i\}}^f/(\bs n_{f\cup\{i\}}^f\cdot\bs n_{F_i}), i\in f^*\}$ of $\mathscr N^f$ is dual to the face normal basis $\{\bs n_{F_i}, i\in f^*\}$.
\end{lemma}
\begin{proof}
Clearly $\bs n_{f\cup\{i\}}^f, \bs n_{F_i}\in \mathscr N^f$ for $i\in f^*$. It suffices to prove
\begin{equation*}
 \bs n_{f\cup\{i\}}^f\cdot\bs n_{F_j}=0 \quad \textrm{ for } i,j\in f^*, i\neq j,
\end{equation*}
which follows from the fact $f\cup\{i\}\subseteq F_j$ and $\bs n_{f\cup\{i\}}^f\in\mathscr T^{f\cup\{i\}}$.
\end{proof}


\begin{figure}[htbp]
\subfigure[Basis $\{\boldsymbol{t}_{0,1}, \boldsymbol{t}_{0,2}, \boldsymbol{t}_{0,3}\}$ and $\{\nabla \lambda_1, \nabla \lambda_2, \nabla \lambda_3\}$.]{
\begin{minipage}[t]{0.5\linewidth}
\centering
\includegraphics[width=5.0cm]{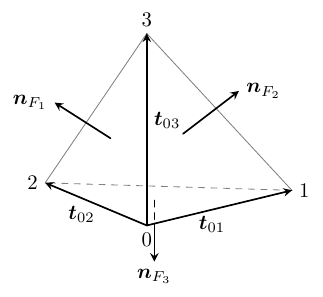}
\end{minipage}}
\subfigure[Basis $\{\bs n_{F_2}, \bs n_{F_3}\}$ and $\{ \bs n^{\{0,1\}}_{\{0,1,2\}}, \bs n^{\{0,1\}}_{\{0,1,3\}}\}.$]
{\begin{minipage}[t]{0.5\linewidth}
\centering
\includegraphics*[width=5.3cm]{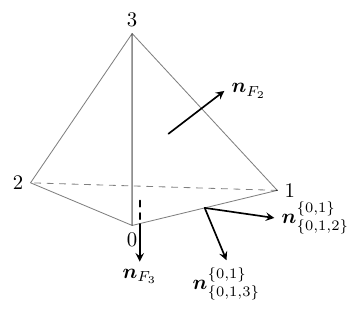}
\end{minipage}}
\caption{Face normal basis and tangential normal basis of a vertex and an edge in a tetrahedron.}
\label{fig:normalbasis}
\end{figure}


\begin{example}\rm
An important example is $f\in \Delta_0(T)$, i.e., $f$ is a vertex. Without loss of generality, let $f = \{0\}$. Then $\bs n_{f\cup\{i\}}^f$ is a unit normal vector of edge $\{0,i\}$: $\bs t_{0i}$ or $\bs t_{i0}$ depending on the orientation. Its dual basis is $\{\bs n_{F_i}/(\bs n_{F_i}\cdot \bs t_{0i}), i=1, \ldots, n\}$. See Fig.~\ref{fig:normalbasis} (a).
\end{example}

\begin{example}\rm
 Let $f = \{0,1\}$ be an edge of a tetrahedron. Then we have two bases of the normal plane $\mathscr N^f$: $\{\bs n_{F_2}, \bs n_{F_3}\}$ and $\{ \bs n^{\{0,1\}}_{\{0,1,2\}}, \bs n^{\{0,1\}}_{\{0,1,3\}}\}.$ They are dual to each other with an appropriate rescaling. See Fig.~\ref{fig:normalbasis} (b). 
\end{example}

\subsection{Barycentric coordinates and Bernstein polynomials}
For a domain $\Omega \subseteq \mathbb{R}^{n}$ and integer $r \geqslant 0$, $\mathbb{P}_{r}(\Omega)$ denotes the space of real valued polynomials defined on $\Omega$ of degree less than or equal to $r$. For simplicity, we let $\mathbb{P}_{r}=\mathbb{P}_{r}\left(\mathbb{R}^{n}\right) .$ Hence, if $n$-dimensional domain $\Omega$ has nonempty interior, then $\operatorname{dim} \mathbb{P}_{r}(\Omega)=\operatorname{dim} \mathbb{P}_{r}=\displaystyle{r+n \choose n}.$
When $\Omega = \{\texttt{v}\}$ is a point, $\mathbb{P}_{r}(\texttt{v})=\mathbb{R} \text { for all } r \geqslant 0$. And we set $\mathbb{P}_{r}(\Omega)= \{0\}$ when $r<0$. 
Let $\mathbb{H}_{r}(\Omega)$ denote the space of real valued polynomials defined on $\Omega$ of degree $r$.

For $n$-dimensional simplex $T$, we denote by $\lambda_{0}, \lambda_{1}, \ldots, \lambda_{n}$ the barycentric coordinate functions with respect to $T$. That is $\lambda_{i} \in \mathbb{P}_{1}(T)$ and $\lambda_{i}\left(\texttt{v}_j\right)=\delta_{i,j}, 0 \leqslant i, j \leqslant n$, where $\delta_{i,j}$ is the Kronecker delta function. The functions $\{\lambda_{i}, i=0,1,\ldots, n\}$ form a basis for $\mathbb{P}_{1}(T)$, $\sum_{i=0}^n\lambda_i (x)= 1$, and $0\leq \lambda_i(x)\leq 1, i=0,1,\ldots,n,$ for $x\in T$. The sub-simplices of $T$ correspond to the zero sets of the barycentric coordinates. Indeed $f = \{x\in T\mid \lambda_i(x) = 0, i\in f^*\}$ for $f\in \Delta_{\ell}(T)$.

We will use the multi-index notation $\alpha \in \mathbb{N}^{n}$, meaning $\alpha=\left(\alpha_{1}, \ldots, \alpha_{n}\right)$ with integer $\alpha_{i} \geqslant 0.$ We define $x^{\alpha}=x_{1}^{\alpha_{1}} \cdots x_{n}^{\alpha_{n}}$, and
$|\alpha|:=\sum_{i=1}^n \alpha_{i} .$ We will also use the set $\mathbb{N}^{0: n}$ of multi-indices $\alpha=\left(\alpha_{0}, \ldots, \alpha_{n}\right)$, with $\lambda^{\alpha}:=\lambda_{0}^{\alpha_{0}} \cdots \lambda_{n}^{\alpha_{n}}$ for $\alpha\in \mathbb{N}^{0: n}$.

We introduce the simplicial lattice~\cite{Chen;Huang:2022FEMcomplex3D,ChenHuang2021Cmgeodecomp}, which is also known as the principal lattice~\cite{nicolaides1973class}. A simplicial lattice of degree $r$ and dimension $n$ is a multi-index set of $n+1$ components and with fixed length $r$, i.e.,
\[
\mathbb T^{n}_r = \left \{ \alpha = (\alpha_0, \alpha_1, \ldots, \alpha_n)\in\mathbb N^{0:n} \mid \alpha_0 + \alpha_1 + \ldots + \alpha_n = r \right \}.
\]
An element $\alpha\in \mathbb T^{n}_r$ is called a node of the lattice. 
The Bernstein representation of polynomial of degree $r$ on a simplex $T$ is
\[
\mathbb P_r(T) :={\rm span}\{ \lambda^{\alpha} = \lambda_{0}^{\alpha_0}\lambda_{1}^{\alpha_1}\ldots \lambda_{n}^{\alpha_n}, \alpha\in \mathbb T_r^{n}\}.
\]
In the Bernstein form, for an $f\in \Delta_{\ell}(T)$, 
\[
\mathbb P_r(f) ={\rm span}\{ \lambda_{f}^{\alpha} = \lambda_{f(0)}^{\alpha_0}\lambda_{f(1)}^{\alpha_1}\ldots \lambda_{f(\ell)}^{\alpha_{\ell}},  \alpha \in \mathbb T_r^{\ell}\}.
\]
Through the natural extension defined by the barycentric coordinate, $\mathbb P_r(f)\subseteq \mathbb P_r(T)$. 
The bubble polynomial of $f$ is a polynomial of degree $\ell + 1$:
\[
b_f: = \lambda_{f} = \lambda_{f(0)}\lambda_{f(1)}\ldots \lambda_{f(\ell)}\in \mathbb P_{\ell + 1}(f).
\]

We have the following property of the bubble polynomial $b_f$.
\begin{lemma}\label{lm:bf}
Let $f,e\in \Delta(T)$. If $f\not\subseteq e$, then $b_f\mid_e = 0$. 
\end{lemma}
\begin{proof}
As $f = (f\cap e^*)\cup (f\cap e)$ and $f\not\subseteq e$, we conclude $f\cap e^* \neq \varnothing$. So $b_f$ contains $\lambda_i$ for some $i\in e^*$ and consequently $b_f|_e = 0$. 
\end{proof}
In particular, $b_f$ vanishes at all sub-simplices other than $f$ with dimensions $\leq \dim f$, and higher dimensional sub-simplices not containing $f$.  

\subsection{Geometric decomposition of Lagrange elements}\label{sec:geodecompLagrange}
We begin with the geometric decomposition of the Lagrange element. 
%
The following geometric decomposition of Lagrange element is given in~\cite{ArnoldFalkWinther2009} without proofs. As it is the foundation of other geometric decompositions in later sections, we present it using our notation and provide a detailed proof. We refer to~\cite[Fig. 2.1]{ArnoldFalkWinther2009} for an illustration of this geometric decomposition. For the $0$-dimensional face, i.e., a vertex $\texttt{v}$, we understand that
$\int_{\texttt{v}} u \dd s = u(\texttt{v})$ for $1\in \mathbb P_r(\texttt{v})=\mathbb R$.

\begin{theorem}[Geometric decomposition of Lagrange element, (2.6) in~\cite{ArnoldFalkWinther2009}]\label{thm:Lagrangedec}
For the polynomial space $\mathbb P_r(T)$ with $r\geq 1$ on an $n$-dimensional simplex $T$, we have the following decomposition 
\begin{align}
\label{eq:Prdec}
\mathbb P_r(T) &= \Oplus_{\ell = 0}^n\Oplus_{f\in \Delta_{\ell}(T)} b_f\mathbb P_{r - (\ell +1)} (f).
\end{align}
And the function $u\in \mathbb P_r(T)$ is uniquely determined by DoFs
\begin{equation}\label{eq:dofPr}
\int_f u \, p \dd s, \quad \quad~p\in \mathbb P_{r - (\ell +1)} (f), f\in \Delta_{\ell}(T), \ell = 0,1,\ldots, n.
\end{equation}
\end{theorem}
\begin{proof}
We first prove the decomposition~\eqref{eq:Prdec}. Each component $b_f\mathbb P_{r - (\ell +1)} (f)\subset \mathbb P_r(T)$ and the sum is direct due to the property of $b_f$, cf. Lemma~\ref{lm:bf}. Then count the dimensions and use the combinatorial identity
$$
\sum_{\ell =0}^n { n+1 \choose \ell + 1} { r - 1 \choose r - \ell - 1} = { n + r \choose r},
$$
which can be proved by looking at the coefficient of $x^r$ in $(1+x)^{n+1} (1+x)^{r-1}= (1+x)^{n+r}$. 

To prove the unisolvence, we choose a basis $\{\phi_i \}$ of $\mathbb P_r(T) $ by the decomposition~\eqref{eq:Prdec} and denote DoFs~\eqref{eq:dofPr} as $\{ N_i\}$. By construction, the dimension of $\{\phi_i \}$ matches the number of DoFs $\{ N_i\}$. The DoF-Basis matrix $(N_i(\phi_j))$ is thus square and block lower triangular in the sense that for $\phi_{f}\in b_f\mathbb P_{r - (\ell +1)} (f),$ 
\[
\int_{e}\phi_{f} p\dd s=  0, \quad \quad e \in \Delta_m(T)\text{ with } m\leq \ell \text{ and } e\neq f, p\in \mathbb P_{r-\dim e - 1}(e)
\]
due to the property of $b_f$ established in Lemma~\ref{lm:bf}. Each diagonal block matrix is the Gram matrix 
$$
\int_f p q b_f \dd x_f, \quad p, q \in \mathbb P_{r-(\ell + 1)}(f),
$$
in the measure $b_f \dd x_f$ and thus symmetric and positive definite. In particular, it is invertible.  
So the unisolvence follows from the invertibility of this lower triangular matrix; see below for an illustration. 
\begin{equation}\label{eq:lowertriangular}
\renewcommand{\arraystretch}{1.35}
\begin{array}{cc}
\begin{array}{c}
 \; N_f \backslash \;   \phi_f 
\end{array}
 &  
\begin{array}{ccccc}
\;0 \qquad\qquad\;\;\;\;& \!\!\!1\;\;\;\;\; \quad\quad& \;\ldots\quad\quad\;\;\;	&  n-1\qquad\quad & \quad\quad\;\;\, n\quad\quad
\end{array}
\medskip
\\
\begin{array}{c}
0 \\ 1 \\ \vdots \\ n-1 \\ n
\end{array} 
& \left(
\begin{array}{>{\hfil$}m{1.5cm}<{$\hfil}|>{\hfil$}m{1.9cm}<{$\hfil}|>{\hfil$}m{0.5cm}<{$\hfil}|>{\hfil$}m{2.5cm}<{$\hfil}|>{\hfil$}m{1.9cm}<{$\hfil}}
\square & 0 & \cdots	& 0 & 0 \\
\hline
\square & \square & \cdots	& 0 & 0 \\
\hline
\vdots & \vdots & \ddots	& \vdots & \vdots \\
\hline
\square & \square & \cdots	& \square & 0 \\
\hline
\square & \square & \cdots	& \square& \square 
\end{array}
\right)
\end{array}
\end{equation}
\end{proof}
\begin{remark}\rm
 It is important to note that $\mathbb P_{r-(\ell +1)}(f) = \{0\}$ when $r< \ell + 1$. As a result, the last non-zero term in \eqref{eq:Prdec} corresponds to $\ell \leq \min\{ r-1 , n\}$. This implies that the degree of the polynomial dictates the dimension of the sub-simplex in decomposition \eqref{eq:Prdec}. For instance, with quadratic polynomials, the summation includes only edge bubbles and excludes face bubbles and higher dimensions. Despite this, the full summation notation $\Oplus_{\ell=0}^n$ is retained for simplicity, with the implicit understanding that the range of non-zero sub-spaces will automatically truncate the limits.
\end{remark}

Let $\{\mathcal {T}_h\}$ be a family of partitions
of $\Omega$ into nonoverlapping simplices with $h_T:=\mbox{diam}(T)$ and $h:=\max_{T\in \mathcal {T}_h}h_T$.
The mesh $\mathcal T_h$ is conforming in the sense that the intersection of any two simplices is either empty or a common lower sub-simplex. Let $\Delta_{\ell}(\mathcal T_h)$ be the set of all $\ell$-dimensional sub-simplices of the partition $\mathcal {T}_h$ for $\ell=0, 1, \ldots, n$. The Lagrange finite element space 
\[
S_h^r := \{ v\in C(\Omega): v\!\!\mid_T\in \mathbb P_r(T), \forall~T\in \mathcal T_h, \text{ DoFs } \eqref{eq:dofPr} \text{ are single-valued}\},
\] 
has the geometric decomposition
\[
S_h^r = \Oplus_{\ell = 0}^n\Oplus_{f\in \Delta_{\ell}(\mathcal T_h)} b_f\mathbb P_{r - (\ell +1)} (f).
\]
Here we extend the polynomial $b_f\mathbb P_{r - (\ell +1)} (f)$ to each element $T$ containing $f$ by the Bernstein form in the barycentric coordinate and thus it is a piecewise polynomial function and continuous in $\Omega$. Consequently $S_h^r\subset H^1(\Omega)$ and the dimension of $S_h^r$ is 
\[
\dim S_h^r = \sum_{\ell = 0}^n |\Delta_{\ell}(\mathcal T_h)| { r - 1 \choose \ell},
\]
where $|\Delta_{\ell}(\mathcal T_h)|$ is the cardinality, i.e., the number of $\ell$-dimensional simplices in $\mathcal T_h$.

The geometric decomposition of the vector Lagrange elements is a straightforward generalization:
\begin{align}
\label{eq:vectorLagrange}
\mathbb P_r(T; \mathbb R^n)  &= \Oplus_{\ell = 0}^n\Oplus_{f\in \Delta_{\ell}(T)} \left [b_f \mathbb P_{r - (\ell +1)} (f) \otimes \mathbb R^n \right ].
\end{align}
In~\eqref{eq:vectorLagrange}, a fixed orthonormal basis of $\mathbb R^n$ is implicitly assumed in which the vector is expanded. It is usually the Cartesian coordinate describing the domain $\Omega$.

\section{Geometric Decompositions of Vector Face Elements}\label{sec:geodecompdivvec}
In this section we consider $H(\div)$-conforming vector finite elements with local shape function space $\mathbb P_r(T;\mathbb R^n)$. Define $H(\div, \Omega):=\{\bs v\in L^2(\Omega;\mathbb R^n): \div \bs v\in L^2(\Omega)\}.$ For a subdomain $K\subseteq \Omega$, the trace operator for the div operator is
$$
{\rm tr}_K^{\div} \boldsymbol v = \boldsymbol v\cdot \boldsymbol n|_{\partial K} \quad \textrm{ for }\;\;\boldsymbol{v}\in C(K),
$$
where $\bs n$ denotes the outwards unit normal vector of $\partial K$. 
Given a triangulation $\mathcal T_h$ and a piecewise smooth function $\bs u$, it is well known that $\bs u\in H(\div, \Omega)$ if and only if $\bs n_F \cdot \bs u$ is continuous across all faces $F\in \Delta_{n-1}(\mathcal T_h)$, which can be ensured by having DoFs on faces. 
An $H(\div)$-conforming finite element is thus also called a face element.
%
%
%
%

\subsection{Examples in three dimensions}\label{sec:divexample}
We first use three-dimensional examples to illustrate the main idea. 
Recall that the geometric decomposition of the vector Lagrange elements in three dimensions reads 
\begin{align}
\label{eq:vectorLagrange3d}
\mathbb P_r(T; \mathbb R^3)  &= \Oplus_{\ell = 0}^3\Oplus_{f\in \Delta_{\ell}(T)} \left [b_f \mathbb P_{r - (\ell +1)} (f) \otimes \mathbb R^3 \right ]. 
\end{align}
A fixed orthonormal basis $\{\bs e_i\}_{i=1}^3$ of $\mathbb R^3$ is used in \eqref{eq:vectorLagrange3d}. See Fig.~\ref{fig:redistribution} (a).

An $H(\div)$ function is a vector proxy of an $(n-1)$-form; see Section \ref{sec:divtensorspace}. As a differential form, it is an intrinsic quantity independent of the choice of coordinates/frames. Based on this observation, we shall choose different frames at different sub-simplex $f\in \Delta_{\ell}(T)$. 

For $f\in \Delta_{\ell}(T)$ with $\ell=0,1,2,3$,
the tangent plane and normal plane of $f$ are 
\begin{equation*}
\mathscr T^f = {\rm span} \{\bs t_i^f, i=1, \ldots, \ell \},
\quad
\mathscr{N}^f =  {\rm span} \{\bs n_i^f, i=1, \ldots, 3 - \ell \}.
\end{equation*}
Then $\mathbb R^3$ admits a tangential-normal ($t$-$n$) decomposition $\mathbb R^3=\mathscr T^f\oplus^{\perp} \mathscr N^f$. Coupled with the bubble polynomials, we obtain a $t$-$n$ decomposition of $\mathbb P_r(T;\mathbb R^3)$ as 
\[
\mathbb P_r(T; \mathbb R^3)  = \Oplus_{\ell = 0}^3\Oplus_{f\in \Delta_{\ell}(T)} \left [\mathbb B_r\mathscr T^f\oplus \mathbb B_r\mathscr N^f \right ],
\]
where $$
\mathbb B_r\mathscr T^f = b_f\mathbb P_{r - (\ell +1)} (f)\otimes \mathscr T^f, \quad 
\mathbb B_r\mathscr N^f = b_f\mathbb P_{r - (\ell +1)} (f)\otimes \mathscr N^f.$$
Notice that for a vertex $\texttt{v}\in \Delta_0(T)$, $\mathbb B_r\mathscr T^{\texttt{v}} = \{ 0\}$, and $\mathbb B_r\mathscr N^{\texttt{v}} = \lambda_{\texttt{v}}\mathbb R^3$ as $b_{\texttt{v}} = \lambda_{\texttt{v}}, \mathbb P_{r - 1} (\texttt{v})= \mathbb R$, and $\mathscr N^{\texttt{v}} = \mathbb R^3$.

Define the polynomial bubble space $\mathbb B_r(\div, T) := \ker({\rm tr}^{\div})\cap \mathbb P_r(T; \mathbb R^3),$
where ${\rm tr}^{\div} \boldsymbol u = \boldsymbol u\cdot \boldsymbol n|_{\partial T}$. 
The tangential component will form the div bubble space: for $r\geq 2$, it holds that
\begin{equation*}
\mathbb B_r(\div, T) =  \Oplus_{\ell = 1}^3\Oplus_{f\in \Delta_{\ell}(T)} \mathbb B_r\mathscr T^f.
\end{equation*}
Verification of
$\mathbb B_r\mathscr T^f \subseteq \mathbb B_{r}(\div,T)$
is straightforward. For face $F$ not containing $f$, $b_f|_F = 0$. For face $F$ containing $f$, $\tr_F^{\div} \bs u = \bs u \cdot \bs n_F  = 0$ as $\bs t_i^f\cdot \bs n_F = 0$. 

\begin{figure}[htbp]
\subfigure[Fixed Cartesian basis $\{\bs e_i\}_{i=1}^3$ is used for the vector Lagrange element.]{
\begin{minipage}[t]{0.305\linewidth}
\centering
\includegraphics*[width=3.35cm]{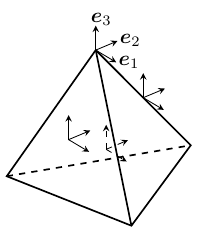}
\end{minipage}}
\quad
\subfigure[In N\'ed\'elec/BDM element, face normal basis $\{\bs n_F, f\subseteq F\}$ is used for $\mathcal N^f$.]
{\begin{minipage}[t]{0.305\linewidth}
\centering
\includegraphics*[width=3.45cm]{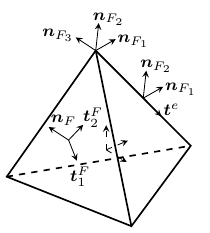}
\end{minipage}}
\quad
\subfigure[Facewise redistribution of the normal DoFs for the quadratic face element.]
{\begin{minipage}[t]{0.305\linewidth}
\centering
\includegraphics*[width=3.35cm]{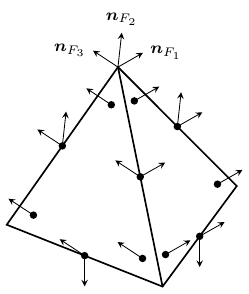}
\end{minipage}}
\caption{Classical face elements can be obtained by a special $t$-$n$ decomposition and face-wise redistribution of normal components.}
\label{fig:redistribution}
\end{figure}

The normal component will contribute to the trace. We can derive the second family of N\'ed\'elec element~\cite{Nedelec1986,BrezziDouglasDuranFortin1987}, which generalizes the Brezzi-Douglas-Marini (BDM) element~\cite{Brezzi.F;Douglas.J;Marini.L1985} in two dimensions,  from a special $t$-$n$ basis. For face $F\in \Delta_{2}(T)$, choose two linearly independent tangential vectors $\bs t_1^F, \bs t_{2}^F$ of $F$ and one normal vector $\bs n_F$ of $F$. For edge $e\in \Delta_{1}(T)$, choose a tangential vector $\bs t^e$ of $e$ and the face normal basis vectors $\{\bs n_{F_1}, \bs n_{F_2}\}$, where $F_1,F_2$ are two faces containing $e$. For vertex $\texttt{v}\in \Delta_{0}(T)$, choose $\{\bs n_{F_1}, \bs n_{F_2}, \bs n_{F_3}\}$, where $F_1, F_2, F_3$ are three faces containing vertex $\texttt{v}$. For tetrahedron $T$, we keep the canonical basis $\{\bs e_i\}_{i=1}^3$, which is considered as tangential vectors of $T$. See Fig.~\ref{fig:redistribution} (b).

We then group normal components $\{\mathbb B_r\mathscr N^f, f\in \Delta(T)\}$ face by face. On each face $F$, again by the geometric decomposition of Lagrange element, the normal components at different sub-simplices $f\in \Delta(F)$ will determine the scalar function $\bs u\cdot \bs n_F$. In Fig.~\ref{fig:redistribution} (c), we show such face-wise redistribution for a quadratic face element. 

\begin{example}[N\'ed\'elec element/BDM element] \rm
Taking $\mathcal N^f :={\rm span}\{\bs n_{F_i}, f\subset F_i\}$, the following DoFs
\begin{subequations}\label{eq:vecbdDoFs}
\begin{align}
\label{eq:vecbdDoFvertex3d}
(\bs v\cdot \bs n_{F_i})|_{F_i}(\texttt{v}),& \quad F_i\in\Delta_2(T), \texttt{v}\in\Delta_{0}(F_i),  \\
\label{eq:vecbdDoFedge3d}
\int_e (\bs v\cdot \bs n_{F_i})|_{F_i}\ p \dd s,& \quad F_i\in\Delta_2(T), e\in\Delta_{1}(F_i), p\in\mathbb P_{r-2}(e), \\
\label{eq:vecbdDoFface3d}
\int_{F_i} (\bs v\cdot \bs n_{F_i})\ p \dd s, &\quad F_i\in\Delta_{2}(T), p\in\mathbb P_{r-3}(F_i), \\
\int_T \bs v \cdot \bs p \dx,&\quad \bs p\in \mathbb B_r(\div, T) \label{eq:bubbleDoF3d}
\end{align}
\end{subequations}
define the N\'ed\'elec element/BDM element. 
DoFs~\eqref{eq:vecbdDoFvertex3d}-\eqref{eq:vecbdDoFface3d} are all located on face $F_i$, and uniquely determine $(\bs v\cdot \bs n_{F_i})\mid_{F_i}\in \mathbb P_r(F_i)$. Hence DoFs~\eqref{eq:vecbdDoFvertex3d}-\eqref{eq:vecbdDoFface3d} can be combined to one DoF
$$
\int_{F_i} (\bs v\cdot \bs n_{F_i})\ p \dd s, \quad  F_i\in\Delta_{2}(T),p\in\mathbb P_{r}(F_i).
$$
The interior DoF~\eqref{eq:bubbleDoF3d} can be further replaced by $\boldsymbol{p}\in \big(\mathbb P_{r-2}(T;\mathbb R^n) \oplus \mathbb H_{r-2}(T; \mathbb K)\boldsymbol x\big)$ with $\mathbb K$ being the skew-symmetric matrix space; see~\cite{Chen;Huang:2021divFinite}.
Therefore DoFs~\eqref{eq:vecbdDoFs} induce the N\'ed\'elec/BDM element; see Fig.~\ref{fig:divelement} (a).  
\end{example}

Different $H(\div)$-conforming finite elements can be obtained by different $t$-$n$ basis. %

\begin{example}[Stenberg element] \rm 
Taking $\{\bs n_i^{\texttt{v}}\}_{i=1}^3 = \{\boldsymbol{e}_i\}_{i=1}^3$ and $\{\bs n_i^e\}_{i=1}^2 = \{\bs n_{F_i}\}_{i=1}^2$ as two face $F_i$ sharing $e$, the following DoFs
\begin{align*}
\bs v(\texttt{v}),& \quad \texttt{v}\in\Delta_{0}(T),  \\
\int_e (\bs v\cdot \bs n_{F_i})|_{F_i} \ p \dd s,& \quad F_i\in\Delta_2(T), e\in\Delta_{1}(F_i), p\in\mathbb P_{r-2}(e), \\
\int_F \bs v\cdot \bs n_F\ p \dd s, &\quad F\in\Delta_{2}(T), p\in\mathbb P_{r-3}(F), \\
\int_T \bs v \cdot \bs p \dx,&\quad \bs p\in \mathbb B_r(\div, T) 
\end{align*}
induce the Stenberg element~\cite{stenbergNonstandardMixedFinite2010}, which is continuous at vertices; see Fig.~\ref{fig:divelement} (b).  
\end{example}

\begin{figure}[htbp]
\subfigure[N\'ed\'elec/BDM element.]
{\begin{minipage}[t]{0.305\linewidth}
\centering
\includegraphics*[width=3.25cm]{figures/3DBDM.pdf}
\end{minipage}}
\subfigure[Stenberg element.]{
\begin{minipage}[t]{0.305\linewidth}
\centering
\includegraphics*[width=3.45cm]{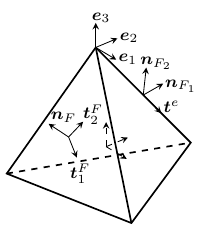}
\end{minipage}}
\subfigure[CHH element.]
{\begin{minipage}[t]{0.305\linewidth}
\centering
\includegraphics*[width=3.275cm]{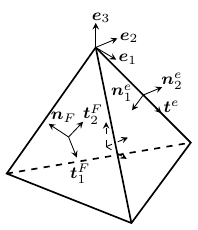}
\end{minipage}}
\caption{Different $H(\div)$-conforming finite elements can be obtained by different $t$-$n$ decompositions.
}
\label{fig:divelement}
\end{figure}

\begin{example}[Christiansen-Hu-Hu element] \rm
Taking $\{\bs n_i^{\texttt{v}}\}_{i=1}^3 = \{\boldsymbol{e}_i\}_{i=1}^3$ and $\{\bs n_i^e\}_{i=1}^2$ depending only on $e$, the following DoFs
\begin{align*}
\bs v(\texttt{v}),& \quad \texttt{v}\in\Delta_{0}(T),  \\
\int_e (\bs v\cdot \bs n_i^e)\ p \dd s,& \quad e\in\Delta_{1}(T), p\in\mathbb P_{r-2}(e),  i=1,2, \\
\int_F (\bs v\cdot \bs n_F)\ p \dd s, &\quad F\in\Delta_{2}(T), p\in\mathbb P_{r-3}(F), \\
\int_T \bs v \cdot \bs p \dx,&\quad \bs p\in \mathbb B_r(\div, T) 
\end{align*}
define the Christiansen-Hu-Hu (CHH) face element~\cite[Section 3.5]{christiansenNodalFiniteElement2018}, which has extra continuity at vertices and on the normal planes of edges; see Fig.~\ref{fig:divelement} (c).   
\end{example}

\subsection{Face elements in arbitrary dimension}
We now present and prove the result in arbitrary $n\geq 2$ dimensions. 
Again the main idea is using different and in general non-orthonormal $t$-$n$ decomposition adapted to different sub-simplices. 

\subsubsection{Bubble spaces}
For an $\ell$-dimensional sub-simplex $f\in \Delta_{\ell}(T)$, choose 
$\ell$ linearly independent tangential vectors $\{\bs t_1^f, \ldots, \bs t_{\ell}^f\}$ of $f$ and $n - \ell$ linearly independent normal vectors $\{\bs n_1^f, \ldots, \bs n_{n-\ell}^f\}$ of $f$.  
The set of $n$ vectors $\{\bs t_1^f, \ldots, \bs t_{\ell}^f, \bs n_1^f, \ldots, \bs n_{n-\ell}^f \}$ forms a basis of $\mathbb R^n$.
Notice that for $\ell = 0$, i.e., at vertices, there are no tangential vectors and for $\ell = n$, there are no normal vectors. Define the tangent plane and normal plane of $f$ as 
\begin{equation*}
\mathscr T^f := {\rm span} \{\bs t_i^f, i=1, \ldots, \ell \},
\quad
\mathscr{N}^f :=  {\rm span} \{\bs n_i^f, i=1, \ldots, n - \ell \}.
\end{equation*}
All vectors are normalized but $\{\bs t_i^f\}$ or $\{ \bs n_i^f\}$ may not be an orthonormal basis. 
Coupled with the bubble polynomial, for $r\geq 1$, define 
\[
\mathbb B_r\mathscr T^f = b_f\mathbb P_{r - (\ell +1)} (f)\otimes \mathscr T^f, \quad \mathbb B_r\mathscr N^f = b_f\mathbb P_{r - (\ell +1)} (f)\otimes \mathscr N^f.
\]

\begin{lemma} \label{lm:vecLdec}
The shape function space $ \mathbb P_r(T; \mathbb R^n)$ has a geometric decomposition
\begin{equation}\label{eq:Prvecdec}
 \mathbb P_r(T; \mathbb R^n)  = \Oplus_{\ell = 0}^n\Oplus_{f\in \Delta_{\ell}(T)} \left [ \mathbb B_r\mathscr T^f  \oplus \mathbb B_r\mathscr N^f \right ].
\end{equation}
A function $\bs u \in \mathbb P_r(T; \mathbb R^n)$ is uniquely determined by the DoFs: for all $f\in \Delta_{\ell}(T), \ell = 0,1,\ldots, n$,
\begin{subequations}\label{eq:dofvec}
\begin{align}
\label{eq:dofvectort}
\int_f (\bs u\cdot \bs t_i^f)\ p \dd s,& \quad i=1,\ldots, \ell, p\in \mathbb P_{r - (\ell +1)} (f), \\
\label{eq:dofvectorn}
\int_f (\bs u\cdot \bs n_j^f)\ p \dd s, &\quad  j=1,\ldots, n - \ell, p\in \mathbb P_{r - (\ell +1)} (f).
\end{align}
\end{subequations}
\end{lemma}
\begin{proof}
Since $\{\bs t_1^f,\ldots, \bs t_{\ell}^f, \bs n_1^f,\ldots, \bs n_{n-\ell}^f\}$ forms a basis of $\mathbb R^n$, DoFs~\eqref{eq:dofvectort}-\eqref{eq:dofvectorn} are equivalent to 
\begin{equation*}
\int_f \bs u\cdot \bs p \dd s, \quad \quad~\bs p\in \mathbb P_{r - (\ell +1)} (f;\mathbb R^n).
\end{equation*}
Then the unisolvence follows from Theorem~\ref{thm:Lagrangedec} for the Lagrange element.
\end{proof}

Next we use $\mathbb B_r\mathscr T^f$ or $\mathbb B_r\mathscr N^f$ to characterize the kernel or range of the trace operator, respectively. Define the polynomial bubble space 
\[\mathbb B_r(\div, T) := \ker({\rm tr}^{\div})\cap \mathbb P_r(T; \mathbb R^n).\]
It is obvious that $\mathbb B_0(\div, T) = \{\boldsymbol0\}$ and $\mathbb B_1(\div, T) = \{\boldsymbol0\}$.

\begin{lemma}\label{lem:divbubbletracespacedecomp}
For $r\geq 2$, it holds that
\begin{equation}\label{eq:divbubbledecomp}
\mathbb B_r(\div, T) =  \Oplus_{\ell = 1}^n\Oplus_{f\in \Delta_{\ell}(T)} \mathbb B_r\mathscr T^f,
\end{equation}
and 
\begin{equation}\label{eq:trNrf}
{\rm tr}^{\div} : \Oplus_{\ell = 0}^{n-1}\Oplus_{f\in \Delta_{\ell}(T)}  \mathbb B_r\mathscr N^f \to  {\rm tr}^{\div} \, \mathbb P_r(T; \mathbb R^n)
\end{equation}
is a bijection. Consequently
\begin{equation*}
\dim \mathbb B_r(\div, T)  = \sum_{\ell=1}^n{n+1 \choose \ell +1} {r-1 \choose \ell} {\ell \choose 1}.
\end{equation*}
\end{lemma}
\begin{proof}
Verification of
\begin{equation*}
\Oplus_{\ell = 1}^n\Oplus_{f\in \Delta_{\ell}(T)} \mathbb B_r\mathscr T^f  \subseteq \mathbb B_{r}(\div,T)
\end{equation*}
is straightforward. For face $F$ not containing $f$, $b_f|_F = 0$. For face $F$ containing $f$, $\tr_F^{\div} \bs u = \bs u \cdot \bs n_F  = 0$ as $\bs t_i^f \cdot \bs n_F = 0$ for $\bs u\in \mathbb B_r\mathscr T^f$. 
 
Then apply the trace operator to the decomposition~\eqref{eq:Prvecdec} and use $\tr^{\div}(\mathbb B_r\mathscr T^f) = 0$ to obtain $${\rm tr}^{\div} \left ( \Oplus_{\ell = 0}^{n-1}\Oplus_{f\in \Delta_{\ell}(T)}  \mathbb B_r\mathscr N^f\right ) =  {\rm tr}^{\div} \, \mathbb P_r(T; \mathbb R^n).$$ So the map $\tr^{\div}$ in~\eqref{eq:trNrf} is onto. 

We then prove it is also injective. For an $f\in \Delta_{\ell}(T)$, recall that $\{ \bs n_{F_i}, i\in f^*\}$ is the face normal basis of $\mathscr N^f$ and $\{ \hat{\bs n}_{F_i}, i\in f^*\}$ is its dual basis of $\mathscr N^f$. We expand $\bs u\in\Oplus_{\ell = 0}^{n-1}\Oplus_{f\in \Delta_{\ell}(T)}  \mathbb B_r\mathscr N^f$ in the dual basis, i.e., $\bs u=\sum\limits_{\ell = 0}^{n-1}\sum\limits_{f\in \Delta_{\ell}(T)} \sum_{i\in f^*}b_fp_f^i \hat{\bs n}_{F_i} $ with $p_f^i \in \mathbb P_{r-(\ell +1)}(f)$. We will prove if  $\tr^{\div}\bs u = 0$, then $\bs u = 0$.

To do so, we consider the operator $N_i^f(\bs u) = (\bs u\cdot \bs n_{F_i})|_f$. Condition $\tr^{\div}\bs u = 0$ implies $N_i^f(\bs u) = 0$ for all $f\in \Delta(T)$ and $i\in f^*$. By the choice of the basis of $\mathscr N^f$ and the property of bubble polynomial $b_f$, the corresponding $N$-$\phi$ matrix is block lower  triangular with diagonal matrices in the diagonal block; see \eqref{eq:lowertriangular} for an illustration. Therefore $N_i^f(\bs u) = 0$ implies $\bs u = 0$. More specifically, we have $N_i^{\texttt{v}}(\bs u) = p^i_{\texttt{v}} = 0$ for $\texttt{v}\in \Delta_0(T)$ and $i\in\texttt{v}^*$. After that, we apply $N_i^{f}(\bs u)$ to conclude $p^f_i = 0$ for $f\in \Delta_{1}(T)$. Using this forward substitution argument for the lower triangular matrix, we conclude all coefficient polynomials $p^f_i = 0$. 

Once we have proved the map $\tr$ in~\eqref{eq:trNrf} is bijective, we conclude~\eqref{eq:divbubbledecomp} from the decomposition~\eqref{eq:Prvecdec}. 
\end{proof}

With this characterization, tangential DoFs in \eqref{eq:dofvectort} can be merged as one
$\mathbb B_r(\div, T)$.

\subsubsection{Brezzi-Douglas-Marini element}
Given an $f\in \Delta_{\ell}(T)$, we choose $\{\bs n_F, f\subseteq F\in \Delta_{n-1} (T)\}$ as the basis for its normal plane $\mathscr N^f$. 

\begin{lemma}[Local BDM element]\label{lm:localBDM}
The shape function space $ \mathbb P_r(T; \mathbb R^n)$ is uniquely determined by the DoFs
\begin{subequations}\label{eq:BDMDoFs}
 \begin{align}
\label{eq:vecbdDoF}
\int_F \bs v\cdot \bs n_F p \dd s, &\quad p\in \mathbb P_r(F), F\in \Delta_{n-1}(T),\\
\int_T \bs v \cdot \bs p \dx, &\quad \bs p\in \mathbb B_r(\div, T). \label{eq:bubbleDoF}
\end{align}
\end{subequations}
\end{lemma}
\begin{proof}

By the geometric decomposition of $\mathbb P_r(F)$ element, \eqref{eq:vecbdDoF} can be decomposed into
\begin{align}
\label{eq:vecbdDoF2}
\int_f (\bs v\cdot \bs n_F)|_F\ p \dd s, &\quad  F\in \Delta_{n-1}(T), f\in \Delta_{\ell}(F), \\
& \quad p\in \mathbb P_{r - (\ell +1)} (f), \ell = 0,1,\ldots, n-1 \notag. 
\end{align} 
We switch the ordering of $f$ and $F$ to
 \begin{align*}
\int_f (\bs v\cdot \bs n_F)|_F\ p \dd s, &\quad  f\in \Delta_{\ell}(T), F\in \Delta_{n-1}(T), f\subseteq F, \\
& \quad p\in \mathbb P_{r - (\ell +1)} (f), \ell = 0,1,\ldots, n-1, 
\end{align*}
which is exactly \eqref{eq:dofvectorn} according to our choice of normal basis. 

Therefore given a $\bs v\in  \mathbb P_r(T; \mathbb R^n)$, if~\eqref{eq:vecbdDoF} vanishes, then $\tr \bs v = 0$ and consequently $\bs v\in \mathbb B_r(\div, T)$. Finally the vanishing DoF~\eqref{eq:bubbleDoF} implies $\bs v = 0$. 
\end{proof}

One benefit of using the decomposed version \eqref{eq:vecbdDoF2} instead of the merged one \eqref{eq:vecbdDoF} is  that the well documented Lagrange basis functions can be used in the implementation of the face element; see~\cite{ChenChenHuangWei2023}.

We call the change from~\eqref{eq:vecbdDoF2} to \eqref{eq:vecbdDoF} the facewise redistribution of normal DoFs. Namely by using the face normal basis, we can redistribute the DoFs on $f$ to each $(n-1)$-dimensional face $F$. 

To glue local finite elements to form an $H(\div)$-conforming finite element, we need to enforce continuity of $\bs v\cdot \bs n_F$ by choosing $\bs n_F$ depending only on $F$ not element $T$.

\begin{lemma}[BDM space]\label{lm:BDM}
For each $F\in \Delta_{n-1}(\mathcal T_h)$, choose a normal vector $\bs n_F$. 
For the shape function space $\mathbb P_r(T; \mathbb R^n)$,
the following {\rm DoFs}
\begin{subequations}\label{eq:BDMDoFTh}
\begin{align}
\label{eq:faceDofTh} 
\int_F \bs v\cdot \bs n_F p \dd s, &\quad p\in \mathbb P_r(F), F\in \Delta_{n-1}(\mathcal T_h),\\
\label{eq:bubbleDofTh} 
\int_T \bs v \cdot \bs p \dx,&\quad \bs p\in \mathbb B_r(\div, T), T\in \mathcal T_h, 
\end{align}
\end{subequations}
define an $H(\div)$-conforming space 
$$V_{\small \rm BDM}=\{\bs v_h\in H(\div, \Omega): \bs v_h|_T\in\mathbb P_r(T; \mathbb R^n), \; \forall~T\in\mathcal T_h\}.$$
\end{lemma}
\begin{proof}
On each element $T$, DoFs~\eqref{eq:BDMDoFTh} will determine a function in $\mathbb P_r(T; \mathbb R^n)$ by Lemma~\ref{lm:localBDM}. DoF~\eqref{eq:faceDofTh} will determine the trace $\bs v\cdot \bs n_F$ on $F$ independent of the element containing $F$, and thus the function is $H(\div,\Omega)$-conforming. 
\end{proof}

We have the geometric decomposition of the global BDM element space
\begin{align*}
V_{\small \rm BDM} = &\Oplus_{F\in \Delta_{n-1}(\mathcal T_h)}\Oplus_{\ell = 0}^{n-1}\Oplus_{f\in \Delta_{\ell}(F)} \mathbb B_r\mathscr N^f(F, \Omega) \oplus \Oplus_{T\in \mathcal T_h} \mathbb B_r(\div, T),
\end{align*}
and
$$
\dim V_{\small \rm BDM} = |\Delta_{n-1}(\mathcal T_h)| { n + r -1 \choose r} + |\Delta_n(\mathcal T_h)| \sum_{\ell=1}^n{n+1 \choose \ell +1} \ell {r-1 \choose \ell}.
$$
Here
\begin{align*}
\mathbb B_r\mathscr N^f(F, \Omega):=&\{\bs v_h\in H(\div, \Omega): \bs v_h|_T\in\mathbb B_r\mathscr N^f\; \textrm{ for } T\in\mathcal T_h, F\subseteq T, \\
&\qquad\qquad\qquad\qquad\qquad\, \bs v_h|_{T'}=0 \;\textrm{ for } T'\in\mathcal T_h, F\not\subseteq T'\}.
\end{align*}

\subsubsection{Stenberg's type element}\label{sec:Stenberg}
We can construct an $H(\div)$-conforming element with more continuity on the normal plane of lower dimensional sub-simplices by choosing a global basis of $\mathscr N^f$. 

In the $t$-$n$ decomposition, if a basis vector $\bs t_i^f$ or $\bs n_i^f$ depends only on $f$ not on element $T$ containing $f$, we call it {\it global} and otherwise $\bs t_i^f(T)$ or $\bs n_i^f(T)$ is local and the corresponding DoFs are different for different $T$ containing $f$. For a global basis vector, the corresponding DoF~\eqref{eq:dofvectort} or~\eqref{eq:dofvectorn} only depends on $f$ and thus imposes continuity in that direction. In the extreme case, if all $t$-$n$ basis vectors are global, we obtain the Lagrange element. 

The following is a generalization of Stenberg element by imposing more continuity on the normal plane for sub-simplices with dimension from $0$ to $m$ for some $-1\leq m \leq n -2$. When $m=0$, it is the original Stenberg's element~\cite{stenbergNonstandardMixedFinite2010}, i.e., only continuous at vertices. When $m = n-2$, it is Christiansen-Hu-Hu element constructed in~~\cite{christiansenNodalFiniteElement2018}. When $m = -1$, no DoF~\eqref{eq:divfemThDoF1} exists and thus reduces to the BDM element. 

\begin{lemma}[Stenberg type element]\label{lm:Stenberg}
Let $-1\leq m \leq n -2$. For each $f\in \Delta_{\ell}(\mathcal T_h)$ with $\ell \leq m$, we choose $n-\ell$ normal vectors $\{\bs n^f_{1}, \ldots, \bs n^f_{n-\ell}\}$. Then the {\rm DoFs}
\begin{subequations}\label{eq:localglobalnormalbasis}
\begin{align}
\label{eq:divfemThDoF1}
\int_f \bs v\cdot \bs n^f_i\, p \dd s, &\quad  p\in \mathbb P_{r-(\ell +1)}(f), \;  f\in \Delta_{\ell}(\mathcal T_h),\\
&\quad i=1,\ldots, n-\ell, \; \ell = 0,\ldots, m, \notag\\
\label{eq:divfemThDoF2}
\int_f (\bs v\cdot \bs n_F)|_F\ p \dd s, &\quad  p\in \mathbb P_{r-(\ell +1)}(f),  F\in \Delta_{n-1}(\mathcal T_h), \; f\in \Delta_{\ell}(F),\\
&\quad  \; \ell = m+1,\ldots, n-1, \notag\\
\label{eq:divfemThDoF3}
\int_T \bs v\cdot \bs p \dx, &\quad \bs p\in \mathbb B_r(\div, T),   \;  T\in \mathcal T_h,
\end{align}
\end{subequations}
will determine a space $V_{\small \rm Stenberg}^r\subset H(\div,\Omega)$.
\end{lemma}
\begin{proof}
For $T\in\mathcal T_h$ and $f\in\Delta_{\ell}(T)$,
both $\{\bs n^f_1, \ldots, \bs n^f_{n-\ell}\}$ and $\{\bs n_F, F\in \Delta_{n-1}(\mathcal T), f\subset F\}$ are bases of the normal plane $\mathscr N^f$.
Then the number of DoFs~\eqref{eq:localglobalnormalbasis} restricted to $T$ equals to the number of DoFs~\eqref{eq:BDMDoFs}. DoF~\eqref{eq:divfemThDoF1} determines DoF~\eqref{eq:faceDofTh} for $\ell = 0,\ldots, m$. Thus we conclude the result from Lemma~\ref{lm:BDM}.
\end{proof}

We have the geometric decomposition of the global Stenberg element space
\begin{align*}
V_{\small \rm Stenberg}^r &= \Oplus_{\ell = 0}^{m}\Oplus_{f\in \Delta_{\ell}(\mathcal T_h)} \mathbb B_r\mathscr N^f(\Omega) \oplus \Oplus_{F\in \Delta_{n-1}(\mathcal T_h)}\Oplus_{\ell = m+1}^{n-1}\Oplus_{f\in \Delta_{\ell}(F)} \mathbb B_r\mathscr N^f(F, \Omega) \\
&\quad \oplus  \Oplus_{T\in \mathcal T_h} \mathbb B_r(\div, T),
\end{align*}
and
\begin{align*}
\dim V_{\small \rm Stenberg}^r& = \sum_{\ell = 0}^{m}|\Delta_{\ell}(\mathcal T_h)| (n - \ell ){ r-1 \choose \ell} + |\Delta_{n-1}(\mathcal T_h)| \sum_{\ell = m+1}^{n-1}{n \choose \ell +1} {r-1 \choose \ell} \\
&\quad + |\Delta_n(\mathcal T_h)|\sum_{\ell=1}^n {n+1 \choose \ell +1}  {r-1 \choose \ell} {\ell \choose 1}.
\end{align*}
Here
\begin{align*}
\mathbb B_r\mathscr N^f(\Omega):=&\{\bs v_h\in H(\div, \Omega): \bs v_h|_T\in\mathbb B_r\mathscr N^f\; \textrm{ for } T\in\mathcal T_h, f\subseteq T, \\
&\qquad\qquad\qquad\qquad\qquad\, \bs v_h|_{T'}=0 \;\textrm{ for } T'\in\mathcal T_h, f\not\subseteq T'\}.
\end{align*}
Clearly $V_{\small \rm Stenberg}^r\subseteq V_{\small \rm BDM}$, and $\dim V_{\small \rm Stenberg}^r<\dim V_{\small \rm BDM}$ for $0\leq m\leq n-2$.

We introduce an \( n \)-dimensional smoothness vector \(\boldsymbol r= (r_0, r_1, \cdots, r_{n-1})^{\intercal}\in\mathbb{R}^n\) to characterize the smoothness of the finite element space at sub-simplices of dimension \(\ell = 0, 1, \ldots, n-1\). For the space \(V_{\small \rm Stenberg}^r\) defined by DoFs~\eqref{eq:localglobalnormalbasis}, the smoothness vector is given by 
\begin{equation*}
\boldsymbol r= \big(0, -\frac{1}{n-1}, \ldots, -\frac{m}{n-1}, -1, \ldots, -1\big)^{\intercal}.
\end{equation*} 
For an \(\ell\)-dimensional face \(f\in\Delta_{\ell}(\mathcal{T}_h)\), a smoothness parameter of \(-\frac{\ell}{n-1}\) means the vector is discontinuous only on the \(\ell\)-dimensional tangential space and continuous on the normal plane \(\mathcal{N}^f\). The value \(-1\) indicates that the DoF is redistributed to faces, and the vector is discontinuous on the tangential space of \((n-1)\)-dimensional faces. The same smoothness vector will determine the same global finite element space, although DoFs may be slightly different; see the modification of DoFs in Corollary \ref{cor:modStenberg}.

Here are the smoothness vectors for the examples in Section~\ref{sec:divexample}: N\'ed\'elec element/BDM element: $\boldsymbol r= \bs{-1}$ with $m=-1$, Stenberg element: $\boldsymbol r= (0,-1,-1)^{\intercal}$ with $m=0$, and CHH element: $\boldsymbol r= (0,-1/2,-1)^{\intercal}$ with $m=1$. These vectors give the smoothness properties of the respective finite element spaces at vertices, edges, and faces.

\subsection{Discrete inf-sup condition}
In the unisolvence of vector Lagrange elements, cf. the proof of Lemma~\ref{lm:vecLdec}, any basis of $\mathbb R^n$ at $f$ is allowed. The $t$-$n$ basis is used for two purposes: 1. the $H(\div)$-conformity; 2.  the discrete div stability. 

If the $H(\div)$-conformity is the only concern, we can simply choose the Lagrange element. Another and more important consideration is from the div stability.
At the continuous level, we have the div stability: $\div: H(\div, \Omega)\to L^2(\Omega)$ is surjective, which has a continuous right inverse. A regular right inverse in $H^1(\Omega)$ also exists~\cite{CostabelMcIntosh2010}.  

By the Euler's formula for homogenous degree polynomials $\mathbb H_{r-1}(T)$, i.e.  $\div(\boldsymbol x q) = (r-1+n)q$ for any $q\in\mathbb H_{r-1}(T)$, clearly we have $\div\mathbb P_r(T;\mathbb R^n)  =  \mathbb P_{r-1}(T) $.
Hence the discrete div stability in one element always holds.
We discuss the global version.
Let 
\begin{align}
\notag V_h:=&\{\bs v_h\in H(\div, \Omega): \bs v_h|_T\in\mathbb P_r(T; \mathbb R^n) \textrm{ for each } T\in\mathcal T_h\}, \\
\label{eq:Qh} Q_h:=&\{q_h\in L^2(\Omega): q_h|_T\in\mathbb P_{r-1}(T) \textrm{ for each } T\in\mathcal T_h\}.
\end{align}
The discrete div stability refers to $\div V_h = Q_h$ and the operator $\div$ has a continuous right inverse. 

We will use the $L^2$-inner product $(\cdot, \cdot)_T$ and define the orthogonal complement of a subspace $V\subset L^2(T)$ as $V^{\perp}$, i.e.,
$$
V^{\perp} = \{u \in L^2(T) : (u,v)_T = 0 \quad \forall~v\in V\}. 
$$
We first give the following characterization of the range of the div operator on the bubble polynomial space.
\begin{lemma}[Lemma 3.6 in~\cite{Chen;Huang:2021divFinite}]\label{lem:divBDMbubbleonto}
It holds
$$
\div\mathbb B_r(\div, T)=\mathbb P_{r-1}(T)\cap \mathbb R^{\perp},\quad r\geq 0.
$$
\end{lemma}
\begin{proof}
When $r = 0,1$, both sides are zero. Therefore we focus on $r\geq 2$. 

The inclusion $ \div( \mathbb B_r(\div, T)) \subseteq (\mathbb P_{r-1}(T)\cap \mathbb R^{\perp})$ is proved through integration by parts
$$
(\div \boldsymbol v, p)_T = -( \boldsymbol v, \grad p)_T = 0, \quad p\in \ker(\grad) = \mathbb R.
$$

If $ \div( \mathbb B_r(\div, T)) \neq \mathbb P_{r-1}(T)\cap \mathbb R^{\perp}$, there exists a $p\in  \mathbb P_{r-1}(T)\cap \mathbb R^{\perp}$ and $p\perp \div( \mathbb B_r(\div, T))$, which is equivalent to $\grad p \perp \mathbb B_r(\div, T)$. Expand the vector $\grad p$ in the basis $\{\boldsymbol n_i, i=1,\ldots, n\}$ as $\grad p = \sum\limits_{i=1}^n q_i {\boldsymbol n}_{i}$ with $q_i\in\mathbb P_{r-2}(T)$. Then set $\boldsymbol v_p = \sum\limits_{i=1}^n|\nabla\lambda_i| q_i \lambda_{0}\lambda_i \boldsymbol t_{i,0}\in \mathbb B_r(\div, T)$, where $\boldsymbol t_{i,0}:=\texttt{v}_0-\texttt{v}_i$.
We have
$$
(\grad p, \boldsymbol v_p)_T = \sum_{i=1}^n \int_Tq_i^2 \lambda_{0}\lambda_i \dx = 0,
$$
which implies $q_i = 0$ for $i=1,\ldots, n$, i.e., $\grad p = 0$ and $p = 0$ as $p\in  \mathbb P_{r-1}(T)\cap \mathbb R^{\perp}$.
\end{proof}

Next we verify the discrete div stability.
\begin{proposition}\label{lem:BDMdiscretedivinfsup}
Let $r\geq 1$ and $V_{\small \rm BDM}$ be the BDM space defined in Lemma~\ref{lm:BDM} and $Q_h$ defined by~\eqref{eq:Qh}. It holds the discrete inf-sup condition
\begin{equation}\label{eq:BDMdiscretedivinfsup}
\|q_h\|_0\lesssim \sup_{\bs v_h\in V_{\small \rm BDM}}\frac{(\div \bs v_h, q_h)}{\|\bs v_h\|_0+\|\div \bs v_h\|_0}\quad\forall~q_h\in Q_h.
\end{equation}
\end{proposition}
\begin{proof}
By Theorem 1.1 in~\cite{CostabelMcIntosh2010},
there exists $\bs v\in H^1(\Omega;\mathbb R^n)$ such that
\begin{equation}\label{eq:20210719-1}
\div \bs v=q_h,\quad \|\bs v\|_1\lesssim \|q_h\|_0.    
\end{equation}
Let $\tilde{\bs v}_h\in V_{\small \rm BDM}$ satisfy 
\begin{align*}
&\int_F \tilde{\bs v}_h\cdot \bs n_{F}\dd s=\int_F\bs v\cdot \bs n_{F}\dd s, \quad \quad F\in \Delta_{n-1}(\mathcal T_h),
\end{align*}
and the rest DoFs vanish. 
By the scaling argument,
\begin{equation}\label{eq:20210719-2}
\|\tilde{\bs v}_h\|_0+\|\div\tilde{\bs v}_h\|_0\lesssim \|\bs v\|_1\lesssim \|q_h\|_0.
\end{equation}
Clearly we have 
$\div(\tilde{\bs v}_h-\bs v)|_T\in\mathbb P_{r-1}(T)\cap \mathbb R^{\perp}$ for each $T\in\mathcal T_h$. By Lemma~\ref{lem:divBDMbubbleonto}, there exists $\bs b_h\in L^2(\Omega;\mathbb R^n)$ such that $\bs b_h|_T\in\mathbb B_r(\div, T)$ for each $T\in\mathcal T_h$, and
\begin{equation}\label{eq:20210719-3}
\div \bs b_h=\div(\bs v-\tilde{\bs v}_h), \quad \|\bs b_h\|_{0,T}\lesssim h_T\|\div(\tilde{\bs v}_h-\bs v)\|_{0,T}.
\end{equation}
Take $\bs v_h=\tilde{\bs v}_h+\bs b_h\in V_{\small \rm BDM}$. By~\eqref{eq:20210719-1} and~\eqref{eq:20210719-3}, it holds
\begin{equation}\label{eq:20210719-4}
\div \bs v_h=\div\tilde{\bs v}_h+\div \bs b_h=\div \bs v=q_h. 
\end{equation}
It follows from~\eqref{eq:20210719-2} and~\eqref{eq:20210719-3} that
\begin{align}
\|\bs v_h\|_0+\|\div \bs v_h\|_0&=\|\bs v_h\|_0+\|q_h\|_0\leq \|\tilde{\bs v}_h\|_0+\|\bs b_h\|_0+\|q_h\|_0 \notag\\
&\lesssim \|\tilde{\bs v}_h\|_0+h\|\div\tilde{\bs v}_h\|_{0}+\|q_h\|_0\lesssim\|q_h\|_0. \label{eq:20210719-5}
\end{align}
Combining~\eqref{eq:20210719-4}-\eqref{eq:20210719-5} yields~\eqref{eq:BDMdiscretedivinfsup}.
\end{proof}

For the Stenberg-type element, the continuity at normal planes introduces some constraint and makes the discrete inf-sup condition harder. As all bubble functions are treated locally, $\div\mathbb B_r(\div, T)=\mathbb P_{r-1}(T)\cap \mathbb R^{\perp}$ still holds. We only need to show $\cup_{T\in \mathcal T_h}\mathbb P_0(T)$ is in the range of $\div V_h$, which requires the face average $\int_F \bs v\cdot \bs n_F\dd s$. 


We propose the following modification to have a better discrete inf-sup condition. 
Denote by
$$
\mathbb B_{r, m+1}(F) := \big [\oplus_{\ell = m+1}^{n-1}\oplus_{f\in \Delta_{\ell}(F)}b_f\mathbb P_{r-(\ell +1)}(f)\big ], \quad -1\leq m\leq n-2.
$$

\begin{corollary}\label{cor:modStenberg}
With the same setting in Lemma \ref{lm:Stenberg} and further assume $r\geq m+2$ for $-1\leq m\leq n-2$. Replacing DoF \eqref{eq:divfemThDoF2} by
\begin{equation}\label{eq:divfemThDoF2new}
\int_F \bs v\cdot \bs n_F\ q\dd s, \quad F\in \Delta_{n-1}(\mathcal T_h),
\, q\in (\mathbb B_{r, m+1}(F)\cap \mathbb P_{0}^{\perp}(F))\oplus \mathbb P_{0}(F),
\end{equation}
will define the same finite element space.
\end{corollary}
\begin{proof}
When $-1\leq m\leq n - 2$ and $r\geq m+2$, $\dim \mathbb B_{r, m+1}(F)\geq |\Delta_{m+1}(F)| \geq \dim \mathbb P_0(F)$. Therefore the number of DoFs remains the same as DoF \eqref{eq:divfemThDoF2}.

Vanishing DoFs \eqref{eq:divfemThDoF1} will imply $\boldsymbol{v}\cdot \bs n_F|_{F}\in \mathbb B_{r, m+1}(F)$, which can be determined by \eqref{eq:divfemThDoF2new}. So the unisolvence follows.
 
As the change will not affect the continuity, it will define the same finite element space.  
\end{proof}
%
Now $\int_F \bs v\cdot \bs n_F\dd s$ is in DoF \eqref{eq:divfemThDoF2new}. The proof of the following result is identical to that of Proposition~\ref{lem:BDMdiscretedivinfsup}. 

\begin{proposition}
Let $-1\leq m \leq n -2$ and $r\geq m+2$. Let $V_{m}^r$ be the $H(\div)$-conforming finite element defined in Lemma~\ref{lm:Stenberg}. The following discrete inf-sup condition holds with a constant independent of $h$
\begin{equation}\label{eq:discretedivinfsup}
\|q_h\|_0\lesssim \sup_{\bs v_h\in V_{m}^r}\frac{(\div \bs v_h, q_h)}{\|\bs v_h\|_0+\|\div \bs v_h\|_0}\quad\forall~q_h\in Q_h.
\end{equation}
\end{proposition}

\section{Geometric Decompositions of Matrix Face Elements}\label{sec:geodecompdivmatrix}
In this section, we generalize the geometric decomposition of $H(\div)$-conforming vector finite elements to two $H(\div)$-conforming matrix finite elements: the traceless matrix $\mathbb T$ and the symmetric matrix $\mathbb S$. 

\subsection{Traceless matrix elements}
We consider the $H(\div,\Omega;\mathbb T)$-conforming finite element spaces, where $\mathbb T\in \mathbb R^{n\times n}$ is the set of square matrices with vanishing trace, i.e., the sum of the diagonal is zero. 

We start from the tensor product of the Lagrange element and $\mathbb T$:
\begin{align*}
\mathbb P_r(T; \mathbb T)  &= \Oplus_{\ell = 0}^n\Oplus_{f\in \Delta_{\ell}(T)} \left [b_f \mathbb P_{r - (\ell +1)} (f) \otimes \mathbb T \right ].
\end{align*}
That is each component of the matrix function is a Lagrange element of degree $r$ and thus is continuous. To be $H(\div)$-conforming, however, normal continuity is sufficient.

To impose the normal continuity of a traceless matrix function, the key is a $t$-$n$ decomposition at each sub-simplex. Here the $t$-$n$ decomposition is with respect to the second component in the tensor product form $\bs u\otimes \bs v$ of representing a matrix. Given a sub-simplex $f\in \Delta_{\ell}(T)$, choose a $t$-$n$ basis $\{\bs t_i^f, \bs n^f_j\}_{i=1,\ldots,\ell}^{j = 1,\ldots,n-\ell}$ and decompose $\mathbb R^n = \mathscr T^f\oplus^{\perp} \mathscr N^f$. All basis vectors are normalized but may not be mutually orthogonal. By the tensor product the $n\times n$ matrix space $\mathbb M$ has the following decomposition
\begin{equation}\label{eq:MdecT}
\mathbb M = (\mathbb R^n \otimes \mathscr T^f )\oplus (\mathbb R^n \otimes \mathscr N^f). 
\end{equation}
For a matrix $\bs A\in \mathbb M$, $\tr_F^{\div}(\bs A)=\bs A\bs n_F$ and thus $\tr_F^{\div}(\mathbb R^n \otimes \mathscr T^f ) = 0$. The normal component $\mathbb R^n \otimes \mathscr N^f$ will contribute to the normal trace. 

We then modify \eqref{eq:MdecT} for $\mathbb M$ to impose the traceless constraint while not changing $\tr^{\div}$. When computing the trace of a matrix, we use ${\rm trace}( \bs u\otimes \bs v) =\bs v\cdot \bs u$. We pick up the element $\bs t_1^f\otimes \bs t_1^f\in \mathbb R^n \otimes \mathscr T^f$ and use it to modify the basis in $(\mathbb R^n \otimes \mathscr T^f )\oplus (\mathbb R^n \otimes \mathscr N^f)$ to get the following $t$-$n$ decomposition on $f\in \Delta_{\ell}(T)$ for $\ell \geq 1$: 
\begin{align*}
\mathscr T^f(\mathbb T) &:= \textrm{span}\big\{\bs n_j^f\otimes\bs t_i^f,  1\leq i\leq\ell,1\leq j\leq n-\ell\big\} \\
&\quad\;\oplus\textrm{span}\big\{\bs t_i^f\otimes \bs t_j^f - (\bs t_i^f\cdot \bs t_j^f) \bs t_1^f\otimes \bs t_1^f,  1\leq i, j\leq \ell\big\},\\
\mathscr N^f(\mathbb T) &:= \textrm{span}\big\{\bs t_i^f\otimes \bs n_j^f,  1\leq i\leq\ell,1\leq j\leq n-\ell\big\} \\
&\quad\;\oplus\textrm{span}\big\{\bs n_i^f\otimes \bs n_j^f - (\bs n_i^f\cdot \bs n_j^f) \bs t_1^f\otimes \bs t_1^f,  1\leq i, j\leq n-\ell\big\}.
\end{align*}
By counting the dimensions, it is easy to show the direct decomposition
%
\begin{equation*}
 \mathbb T = \mathscr T^f(\mathbb T) \oplus \mathscr N^f(\mathbb T), \quad \quad f\in \Delta_{\ell}(T), \ell= 1,\ldots, n.
\end{equation*}
For $\ell = 0$, i.e., at vertex $\texttt{v}\in \Delta_0(T)$, we understand $ \mathscr T^\texttt{v}(\mathbb T) = \{ 0\}$ and $\mathscr N^\texttt{v}(\mathbb T) = \mathbb T$. For $\ell = n$, $ \mathscr T^T(\mathbb T) = \mathbb T$ and $\mathscr N^T(\mathbb T) = \{ 0\}$.
Coupled with the bubble polynomials, we define
\begin{align*}
\mathbb B_r\mathscr T^f(\mathbb T) := b_f\mathbb P_{r-(\ell +1)}(f)\otimes \mathscr T^f(\mathbb T),\quad
\mathbb B_r\mathscr N^f(\mathbb T) := b_f\mathbb P_{r-(\ell +1)}(f)\otimes \mathscr N^f(\mathbb T).
\end{align*}

\begin{lemma} \label{lm:TLdec}
The shape function space $ \mathbb P_r(T; \mathbb T)$ has a geometric decomposition
\begin{equation*}
 \mathbb P_r(T; \mathbb T)  = \Oplus_{\ell = 0}^n\Oplus_{f\in \Delta_{\ell}(T)} \left [\mathbb B_r\mathscr T^f(\mathbb T)  \oplus \mathbb B_r\mathscr N^f(\mathbb T) \right ].
\end{equation*}
A function $\boldsymbol{A}\in \mathbb P_r(T; \mathbb T)$ is uniquely determined by the DoFs
\begin{subequations}\label{eq:dofT}
\begin{align}
\label{eq:dofT1}
\int_f\boldsymbol{A}:\boldsymbol{q}\dd s,
&\quad\quad ~\boldsymbol{q}\in \mathbb P_{r-(\ell +1)}(f)\otimes\mathscr T^f(\mathbb T), f\in \Delta_{\ell}(T), \ell = 1,\ldots, n,\\
\label{eq:dofT2}
\int_f\boldsymbol{A}:\boldsymbol{q}\dd s,
&\quad\quad ~\boldsymbol{q}\in \mathbb P_{r-(\ell +1)}(f)\otimes\mathscr N^f(\mathbb T), f\in \Delta_{\ell}(T), \ell = 0,\ldots, n-1.
\end{align}
\end{subequations}
\end{lemma}
\begin{proof}
Since $\mathbb T =  \mathscr T^f(\mathbb T) \oplus \mathscr N^f(\mathbb T)$, DoFs~\eqref{eq:dofT1}-\eqref{eq:dofT2} are equivalent to 
\begin{equation*}
\int_f\boldsymbol{A}:\boldsymbol{q}\dd s, \quad \quad~\boldsymbol{q}\in \mathbb P_{r - (\ell +1)} (f)\otimes \mathbb T, f\in \Delta_{\ell}(T), \ell = 0,\ldots, n.
\end{equation*}
Then the unisolvence follows from the unisolvence of the Lagrange element.
\end{proof}

Define the bubble polynomial space
\[
\mathbb B_r(\div,T; \mathbb T):= \mathbb P_r(T; \mathbb T)\cap \ker({\rm tr}^{\div}).
\]
%
Follow the same proof of Lemma~\ref{lem:divbubbletracespacedecomp}, we have the characterization of the bubble space. 
\begin{lemma}
For $r\geq 2$, it holds that
\begin{equation*}
\mathbb B_r(\div,T; \mathbb T)=\Oplus_{\ell = 1}^n\Oplus_{f\in \Delta_{\ell}(T)} \mathbb B_r\mathscr T^f(\mathbb T),
\end{equation*}
and
\begin{equation*}
{\rm tr}^{\div}: \Oplus_{\ell = 0}^{n-1}\Oplus_{f\in \Delta_{\ell}(T)} \mathbb B_r\mathscr N^f(\mathbb T) \to {\rm tr}^{\div} \, \mathbb P_r(T; \mathbb T)
\end{equation*}
is a bijection.
\end{lemma}

Similar to the generalized Stenberg element, we can redistribute some normal DoFs onto the $(n-1)$-dimensional faces to obtain the following $H(\div;\mathbb T)$ element.
\begin{theorem}[$H(\div;\mathbb T)$-conforming finite elements]\label{thm:femTtensor}
Let $0\leq m\leq n -2$. For each $f\in \Delta_{\ell}(\mathcal T_h)$ with $\ell \leq m$, we choose $n-\ell$ normal vectors $\{\bs n^f_{1}, \ldots, \bs n^f_{n-\ell}\}$. For each $F\in \Delta_{n-1}(\mathcal T_h)$, choose a normal vector $\bs n_F$. Then the {\rm DoFs}
\begin{subequations}\label{eq:divfemTkdof}
\begin{align}
\label{eq:divfemTkdof1}
\boldsymbol{A}(\texttt{v}),
&\quad \texttt{v}\in \Delta_{0}(\mathcal{T}_h), \bs A\in \mathbb T, \\
\label{eq:divfemTkdof2}
\int_f(\boldsymbol{A}\bs n^f_i)\cdot\boldsymbol{q}\dd s,
&\quad f\in \Delta_{\ell}(\mathcal{T}_h), \boldsymbol{q}\in \mathbb P_{r-(\ell +1)}(f;\mathbb R^n), \\
&\quad i=1,\ldots, n-\ell, \; \ell = 1,\ldots, m, \notag\\
\label{eq:divfemTkdof3}
\int_f (\boldsymbol{A}\bs n_F)|_F\cdot\boldsymbol{q}\dd s, &\quad F\in \Delta_{n-1}(\mathcal T_h), \; f\in \Delta_{\ell}(F),\\
&\quad \boldsymbol{q}\in \mathbb P_{r-(\ell +1)}(f;\mathbb R^n),   \; \ell = m+1,\ldots, n-1, \notag\\
\label{eq:divfemTkdof4}
\int_T \boldsymbol{A}:\boldsymbol{q} \dx, &\quad   T\in \mathcal T_h, \; \bs q\in \mathbb B_r(\div, T;\mathbb T),
\end{align}
\end{subequations}
will determine a space $V^r(\mathbb T)\subset H(\div,\Omega;\mathbb T)$, where
\begin{align*}
V^r(\mathbb T):=\{&\boldsymbol{A}\in L^2(\Omega;\mathbb T): \boldsymbol{A}|_T\in\mathbb P_r(T; \mathbb T) \quad\forall~T\in\mathcal T_h, \\
&\textrm{DoFs \eqref{eq:divfemTkdof1}-\eqref{eq:divfemTkdof2} are single-valued across $f\in \Delta_{\ell}(\mathcal{T}_h)$ for $\ell=0,\ldots, m$}, \\
&\textrm{DoF \eqref{eq:divfemTkdof3} is single-valued across $F\in \Delta_{n-1}(\mathcal{T}_h)$} \}.
\end{align*}
\end{theorem}
\begin{proof}
For $T\in\mathcal T_h$ and $f\in\Delta_{\ell}(T)$,
both $\{\bs n^f_1, \ldots, \bs n^f_{n-\ell}\}$ and $\{\bs n_F, F\in \Delta_{n-1}(\mathcal T)$, $f\subset F\}$ are basis of the normal plane $\mathscr N^f$.
Then DoFs~\eqref{eq:divfemTkdof} restricted to $T$ are equivalent to DoFs~\eqref{eq:dofT}, thus uniquely determine $\mathbb P_r(T; \mathbb T)$. 

For $F\in\Delta_{n-1}(\mathcal T_h)$,
DoFs~\eqref{eq:divfemTkdof1}-\eqref{eq:divfemTkdof2} restricted to $F$ will determine
\begin{align*}
\int_f (\boldsymbol{A}\bs n_F)|_F\cdot\boldsymbol{q}\dd s, &\quad f\in \Delta_{\ell}(F), \boldsymbol{q}\in \mathbb P_{r-(\ell +1)}(f;\mathbb R^n),   \; \ell = 0,\ldots, m.
\end{align*}
Then by the unisolvence of Lagrange element in Theorem~\ref{thm:Lagrangedec}, $(\boldsymbol{A}\bs n_F)|_F$ is determined by DoFs~\eqref{eq:divfemTkdof1}-\eqref{eq:divfemTkdof3} restricted to $F$. Therefore $V^r(\mathbb T)\subset H(\div,\Omega;\mathbb T)$.
\end{proof}

To have a better discrete div stability, we modify the face DoFs.

\begin{corollary}
With the same setting in Theorem \ref{thm:femTtensor} and further assume $r\geq m+2$ for $0\leq m\leq n-3$, and $r\geq n+1$ for $m=n-2$. Replacing DoF \eqref{eq:divfemTkdof3} by
\begin{equation}\label{eq:divfemTkdof3new}
\int_F (\boldsymbol{A}\bs n_F)\cdot\boldsymbol{q}\dd s, \quad F\in \Delta_{n-1}(\mathcal T_h),
\, \boldsymbol{q}\in \big[(\mathbb B_{r, m+1}(F)\cap \mathbb P_{1}^{\perp}(F))\oplus \mathbb P_{1}(F)\big]\otimes\mathbb R^n,
\end{equation}
will define the same finite element space.
\end{corollary}
\begin{proof}
When $m = n-2$, we require $r\geq n+1$ so that $r - (m+1+1)\geq 1$ and $ \dim \mathbb B_{r, m+1}(F)\geq \dim \mathbb P_1(F)$. When $0\leq m\leq n - 3$ and $r\geq m+2$, $\dim \mathbb B_{r, m+1}(F)\geq |\Delta_{m+1}(F)| \geq \dim \mathbb P_1(F)$. Therefore the number of DoFs remains the same as DoF \eqref{eq:divfemTkdof3}.

Vanishing DoFs \eqref{eq:divfemTkdof1}-\eqref{eq:divfemTkdof2} will imply $\boldsymbol{A}\bs n_F|_{F}\in \mathbb B_{r, m+1}(F)\otimes \mathbb R^n$, which can be determined by \eqref{eq:divfemTkdof3new}. So the unisolvence follows.
 
As the change will not affect the continuity, it will define the same finite element space. 
\end{proof}

DoF \eqref{eq:divfemTkdof3new} is more friendly for verifying the discrete inf-sup condition and \eqref{eq:divfemTkdof3} is better for the uni-solvence and implementation.

When $m=0$, it is the generalization of Stenberg element for vector functions to traceless tensor functions. Almost all DoFs are redistributed to face $F$ except at the vertex, where the traceless constraint is imposed. The case $n=3,r\geq 2$ is the $H(\div;\mathbb T)$ element constructed in~\cite{Chen;Huang:2020Discrete}.


\begin{remark}\rm Comparing with the vector face element, cf. Lemma~\ref{lm:Stenberg}, $m$ starts from $0$ not $-1$.
Namely the $H(\div;\mathbb T)$-element should be continuous at vertices. We argue that the continuity at vertices is also necessary. 
Take a vertex in $\Delta_{0}(T)$, for example $\texttt{v}_0$. Then $\bs A\bs n_{F_i}(\texttt{v}_0)$ is determined by the vector $\bs A\bs n_{F_i}|_{F_{i}}\in \mathbb R^n$ for $i=1,\ldots, n$. If it is continuous on each face but not on vertices, the number of elements in $\bs A\bs n_{F_i}(\texttt{v}_0)$ is $n$ for each face. Running $i$ from $1$ to $n$, $\bs A(\texttt{v}_0)$ is determined by $n^2$ conditions, which is more than $\dim \mathbb T = n^2-1$.
\end{remark}

\subsection{Discrete div stability for traceless tensors}
We first show the div stability of the bubble space $\mathbb B_r(\div,T; \mathbb T)$.
Denote by $\bs t_{i,j}$ the edge vector from $\texttt{v}_i$ to $\texttt{v}_j$. By computing the constant directional derivative $\boldsymbol t_{i,j}\cdot\nabla \lambda_{\ell}$ by values on the two vertices, we have
\begin{equation*}
\boldsymbol t_{i,j}\cdot\nabla \lambda_{\ell}=\delta_{j\ell}-\delta_{i\ell}=\begin{cases}
1, & \textrm{ if } \ell=j,\\
-1, & \textrm{ if } \ell=i,\\
0, & \textrm{ if } \ell\neq i,j.
\end{cases}  
\end{equation*}

\begin{lemma}\label{lm:Tbasis}
The set of traceless tensors $\{\nabla\lambda_i\otimes\boldsymbol{t}_{i+1,j}\}_{j\in\{0,\ldots,n\}\backslash\{i,i+1\}}^{i=0,\ldots,n}$ is a basis of the traceless tensor space $\mathbb T$. Its dual basis is $\{\boldsymbol{t}_{j,i}\otimes\nabla\lambda_j+\frac{1}{n}I\}_{j\in\{0,\ldots,n\}\backslash\{i,i+1\}}^{i=0,\ldots,n}$. All indices are modulo $n$.
\end{lemma}
\begin{proof}
It suffices to prove 
\[
(\boldsymbol{t}_{\ell,k}\otimes\nabla\lambda_{\ell}):(\nabla\lambda_i\otimes\boldsymbol{t}_{i+1,j})=\delta_{ik}\delta_{j\ell}
\]
for  $0\leq i,k\leq n$, $j\in\{0,\ldots,n\}\backslash\{i,i+1\}$, and $\ell\in\{0,\ldots,n\}\backslash\{k,k+1\}$.

When $i=k$, by $\ell\neq i,i+1$, it follows
\[
(\boldsymbol{t}_{\ell,k}\otimes\nabla\lambda_{\ell}):(\nabla\lambda_i\otimes\boldsymbol{t}_{i+1,j})=\nabla\lambda_{\ell}\cdot\boldsymbol{t}_{i+1,j}=\delta_{j\ell}.
\]
When $i=\ell$, by $i\neq j$, it follows
\[
(\boldsymbol{t}_{\ell,k}\otimes\nabla\lambda_{\ell}):(\nabla\lambda_i\otimes\boldsymbol{t}_{i+1,j})=0.
\]
When $i\neq k,\ell$, clearly $(\boldsymbol{t}_{\ell,k}\otimes\nabla\lambda_{\ell}):(\nabla\lambda_i\otimes\boldsymbol{t}_{i+1,j})=0$.
\end{proof}

Let ${\rm RT}= \{a\boldsymbol x + \boldsymbol b: a\in \mathbb R, \boldsymbol b \in \mathbb R^n\}$. For a matrix $\bs A$, define $\dev \bs A = \bs A - \frac{1}{n}{\rm trace}(\bs A) I$ as the projection of $\bs A$ to $\mathbb T$. It is straight forward to verify $\ker(\dev\grad) = {\rm RT}$. Again let ${\rm RT}^{\perp}$ be the $L^2$-orthogonal complement in $L^2(T; \mathbb R^n)$.

\begin{lemma}\label{lm:TdivBrsurjection}
For each $T\in\mathcal T_h$,
it holds
\begin{equation}\label{eq:TdivBrsurjection}
\div \mathbb B_r(\div, T; \mathbb T) = \mathbb P_{r-1}(T;\mathbb R^n)\cap{\rm RT}^{\perp}.
\end{equation}
\end{lemma}
\begin{proof}
It follows from the integration by parts that
 \[
 \div \mathbb B_r(\div, T; \mathbb T) \subseteq (\mathbb P_{r-1}(T;\mathbb R^n)\cap{\rm RT}^{\perp}).
 \]

We claim the equality holds. If $ \div \mathbb B_r(\div, T; \mathbb T) \subset (\mathbb P_{r-1}(T;\mathbb R^n)\cap{\rm RT}^{\perp})$, then there exists $\boldsymbol{u}\in \mathbb P_{r-1}(T;\mathbb R^n)\cap{\rm RT}^{\perp}$ satisfying the orthogonality condition $(\boldsymbol{u}, \div\boldsymbol{A})_T=0$ for any $\boldsymbol{A}\in\mathbb B_r(\div, T; \mathbb T)$. Equivalently 
\[
(\dev\grad\boldsymbol{u}, \boldsymbol{A})_T=0\quad\forall~\boldsymbol{A}\in\mathbb B_r(\div, T; \mathbb T).
\]
By expressing $\dev\grad\boldsymbol{u}=\sum\limits_{i=0}^{n}\sum\limits_{j\in\{0,\ldots,n\}\backslash\{i,i+1\}}q_{ij}(\boldsymbol{t}_{j,i}\otimes\nabla\lambda_j+\frac{1}{n}I)$ with $q_{ij}\in \mathbb P_{r-2}(T)$, we choose 
\[
\boldsymbol{A}
=\sum_{i=0}^{n}\sum_{j\in\{0,\ldots,n\}\backslash\{i,i+1\}}\lambda_{i+1}\lambda_{j}q_{ij}\nabla\lambda_i\otimes\boldsymbol{t}_{i+1,j}\in\mathbb B_r(\div, T; \mathbb T).
\]
Then we have
\[
\sum_{i=0}^{n}\sum_{j\in\{0,\ldots,n\}\backslash\{i,i+1\}}(\lambda_{i+1}\lambda_{j}q_{ij}, q_{ij})_T=0.
\]
Therefore $q_{ij}=0$ for all $i$ and $j$. Thus $\boldsymbol{u}=0$.
\end{proof}

We mention that characterization \eqref{eq:TdivBrsurjection} in three dimensions is firstly proved in~\cite{HuLiang2021}.

\begin{proposition}[Discrete inf-sup condition for $H(\div;\mathbb T)$-conforming finite elements]\label{pro:Tdiscretedivinfsup}
Let $0\leq m\leq n -2$. Let $r\geq m+2$ for $0\leq m\leq n-3$, and $r\geq n+1$ for $m=n-2$.
Let $V^r(\mathbb T)$ be defined in Theorem~\ref{thm:femTtensor}. It holds the discrete inf-sup condition
\begin{equation}\label{eq:Tdiscretedivinfsup}
\|\boldsymbol{v}_h\|_0\lesssim \sup_{\boldsymbol{A}_h\in V^r(\mathbb T)}\frac{(\div\boldsymbol{A}_h, \boldsymbol{v}_h)}{\|\boldsymbol{A}_h\|_0+\|\div\boldsymbol{A}_h\|_0}\quad\forall~\boldsymbol{v}_h\in Q_h,
\end{equation}
where $Q_h:=\{\boldsymbol{v}_h\in L^2(\Omega;\mathbb R^n): \boldsymbol{v}_h|_T\in\mathbb P_{r-1}(T;\mathbb R^n) \textrm{ for each } T\in\mathcal T_h\}.$
\end{proposition}
\begin{proof}
First
there exists $\boldsymbol{A}\in H^1(\Omega;\mathbb T)$ such that
\begin{equation}\label{eq:20221009-1}
\div\boldsymbol{A}=\boldsymbol{v}_h,\quad \|\boldsymbol{A}\|_1\lesssim \|\boldsymbol{v}_h\|_0.    
\end{equation}
Thanks to DoFs \eqref{eq:divfemTkdof} and \eqref{eq:divfemTkdof3new},
we let $\widetilde{\boldsymbol{A}}_h\in V^r(\mathbb T)$ satisfy 
\begin{align*}
\int_F (\widetilde{\boldsymbol{A}}_h\bs n_{F})\cdot\boldsymbol{q} \dd s&=\int_F(\boldsymbol{A}\bs n_{F})\cdot\boldsymbol{q}\dd s, \quad \quad~\boldsymbol{q}\in \mathbb P_{1}(F;\mathbb R^n), F\in \Delta_{n-1}(\mathcal T_h), 
\end{align*}
and other DoFs vanish. By the scaling argument,
\begin{equation}\label{eq:20221009-2}
\|\widetilde{\boldsymbol{A}}_h\|_0+\|\div\widetilde{\boldsymbol{A}}_h\|_0\lesssim \|\boldsymbol{A}\|_1\lesssim \|\boldsymbol{v}_h\|_0.
\end{equation}
Then through integration by parts we have 
$\div(\widetilde{\boldsymbol{A}}_h-\boldsymbol{A})|_T\in\mathbb P_{r-1}(T;\mathbb R^n)\cap{\rm RT}^{\perp}$ for each $T\in\mathcal T_h$. By Lemma~\ref{lm:TdivBrsurjection}, there exists $\bs b_h\in L^2(\Omega;\mathbb T)$ such that $\bs b_h|_T\in\mathbb B_r(\div, T; \mathbb T)$ for each $T\in\mathcal T_h$, and
\begin{equation}\label{eq:20221009-3}
\div \bs b_h=\div(\boldsymbol{A}-\widetilde{\boldsymbol{A}}_h), \quad \|\bs b_h\|_{0,T}\lesssim h_T\|\div(\widetilde{\boldsymbol{A}}_h-\boldsymbol{A})\|_{0,T}.
\end{equation}
Take $\boldsymbol{A}_h=\widetilde{\boldsymbol{A}}_h+\bs b_h\in V^r(\mathbb T)$. By~\eqref{eq:20221009-1} and~\eqref{eq:20221009-3}, it holds
\begin{equation}\label{eq:20221009-4}
\div \boldsymbol{A}_h=\div\widetilde{\boldsymbol{A}}_h+\div \bs b_h=\div\boldsymbol{A}=\boldsymbol{v}_h. 
\end{equation}
It follows from~\eqref{eq:20221009-2} and~\eqref{eq:20221009-3} that
\begin{align}
\|\boldsymbol{A}_h\|_0+\|\div\boldsymbol{A}_h\|_0&=\|\boldsymbol{A}_h\|_0+\|\boldsymbol{v}_h\|_0\leq \|\widetilde{\boldsymbol{A}}_h\|_0+\|\bs b_h\|_0+\|\boldsymbol{v}_h\|_0 \notag\\
&\lesssim \|\widetilde{\boldsymbol{A}}_h\|_0+h\|\div\widetilde{\boldsymbol{A}}_h\|_{0}+\|\boldsymbol{v}_h\|_0\lesssim\|\boldsymbol{v}_h\|_0. \label{eq:20221009-5}
\end{align}
Combining~\eqref{eq:20221009-4}-\eqref{eq:20221009-5} yields~\eqref{eq:Tdiscretedivinfsup}.
\end{proof}

\subsection{Symmetric matrix elements}\label{sec:HdivS3D}
We start from the tensor product of the Lagrange element and the symmetric matrix $\mathbb S$:
\begin{align}
\label{eq:SLagrange}
\mathbb P_r(T; \mathbb S)  &= \Oplus_{\ell = 0}^n\Oplus_{f\in \Delta_{\ell}(T)} \left [b_f \mathbb P_{r - (\ell +1)} (f) \otimes \mathbb S \right ].
\end{align}
The construction process is similar to the traceless case in which $t$-$n$ decompositions of $\mathbb S$ at sub-simplices are the key. Additional complication arises as the $n(n-1)/2$ symmetry constraints are more complicated than only $1$ traceless constraint. 

Let $\{\bs v_1, \ldots, \bs v_n\}$ be a basis of $\mathbb R^n$. Then $\mathbb M = \{ \sum_{i,j=1}^na_{ij}\bs v_i\otimes \bs v_j, a_{ij}\in\mathbb R\}$ is the space of $n\times n$-matrices. An element $\boldsymbol{A}\in \mathbb M$ can be identified with the coefficient matrix $(a_{ij})$ and will be still denoted by $\boldsymbol{A}$. For a better explanation, we illustrate an $n\times n$-matrix by $n\times n$ blocks; see Fig.~\ref{fig:TNdecompositionS} for the case $n=3$. Block $(i,j)$ corresponds to the basis function $ \bs v_i\otimes \bs v_j$. We will identify $n(n+1)/2$ blocks and modify corresponding basis function to form a basis of $\mathbb S$ with consideration of the normal continuity. 


\begin{figure}[htbp]
\subfigure[$t$-$n$ decomposition at a vertex. No tangent component.]{
\begin{minipage}[t]{0.23\linewidth}
\centering
\includegraphics*[width=3cm]{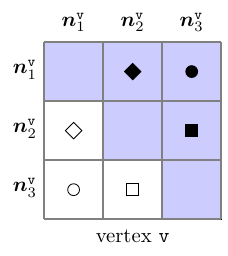}
\end{minipage}}
\;
\subfigure[$t$-$n$ decomposition on an edge.]
{\begin{minipage}[t]{0.23\linewidth}
\centering
\includegraphics*[width=3cm]{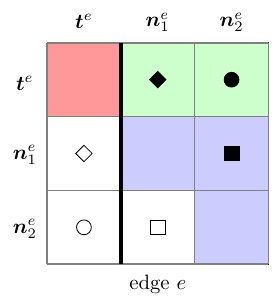}
\end{minipage}}
\;
\subfigure[$t$-$n$ decomposition on a face.]
{\begin{minipage}[t]{0.23\linewidth}
\centering
\includegraphics*[width=3cm]{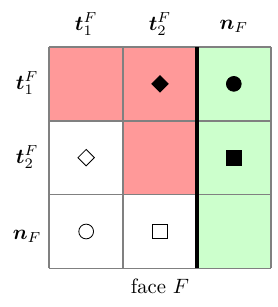}
\end{minipage}}
\;
\subfigure[$t$-$n$ decomposition of a tetrahedron. No normal component.]
{\begin{minipage}[t]{0.23\linewidth}
\centering
\includegraphics*[width=3cm]{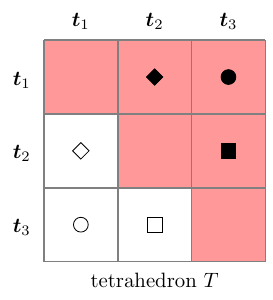}
\end{minipage}}
\caption{Several $t$-$n$ decompositions for $\mathbb S$ in $\mathbb R^3$. Blocks with the same symbol (circle, square, or diamond) are in the same constraint sequence and the white block is used as the pair index. Color of the block represents: Green: free rows and free blocks; Blue: all free indices not in free rows; Red: bubble functions. Blue or green blocks are free indices in $\mathscr N^f(\mathbb S)$. All white blocks are pair indices and the corresponding coefficients are determined by the free variables through the constraints.}
\label{fig:TNdecompositionS}
\end{figure}

Given an $f\in \Delta_{\ell}(T)$, choose a $t$-$n$ basis $\{\bs t_i^f, \bs n^f_j\}_{i=1,\ldots,\ell}^{j = 1,\ldots,n-\ell}$ and decompose $\mathbb R^n = \mathscr T^f\oplus \mathscr N^f$. All basis vectors are normalized. We have the matrix decomposition
\begin{equation*}
\mathbb M = (\mathbb R^n \otimes \mathscr T^f )\oplus (\mathbb R^n \otimes \mathscr N^f). 
\end{equation*}
We modify each component to impose the symmetric constraint. For the tangential component, we simply take
\begin{equation*}
\mathscr T^f(\mathbb S) := (\mathbb R^n \otimes \mathscr T^f )\cap \mathbb 
S.
\end{equation*}
Using the tensor product with the scalar bubble polynomial, we can construct 
\[
\mathbb B_r\mathscr T^f(\mathbb S) :=b_f\mathbb P_{r-(\ell +1)}(f)\otimes\mathscr T^f(\mathbb S)\subset \ker({\rm tr}^{\div})\cap \mathbb P_r(T; \mathbb S).
\]

We provide a specific example for $n=3, \ell = 2$, i.e., a $t$-$n$ decomposition on a face $F$ of a tetrahedron. Choose a $t$-$n$ basis $\{\bs t^F_1, \bs t^F_2, \bs n_F\}$. An element in $\mathbb R^n \otimes \mathscr T^F$ has the matrix representation $
\begin{pmatrix}
 * & * & 0\\
 * & * & 0\\
 * & * & 0
\end{pmatrix},
$
where $*$ represents a generic value. The symmetric constraint implies an element in $\mathscr T^{F}(\mathbb S)$ has the form
$
\begin{pmatrix}
 * & * & 0\\
   & * & 0\\
0 & 0 & 0
\end{pmatrix}
$. The value $a_{21}$ is left blank since due to the symmetric constraint it is equal to $a_{12}$. We call $(1,2)$ a free index while $(2,1)$ is the pair index of $(1,2)$.

We then move to the normal component. A naive definition of $\mathscr N^f(\mathbb S)$ would be 
$(\mathbb R^n \otimes \mathscr N^f) \cap \mathbb S$. Unfortunately in general for three subspaces of a vector space:
\[
(A\cap C) \oplus (B\cap C) \subseteq (A\oplus B)\cap C.
\]
And the equality may not hold.  We continue the example with $n=3, \ell = 2$. An element in $\mathbb R^n \otimes \mathscr N^F$ has the matrix form $
\begin{pmatrix}
 0 & 0 & *\\
 0 & 0 & *\\
 0 & 0 & *
\end{pmatrix}
$
w.r.t the basis $\{\bs t^F_1, \bs t^F_2, \bs n_F\}$. An element in $(\mathbb R^n \otimes \mathscr N^F) \cap \mathbb S$ has the form
$
\begin{pmatrix}
 0 & 0 & 0\\
  0 & 0 & 0\\
  0 & 0 & *
\end{pmatrix}
$. Then 
\[
\dim ((\mathbb R^n \otimes \mathscr T^F) \cap \mathbb S) + \dim ((\mathbb R^n \otimes \mathscr N^F) \cap \mathbb S) = 3 + 1 = 4 < 6 = \dim \mathbb S. 
\]
The discrepancy is due to the fact that the constraints $a_{13} = a_{31}$ and $a_{23} = a_{32}$ have been  used duplicately in both $(\mathbb R^n \otimes \mathscr T^F) \cap \mathbb S$ and $(\mathbb R^n \otimes \mathscr N^F) \cap \mathbb S$. We need to make sure one constraint is used only once either in $\mathscr T^f(\mathbb S)$ or $\mathscr N^f(\mathbb S)$.

To this end, for each constraint, we set the index $(j, i)$ with $i< j$ as the pair index and call $(i, j)$ a free index. The pair index value $a_{ji}$ is determined by the free variable $a_{ij}$ through the constraint $a_{ji} = a_{ij}$. 

We introduce the concept: {\em normal constraints}. We call the constraint $a_{ij}=a_{ji}, i\neq j,$ a normal constraint if both $(i,j)$ and $(j,i)$ are in the normal component, i.e., $i,j\geq \ell +1$ w.r.t the $t$-$n$ basis $\{\bs t_i^f, \bs n^f_j\}_{i=1,\ldots,\ell}^{j = 1,\ldots,n-\ell}$. For example, in Fig.~\ref{fig:TNdecompositionS} (b), for an edge $e\in \Delta_1(T)$, there are $(n-1)(n-2)/2$ normal constraints. As the constraint involves two entries, no normal constraints for $F\in \Delta_{n-1}(T)$; see Fig.~\ref{fig:TNdecompositionS} (c). 
The normal constraints will be imposed inside the normal components. 

For non-normal constraint, in the pair index $(j,i)$, $i$ is inside the tangential component. We can use it to change the basis without affecting the normal trace. As an example, consider the circle block in Fig.~\ref{fig:TNdecompositionS} (b)-(c), the basis function is changing from $\bs t\otimes \bs n$ to $\sym(\bs t\otimes \bs n) = (\bs n\otimes \bs t + \bs t\otimes \bs n)/2$ and $ 2\sym(\bs t\otimes \bs n)\bs n_F =  (\bs t\otimes \bs n)\bs n_F$. 
By doing this way, we ensure each constraint is used only once and the normal trace remains unchanged.  

In summary, for $f\in\Delta_{\ell}(T)$, $\ell=0,\ldots,n$, 
define
\begin{align*}
\mathscr T^f(\mathbb S) &:= \textrm{span}\big\{\sym(\bs t_i^f\otimes \bs t_j^f),  1\leq i\leq j\leq \ell\big\},\\
\mathscr N^f(\mathbb S) &:= \textrm{span}\big\{\sym(\bs t_i^f\otimes \bs n_j^f),  1\leq i\leq\ell,1\leq j\leq n-\ell\big\} \\
&\quad\;\oplus\textrm{span}\big\{\sym(\bs n_i^f\otimes \bs n_j^f),  1\leq i\leq j\leq n-\ell\big\}.
\end{align*}
By counting the dimensions, it is easy to show the direct decomposition
\[
\mathbb S = \mathscr T^f(\mathbb S) \oplus \mathscr N^f(\mathbb S).
\]
Then tensor product with the scalar bubble function to obtain 
\[
\mathbb B_r\mathscr T^f(\mathbb S)=b_f\mathbb P_{r-(\ell +1)}(f)\otimes\mathscr T^f(\mathbb S),\quad \mathbb B_r\mathscr N^f(\mathbb S)=b_f\mathbb P_{r-(\ell +1)}(f)\otimes\mathscr N^f(\mathbb S).
\]
With the $t$-$n$ decomposition,~\eqref{eq:SLagrange} can be rewritten as
\[
\mathbb P_r(T; \mathbb S)  = \Oplus_{\ell = 0}^n\Oplus_{f\in \Delta_{\ell}(T)} \left [\mathbb B_r\mathscr T^f(\mathbb S)\oplus \mathbb B_r\mathscr N^f(\mathbb S) \right ].
\]
Define the polynomial bubble space 
\[
\mathbb B_r(\div, T; \mathbb S) := \ker({\rm tr}^{\div})\cap \mathbb P_r(T; \mathbb S).
\]
Again the tangential component $\mathbb B_r\mathscr T^f(\mathbb S)$ will contribute to the bubble space. Namely 
for $r\geq2$, it holds that
\begin{equation}\label{eq:divSbubbledecomp}
\mathbb B_r(\div, T; \mathbb S) =  \Oplus_{\ell = 1}^n\Oplus_{f\in \Delta_{\ell}(T)} \mathbb B_r\mathscr T^f(\mathbb S).
\end{equation}
Characterization of $\mathbb B_r(\div, T; \mathbb S)$ in \eqref{eq:divSbubbledecomp} is new and different with the one given in~\cite{Hu2015,HuZhang2015}.

The normal component $\mathbb B_r\mathscr N^f(\mathbb S)$ will determine the trace of the div operator. Notice that due to the constraint, not all $n^2$ components of the matrix are included when defining $\mathbb B_r\mathscr T^f(\mathbb S)$ and $\mathbb B_r\mathscr N^f(\mathbb S)$. In Fig.~\ref{fig:TNdecompositionS}, function values in all white blocks are determined by the corresponding free variables through the constraints. 
%

Unlike the traceless case, not all normal DoFs can be redistributed to faces since the normal constraint should be imposed on $\mathscr N^f \otimes \mathscr N^f$ with a global normal basis $\{\bs n_i^f\}$, which can be thought of as the super-smoothness induced by the constraint.
%
%
For example, for symmetric matrix $\boldsymbol{A}$, the restriction of $\boldsymbol{A}$ to the normal plane of $f$, which is a symmetric matrix of smaller size $(n-\ell)\times (n-\ell)$, should be continuous due to DoF \eqref{eq:divfemSkdof2}.  The tangential-normal component can be redistributed to face $F$. Therefore, in DoFs \eqref{eq:divfemSkdof}, \eqref{eq:divfemSkdof3} is posed globally and \eqref{eq:divfemSkdof4} is facewisely.


\begin{theorem}[$H(\div;\mathbb S)$-conforming finite elements]\label{thm:femStensor}
Let $0\leq m\leq n -2$. For each $f\in \Delta_{\ell}(\mathcal T_h)$, we choose a global $t$-$n$ basis $\{\bs t^f_1, \ldots, \bs t_{\ell}^f, \bs n^f_{1}, \ldots, \bs n^f_{n-\ell}\}$. Then the {\rm DoFs}
\begin{subequations}\label{eq:divfemSkdof}
\begin{align}
\label{eq:divfemSkdof1}
\boldsymbol{A}(\texttt{v}),
&\quad \texttt{v}\in \Delta_{0}(\mathcal{T}_h), \bs A\in \mathbb S, \\
\label{eq:divfemSkdof2}
\int_f((\bs n^f_i)^{\intercal}\boldsymbol{A}\bs n^f_j)\,{q}\dd s,
&\quad f\in \Delta_{\ell}(\mathcal{T}_h), {q}\in \mathbb P_{r-(\ell +1)}(f), \\
&\quad 1\leq i\leq j\leq n-\ell, \; \ell = 1,\ldots, n-1, \notag\\
\label{eq:divfemSkdof3}
\int_f((\bs t^f_i)^{\intercal}\boldsymbol{A}\bs n^f_j) \, q\dd s,
&\quad f\in \Delta_{\ell}(\mathcal{T}_h), q\in \mathbb P_{r-(\ell +1)}(f), \\
&\quad i=1,\ldots, \ell, \; j = 1, \ldots, n- \ell, \, \ell = 1,\ldots, m, \notag\\
\label{eq:divfemSkdof4}
\int_f ((\bs t^f_i)^{\intercal}\boldsymbol{A}\bs n_F)|_F\,{q}\dd s, &\quad F\in \Delta_{n-1}(\mathcal T_h), \; f\in \Delta_{\ell}(F), {q}\in \mathbb P_{r-(\ell +1)}(f),\\
&\quad i=1,\ldots,\ell,   \; \ell = m+1,\ldots, n-1, \notag\\
\label{eq:divfemSkdof5}
\int_T \boldsymbol{A}:\boldsymbol{q} \dx, &\quad   T\in \mathcal T_h, \; \bs q\in \mathbb B_r(\div, T;\mathbb S),
\end{align}
\end{subequations}
will determine a space $V^r(\mathbb S)\subset H(\div,\Omega;\mathbb S)$, where
\begin{align*}
V^r(\mathbb S):=\{&\boldsymbol{A}\in L^2(\Omega;\mathbb S): \boldsymbol{A}|_T\in\mathbb P_r(T; \mathbb S) \quad\forall~T\in\mathcal T_h, \\
&\textrm{DoFs \eqref{eq:divfemSkdof1}-\eqref{eq:divfemSkdof2} are single-valued across $f\in \Delta_{\ell}(\mathcal{T}_h)$ for $\ell=0,\ldots, n-1$}, \\
&\textrm{DoF \eqref{eq:divfemSkdof3} is single-valued across $f\in \Delta_{\ell}(\mathcal{T}_h)$ for $\ell=1,\ldots, m$}, \\
&\textrm{DoF \eqref{eq:divfemSkdof4} is single-valued across $F\in \Delta_{n-1}(\mathcal{T}_h)$} \}.
\end{align*}
\end{theorem}
\begin{proof}
For $T\in\mathcal T_h$ and $f\in\Delta_{\ell}(T)$,
both $\{\bs n^f_1, \ldots, \bs n^f_{n-\ell}\}$ and $\{\bs n_F, F\in \Delta_{n-1}(\mathcal T_h)$, $f\subset F\}$ are basis of the normal plane $\mathscr N^f$. DoFs~\eqref{eq:divfemSkdof1}-\eqref{eq:divfemSkdof4} restricted to $T$ will determine normal component $\mathbb B_r\mathscr N^f(\mathbb S)$ and \eqref{eq:divfemSkdof5} for the tangential component $\mathbb B_r\mathscr T^f(\mathbb S)$, thus uniquely determine $\mathbb P_r(T; \mathbb S)$. 

For $F\in\Delta_{n-1}(\mathcal T_h)$,
DoFs~\eqref{eq:divfemSkdof1}-\eqref{eq:divfemSkdof4} restricted to $F$ will determine
\begin{align*}
\int_f (\boldsymbol{A}\bs n_F)|_F\cdot\boldsymbol{q}\dd s, &\quad f\in \Delta_{\ell}(F), \boldsymbol{q}\in \mathbb P_{r-(\ell +1)}(f;\mathbb R^n),   \; \ell = 0,\ldots, n-1,
\end{align*}
and thanks to the unisolvence of Lagrange element in Theorem~\ref{thm:Lagrangedec} will determine $(\boldsymbol{A}\bs n_F)|_F$. Therefore $V^r(\mathbb S)\subset H(\div,\Omega;\mathbb S)$.
\end{proof}

When $m=n-2$, it is the Hu-Zhang element~\cite{Hu2015,HuZhang2015}. 
When $m=0$, DoF \eqref{eq:divfemSkdof4} can be further merged to one and lead to the modification in~\cite[Lemma 4.5]{Chen;Huang:2021divFinite}
\begin{equation}\label{eq:divStndof}
\int_F (\Pi_F\boldsymbol{A}\bs n_F)\cdot\boldsymbol{q}\dd s, \quad F\in \Delta_{n-1}(\mathcal T_h), \boldsymbol{q}\in \textrm{ND}_{r-2}(F),
\end{equation}
where $\textrm{ND}_{r-2}(F):=\{\boldsymbol{q}\in \mathbb P_{r-1}(F;\mathbb R^{n-1}): \boldsymbol{q}\cdot\boldsymbol{x}\in\mathbb P_{r-1}(F)\}$ and $\Pi_F$ is the projection of a vector to the plane $\mathscr T^F$.

\subsection{Discrete div stability for symmetric tensors}
For each $T\in\mathcal T_h$, the range of the div operator on the bubble space of symmetric tensors~\cite{Hu2015,HuZhang2015} is
\begin{equation}\label{eq:SdivBrsurjection}
\div \mathbb B_r(\div, T; \mathbb S) = \mathbb P_{r-1}(T;\mathbb R^n)\cap {\rm RM}^{\perp},
\end{equation}
where ${\rm RM} = \{\bs N \boldsymbol x + \boldsymbol b: \bs N\in \mathbb K, \boldsymbol b \in \mathbb R^n\}$ and ${\rm RM}^{\perp}$ is the $L^2$-orthogonal complement in $L^2(T; \mathbb R^n)$. It can be proved similar to Lemma \ref{lm:TdivBrsurjection} and an abstract version will be proved in Lemma \ref{lm:divbubbleX}.

Applying the same argument as in Proposition~\ref{pro:Tdiscretedivinfsup}, we derive the div stability of space $V^r(\mathbb S)$ from \eqref{eq:SdivBrsurjection}. 

\begin{proposition}\label{pro:Sdiscretedivinfsup}
Let $r\geq n+1$ and $V^r(\mathbb S)$ be defined in Theorem~\ref{thm:femStensor}. It holds the discrete inf-sup condition
\begin{equation*}
\|\boldsymbol{v}_h\|_0\lesssim \sup_{\boldsymbol{A}_h\in V^r(\mathbb S)}\frac{(\div\boldsymbol{A}_h, \boldsymbol{v}_h)}{\|\boldsymbol{A}_h\|_0+\|\div\boldsymbol{A}_h\|_0}\quad\forall~\boldsymbol{v}_h\in Q_h.
\end{equation*}
\end{proposition}

Notice that due to the extra normal continuity, we cannot modify the face DoFs to relax the degree requirement  $r\geq n+1$ to $r\geq 2$ as we have done for the traceless element in Proposition~\ref{pro:Tdiscretedivinfsup}. 
Thanks to DoF \eqref{eq:divStndof}, the tangential-normal component contains $\mathbb P_1(F)$ for $r\geq 2$. Lower order $H(\div;\mathbb S)$-conforming finite elements are designed by enriching the symmetric quadratic polynomial space with only $(n + 1)$-order normal-normal face bubbles in~\cite{HuangZhangZhouZhu2023}, which have only $n(n+1)^2$ DoFs for the reduced one.

\section{Constraint Tensor Spaces}\label{sec:divtensorspace}
In this section we shall introduce the constraint tensor space $\mathbb X$ as a kernel space, and discover bases of $\mathbb X$. We first recall some background on differential forms, then give concrete formulae on the algebraic operator $s^{k,n-1}$ and define the constraint tensor space $\mathbb X$ as the kernel of $s^{k,n-1}$. Lastly we present two bases of $\mathbb X$.

\subsection{Background on differential forms}
\subsubsection{Increasing sequence}
We mainly follow the notation set in~\cite{ArnoldFalkWinther2009} but with some simplification. For non-negative integers $a, b, l, m$, with $0 \leqslant b-a \leqslant m-l$, define the set of increasing sequences as
$$
\Sigma(a: b, l: m):=\{\sigma:\{a, \ldots, b\} \rightarrow\{l, \ldots, m\} \mid \sigma(a)<\sigma(a+1)<\cdots<\sigma(b)\}.
$$
We will overload the notation $\sigma$ as its range, i.e., for $\sigma \in \Sigma(a: b, l: m)$, we use the same notation $\sigma$ to refer to the set $\{\sigma(i) \mid i=a, \ldots, b\} .$ The set $\Sigma(0: k, 0: n)$ will be mainly used for the description of sub-simplices, and $\Sigma(1: k, 1: n)$ for $k$ differential forms in $\mathbb R^n$. For $\sigma\in \Sigma(0: k, 0: n)$, $f_{\sigma}\in \Delta_{k}(T)$ is the sub-simplex formed by vertices with index $\{\sigma(0), \ldots, \sigma(k)\}$. On the other hand, for $f\in \Delta_k(T)$, the index of its vertices can be sorted in ascending order to get an increasing sequence $\sigma_f$. 

For $\sigma \in \Sigma(0: k, 0: n)$, denote by $\sigma^{*} \in \Sigma(1: n-k, 0: n)$ the complementary map characterized by
\begin{equation*}
\sigma \cup \sigma^{*}=\{0,1, \ldots, n\}.
\end{equation*}
For $\sigma \in \Sigma(1: k, 1: n)$, its complementary map $\sigma^{c} \in \Sigma(1: n-k, 1: n)$ satisfies
\begin{equation*}
 \sigma  \cup  \sigma^{c} =\{1, \ldots, n\}.
\end{equation*}
For the unique element in $\Sigma(a:b,a:b)$, we simply write it as $[a:b]$.

We follow~\cite{Licht2022} to introduce notation on the addition and subtraction of increasing sequences.
Let $\sigma \in \Sigma(a: b, l: m)$. If $q \in[l: m] \backslash  \sigma $, then we write $\sigma+q = q+ \sigma$ for the unique element of $\Sigma(a: b+1, l: m)$ with image $  \sigma  \cup \{q\} .$ In that case, we also write $\epsilon(q, \sigma)$ for the signum of the permutation that orders the sequence $[q, \sigma(a), \ldots, \sigma(b)]$ in the ascending order. For $q\in  \sigma $, $\sigma - q$ is the unique element in  $\Sigma(a: b-1, l: m)$ s.t. $(\sigma - q) + q = \sigma$.

\subsubsection{Differential forms}
We consider an $n$-dimensional domain $\Omega\subset \mathbb R^n$. Usually we choose a Cartesian coordinate and describe a point $x = (x_1, \ldots, x_n)\in \Omega$ in this coordinate. We also use $\mathbb R^n$ to denote the $n$-dimensional linear vector space, which can be identified with the space of points by identifying a point $x$ with the vector $\bs x = \vec{ox}$. We use $\partial_{x_i}$ as the unit vector from the origin $o$ to point $(0,\ldots, 1, \ldots,0)$, which is considered as an element in the tangent space $T_o\Omega$. Its dual basis of $(\mathbb R^n)^*$ is denoted by $\{\dd x_i\}_{i=1}^n$, i.e., $\dd x_i (\partial_{x_j}) = \delta_{i,j}$. We use the standard inner product of vectors to make $\mathbb R^n$ a Hilbert space, which introduces an inner product on $(\mathbb R^n)^*$: $\langle \dd x_i, \dd x_j \rangle = \delta_{i,j}, i,j=1,\ldots,n$. We shall reserve notation $\{\dd x_i\}_{i=1}^n$ for the orthonormal basis induced by the ambient orthonormal coordinate of $\mathbb R^n$. 

A generic basis will be denoted by $\{\dd y_i\}_{i=1}^n$, which may not be orthonormal. We can find another basis $\{\dd \hat{y}_i\}_{i=1}^n$ dual to $\{\dd y_i\}_{i=1}^n$  in the sense that $\langle \dd \hat{y}_i, \dd y_j\rangle =  \delta_{i,j}$. Indeed let $M =( \langle \dd y_i, \dd y_j \rangle)_{i,j=1}^n$. Then $ (\dd \hat{y}_1, \ldots, \dd \hat{y}_n)^{\intercal} = M^{-1}(\dd y_1, \ldots, \dd y_n)^{\intercal}$. When $\{\dd y_i\}_{i=1}^n$ is orthonormal, $\dd\hat{y}_i = \dd y_i$ for $i=1,\ldots, n$ as $M$ is identity.

For a vector space $V$, we define the space of {\rm exterior $k$-forms} as the alternating multilinear  functional space on $V^k:=\underbrace{V\times\cdots\times V}_{k}$ and denote it by $\alt^k(V)$ or simply $\alt^k$ if $V$ is clear in the context. By definition, $\alt^k\subset (V^k)^*$. The best way to study a $k$-form is through the action on $k$ vectors in $V$. 

Let $\omega \in \alt^p$ and $\eta \in \alt^q$, we define the {\rm wedge product} $\omega \wedge \eta \in \alt^{p+q}$:
\begin{eqnarray*}
(\omega \wedge \eta)(v_1,\ldots,v_{p+q})=\sum _{\sigma} {\rm sign} (\sigma)\omega (v_{\sigma(1)},\ldots,v_{\sigma(p)})\eta (v_{\sigma (p+1)},\ldots,v_{\sigma (p+q)}),
\end{eqnarray*}
where the sum is over all permutations $\sigma$ of $\{1,\ldots,p+q\}$, for which $\sigma(1)<\sigma(2)<\cdots<\sigma(p)$, $\sigma(p+1)<\sigma(p+2)<\cdots<\sigma(p+q)$, and ${\rm sign} (\sigma)$ is the signature of the permutation $\sigma$.
We have the determinant formula on the wedge product. 
For $\omega _i\in V^*, v_i\in V,$ $i=1,\ldots,p$,
\begin{equation*}
(\omega _1\wedge \cdots \wedge \omega _p)(v_1,\ldots,v_p)
= \det \left ( \omega _i(v_j) \right )_{i,j=1,\ldots,p}.
\end{equation*}

For a smooth manifold $\Omega$, a $k$-th order {\rm differential form} is a section of the tangent bundle $\cup _{x\in \Omega}\alt^k(T_x\Omega)$, where $T_x\Omega$ is the tangent space at $x$. The linear space formed by all $k$-th differential forms is denoted by $\Lambda ^k(\Omega)$, or simply $\Lambda ^k$. As $\Omega$ is a domain in $\mathbb R^n$, given any point $x$ in the interior of $\Omega$, the tangent space $T_x\Omega$ is isomorphism to $T_o\Omega$ by shifting the origin to $x$. That is we can use one basis $\{ \dd y_i\}$ for all $\alt^1(T_x\Omega), x\in \Omega$.  

For $\sigma\in \Sigma(1: k, 1: n)$, we extend the multi-index notation to write $\dd y_{\sigma}\in \alt^k$: 
$$
\dd y_{\sigma} :=\dd y_{\sigma(1)}\wedge\cdots\wedge\dd y_{\sigma(k)}.
$$
An element $\omega \in \Lambda ^k(\Omega)$ thus has a representation
\begin{equation}\label{eq:representation}
\omega 
=
\sum _{\sigma \in \Sigma(1:k,1:n)} a_{\sigma}(x)\dd y_{\, \sigma}, \quad x\in \Omega.
\end{equation}
For a manifold, the basis $\{\dd y_i \}_{i=1}^n$ is defined on a local chart while in~\eqref{eq:representation}, as $\Omega$ is flat, a global coordinate is used. 

Using~\eqref{eq:representation}, we define the exterior derivative $\dd: \Lambda ^k(\Omega)\rightarrow \Lambda ^{k+1}(\Omega)$ as: for  $\omega =\sum _{\sigma \in \Sigma(1:k,1:n)} a_{\sigma}(x)\dd y_{\sigma}$, define $\dd \omega \in \Lambda ^{k+1}(\Omega)$ by
$$
\dd \omega =\sum _{\sigma \in \Sigma (1:k,1:n)}\sum _{i\notin\sigma} \partial_{y_i} a_{\sigma} \dd y_i\wedge \dd y_{\sigma} = \sum _{\sigma \in \Sigma (1:k,1:n)}\sum _{i\notin \sigma} \left (\partial_{y_i} a_{\sigma} \right )\epsilon(i,\sigma)\dd y_{ i+ \sigma}.
$$
It can be verified that this definition of $\dd \omega$ is independent of the choice of bases.

The Hodge star for the ambient  orthonormal basis $\{\dd x_i \}_{i=1}^n$ is defined as
$$
\star \dd x_i = (-1)^{i-1} \dd x_{i^c},\quad \star \dd x_{i^c} = (-1)^{n-i} \dd x_i,
$$
which satisfy 
$$
\dd x_i\wedge \star \dd x_i = \dd x, \quad \dd x_{i^c}\wedge \star \dd x_{i^c} = \dd x,
$$
with the volume $\dd x: = \dd x_1\wedge \dd x_2\wedge \cdots\wedge \dd x_n.$ By definition, $\star\star \omega = (-1)^{n-1}\omega$ for $\omega\in \alt^1$ or $\alt^{n-1}$. 

We extend the definition to a generic coordinate and define 
$$
*\dd y_i: = (-1)^{i-1}\dd y_{i^c}\quad  \text{satisfying} \; \dd y_i\wedge *\dd y_j = \delta_{i,j}\dd y.
$$

\subsubsection{Inner product}

An intrinsic definition of an inner product on $\alt^k$ is
\begin{equation*}
\langle \omega, \eta \rangle := \sum_{\sigma \in \Sigma(1:k,1:n)}\omega(e_{\sigma(1)}, \ldots, e_{\sigma(k)})\eta(e_{\sigma(1)}, \ldots, e_{\sigma(k)}), 
\end{equation*}
where $(e_{1}, \ldots, e_{n})$ is any orthonormal basis of $\mathbb R^n$.
Then by definition
$$
\omega \wedge \star \ \eta=\langle \omega, \eta \rangle \dx, \quad\omega, \eta\in \alt^k. 
$$
Recall that $\{\dd x_i\}_{i=1}^n$ is an orthonormal basis of $\alt^1$, i.e., $\langle \dd x_i, \dd x_j \rangle = \delta_{i,j}, i,j=1,\ldots,n$. It is naturally extended to an orthonormal basis $\{ \dd x_{\sigma}, \sigma \in \Sigma(1:k,1:n)\}$ of $\alt^k$, i.e. 
\begin{equation*}
\langle \dd x_{\sigma}, \dd x_{\eta} \rangle  = \delta_{\sigma, \eta},  \quad \sigma, \eta \in \Sigma(1:k,1:n).
\end{equation*}
The duality of $\{\dd \hat{y}_i\}_{i=1}^n$ and $\{\dd y_i\}_{i=1}^n$ are also extended to the $k$-forms
\begin{equation}\label{eq:dualkform}
\langle \dd \hat{y}_{\sigma}, \dd y_{\eta} \rangle  = \delta_{\sigma, \eta},  \quad \sigma, \eta \in \Sigma(1:k,1:n).
\end{equation}
But $\{\dd y_i\}_{i=1}^n$ may not be orthogonal.

For $\omega, \eta\in \Lambda^k(\Omega)$, a further integral over the domain is included, i.e.,
$$
(\omega, \eta)_{\Omega} = \int_{\Omega}\langle \omega, \eta \rangle\dx,\quad \omega, \eta\in \Lambda^k(\Omega).
$$
For a sub-manifold $f$ of $\Omega$, the volume form $\dd x$ will induce the one for $f$ and denoted by $\dd x_f$. Define
$$
(\omega, \eta)_f = \int _{f}\langle \omega, \eta \rangle\dd x_f.
$$
For $\omega, \eta\in \Lambda ^k(\Omega)$ with expression
$$
\omega=\sum _{\sigma \in \Sigma(1:k,1:n)} a_\sigma(x) \dd x_{\, \sigma}, \;\hbox{ and }\; \eta=\sum _{\sigma \in \Sigma(1:k,1:n)} b_\sigma(x) \dd x_{\, \sigma},
$$
it is easy to prove that
\begin{equation}\label{eq:innerproduct}
(\omega, \eta)_{\Omega}  
=
\sum _{\sigma \in \Sigma(1:k,1:n)}\int _{\Omega}a_{\sigma}(x)b_{\sigma}(x)\dd x,
\end{equation}
and \eqref{eq:innerproduct} is invariant when changing to another orthonormal basis. 
For non-orthonormal basis, transformation will enter the formulae of the inner product.
%

Denote by $\mathbb P_r\Lambda^k(\Omega)$ the space with polynomial coefficients, and $L^2\Lambda^k(\Omega)$ is the space with square-integrable coefficient functions. The space $H\Lambda^k(\Omega) := \{\omega \in L^2\Lambda^k(\Omega): \dd \omega \in L^2\Lambda^{k+1}(\Omega)\}.$ When $k = n-1$, $H\Lambda^{n-1}(\Omega)$ is isomorphism to $H(\div, \Omega):=\{\bs v\in L^2(\Omega;\mathbb R^n): \div \bs v\in L^2(\Omega)\}.$ 

\subsubsection{Proxy vectors of differential forms}
Representation \eqref{eq:representation} enables us to identify a differential form with a vector function:
$$
\omega \leftrightarrow (a_{\sigma})_{\sigma \in \Sigma(1:k,1:n)},
$$ 
and $(a_{\sigma})$ is called a vector proxy of $\omega$. Be aware that, by definition, the differential form is coordinate independent while a vector proxy depends on the coordinate. 

For a $1$-form $\omega = \sum_{i=1}^n u_i \dd x_i \in \Lambda^1$, define $$\prox_1(\omega) = \bs u = (u_1, u_2, \ldots, u_n)^{\intercal}.$$ 
For a vector $\bs t = (t_1, t_2,\ldots, t_n)^{\intercal}$ representing the tangent vector $ \sum_{i=1}^n t_i\partial_{x_i}$, the action is
$$
\omega (\bs t) = \sum_{i,j=1}^n u_it_j \langle \dd x_i, \partial_{x_j} \rangle = \bs u\cdot \bs t. 
$$
For $\omega\in\Lambda^{n-1}$,
we can write 
$
\omega = \sum_{i=1}^n u_i \star \dd x_i, 
$ 
which induces an isomorphism 
$$
\prox_{n-1}: \omega \to \bs u = (u_1, u_2, \ldots, u_n)^{\intercal}.
$$
By definition, we have
\begin{align*}
\prox_{1} (\star \ \omega) =&\, (-1)^{n-1} \prox_{n-1}(\omega), &  &\omega\in \Lambda^{n-1},&\\
\prox_{n-1} (\star \ \omega) =&\, \prox_{1}(\omega),&  &\omega\in \Lambda^{1}. &
\end{align*}

Notice that the proxy vectors are defined using an orthonormal basis. Using the proxy vectors, we can change the wedge product to the inner product of vectors
\begin{align*}
\langle \omega, \eta \rangle =&\  (\prox_{1} \omega, \prox_{n-1} (\star\ \eta)), \quad \omega, \eta\in \Lambda^1,\\
\langle \omega, \eta \rangle = &\ (\prox_{k} \omega, \prox_{k} \eta), \ \quad \qquad \omega\in \Lambda^k, \eta\in \Lambda^{k}, k=1 \text{ or } n-1,\\
\omega \wedge \eta = &\ (\prox_{1} \omega, \prox_{n-1} \eta) \dd x, \quad \omega\in \Lambda^1, \eta\in \Lambda^{n-1}.
\end{align*}
For $\omega\in \Lambda^{n-1}$, the representation of $\dd \omega $ using the proxy vector is $\div \bs u$, i.e., 
$$
\dd \omega =  (\div\bs u) \dd x, \quad \text{ with } \bs u = \prox_{n-1}\omega.
$$

The $\prox$ operator is a bijection. More precisely, given a vector $\bs u = (u_1,\ldots, u_n)^{\intercal}$ represented in the ambient coordinate, let $\omega = \sum_{i=1}^n u_i\dd x_i\in \Lambda^1$ and $\star \ \omega = \sum_{i=1}^n u_i \star \dd x_{i}\in \Lambda^{n-1}$. Then $\prox_1(\omega) = \prox_{n-1}(\star \ \omega) = \bs u$. 
To resemble the notation of differential forms, we introduce notation
\begin{align*}
\dd \bs u := \prox_1^{-1}(\bs u) =  \sum_{i=1}^n u_i \dd x_i , \qquad 
 \star \, \dd \bs u := \prox_{n-1}^{-1}(\bs u) = \sum_{i=1}^n u_i \star \dd x_{i}.
\end{align*}
Here in $\dd \bs u$, $\dd$ is  understood as a dual operator mapping a tangent vector $\bs u$ to a co-tangent vector $\dd \bs u\in \Lambda^1$, and the symbol $\dd$ is not associated to any differentiation. A textbook notation of $\dd \bs u$ is  ${}^{\flat _1}\bs u$ and $\star \dd \bs u$ is ${}^{\flat _{n-1}}\bs u$.

Denote the proxy vector of $\{\dd y_i\}_{i=1}^n$ by $\{\bs v_i\}_{i=1}^n$ and $\{\dd\hat{y}_i\}_{i=1}^n$ by $\{\hat{\bs v}_i\}_{i=1}^n$. Then $\{\hat{\bs v}_i\}_{i=1}^n$ is dual to  $\{\bs v_i\}_{i=1}^n$ in the sense that $(\hat{\bs v}_i, \bs v_j) = \delta_{i,j}$. Treat $\bs v_i$ as a column vector and form the matrix $V = (\bs v_1, \ldots, \bs v_n)$ and $\hat V = (\hat{\bs v}_1, \ldots, \hat{\bs v}_n)$. The gram matrix is $M = V^{\intercal}V$. Then we have the relation $\hat{V} = V M^{-1}$. 

From 
\begin{align*}
(\bs v_i, \prox_{n-1}(*\dd y_j) ) \dd x = \dd y_i\wedge *\dd y_j = \delta_{i,j}\dd y = \delta_{i,j} \det(V) \dd x,\\
(\hat{\bs v}_i, \prox_{n-1}(*\dd\hat{y}_j) ) \dd x = \dd\hat{y}_i\wedge *\dd\hat{y}_j = \delta_{i,j}\dd\hat{y} = \delta_{i,j} \det(\hat V) \dd x,
\end{align*}
we get 
\begin{align*}
\prox_{n-1}(* \dd y_i) = \det(V) \prox_{1}(\dd \hat{y}_i) = \det(V)\,\hat{\bs v}_i, \\
\prox_{n-1}(* \dd\hat{y}_i) = \det(\hat{V}) \prox_{1}(\dd y_i)  = \det(\hat{V}) \, \bs v_i,
\end{align*}
where the scaling $\det(V)$ or $ \det(\hat{V})$ is due to the non-orthogonality.
In view of proxy vectors, $*$ is like a kind of dual operator mapping a vector to its dual vector. 
%
%
%

\subsubsection{Trace operator}
For $F\in\Delta_{n-1}(T)$, let the trace operator $\tr_F^{\div}:\Lambda^{n-1}(T) \to\Lambda^{n-1}(F)$ be the pullback of the inclusion $F\hookrightarrow T$.
That is for any tangent vectors $v_1, \ldots, v_{n-1}$ of $F$, we also treat them as tangent vectors of $T$ and define
$$
{\rm tr}_F^{\div}\omega(v_1, \cdots, v_{n-1}) :=\omega(v_1, \cdots, v_{n-1}),\quad \omega\in\Lambda^{n-1}(T).
$$
We denote by $\tr^{\div}: \Lambda^{n-1}(T) \to \cup_{F\in \partial T}\Lambda^{n-1}(F)$ as $\tr^{\div}\omega|_F = {\rm tr}_F^{\div}\omega$. 

Let $\bs n_F$ be the normal vector of $F$ so that the orientation of $F$, which is given by the volume $\dd x_F\in \Lambda^{n-1}(F)$, and $\bs n_F$ form a consistent orientation of the ambient orthonormal basis. Then 
\begin{equation}\label{eq:trF}
{\rm tr}_F^{\div} \omega = \bs u\cdot \bs n_F\dd x_F  \quad \text{ with } \bs u = \prox_{n-1}\omega.
\end{equation}
On the other hand, for any $p\in L^2(F)$, we have
$$
(\omega, p \star \dd \bs n_F)_F = \int_F \bs u\cdot \bs n_F p \dd x_F = \int_{F} p\  {\rm tr}_F^{\div}  \omega.
$$

Based on~\eqref{eq:trF}, we can discuss the trace operator in the more familiar vector function setting. The trace operator for space $H(\div, T)$  
$$
{\rm tr}^{\div}: H(\div, T) \to H^{-1/2}(\partial T)
$$
is a continuous extension of ${\rm tr}^{\div} \boldsymbol u = \boldsymbol u\cdot \boldsymbol n|_{\partial T}$ defined on smooth functions. 

\subsubsection{Differential forms in the barycentric coordinates}
As $\sum_{i=0}^{n}\lambda_i = 1$, $\sum_{i=0}^{n}\dd \lambda_i = 0$ and $\{\dd \lambda_0, \ldots, \dd \lambda_n\}$ is not a basis of $\alt^1$. Set a vertex as the origin, without loss of generality, say $\texttt{v}_0$, then $\{\dd \lambda_1, \ldots, \dd \lambda_n\}$ forms a basis of $\alt^1$. In general, through the index $\sigma$, there is one-to-one correspondence between $\Delta_{k-1}(F_{0})$ and $\alt^k(T)$. Namely for $\sigma \in \Sigma(1:k,1:n)$, $f_{\sigma}$ is a $(k-1)$-dimensional simplex in $\Delta_{k-1}(F_{0})$ with vertices $\{\sigma(1), \ldots, \sigma(k)\}$ and $\dd \lambda_{\sigma} = \dd \lambda_{\sigma(1)}\wedge \cdots \wedge \dd \lambda_{\sigma(k)} \in \alt^k$. We can also write as $\dd \lambda_f$ assuming the index of the vertices of $f$ is sorted in the ascending order. 

The $1$-form $\dd \lambda_i$ has a vector representation $\nabla \lambda_i$, which is a scaled normal vector $\bs n_{F_i}$ of face $F_{i}$. For a simplex $f\in \Delta_{\ell}(T)$, $\{\nabla \lambda_i, i\in  f^{*} \}$ are $n- \ell$ normal vectors of $f$ and can span the normal plane of $f$.
The vector representations of $(n-1)$-forms, for $i=1, \ldots, \ell$,
$$
\dd \lambda_{[\sigma(0),\sigma(i)]^*} = 
\dd \lambda_{[0:n]-\sigma(0) - \sigma(i)}:=\dd\lambda_{0}\wedge\cdots\wedge\widehat{\dd\lambda}_{\sigma(0)}\wedge\cdots\wedge\widehat{\dd\lambda}_{\sigma(i)}\wedge\cdots\wedge\dd\lambda_{n},$$
are scaling of tangential vectors $\bs t_{\sigma(0)\sigma(i)}$ of $f_{\sigma}$ and can span the tangent plane of $f$. This is illustrated in Fig.~\ref{fig:tangentnormalvectors}. 

\begin{figure}[htbp]
\subfigure[Face $f_{123}$.]{
\begin{minipage}[t]{0.5\linewidth}
\centering
\includegraphics*[width=2.5in]{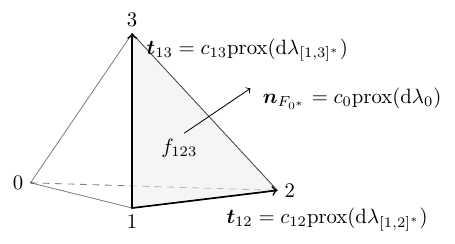}
\end{minipage}}
\subfigure[Edge $f_{13}$.]
{\begin{minipage}[t]{0.5\linewidth}
\centering
\includegraphics*[width=2.25in]{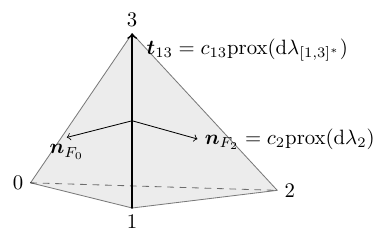}
\end{minipage}}
\caption{Sub-simplices and their tangential and normal vectors.}
\label{fig:tangentnormalvectors}
\end{figure}

\subsection{An algebraic operator}
Define $\alt^{k, i}(\mathbb R^n)=\alt^{k}(\mathbb R^n) \otimes \alt^{i}(\mathbb R^n)$ for $k, i =0,1,\ldots, n$. 
In particular $\alt^{k,n-1}(\mathbb R^n)=\alt^{k} (\mathbb R^n)\otimes \alt^{n-1}(\mathbb R^n) \cong \alt^{k}(\mathbb R^n) \otimes \mathbb R^n$. To simplify the notation, $\alt^{k, i}(\mathbb R^n)$ is abbreviated  as $\alt^{k, i}$. 
In~\cite{arnoldComplexesComplexes2021}, the algebraic operator $s^{k, n-1}: \alt^{k,n-1} \rightarrow \alt^{k-1, n}$ is defined as
\begin{align*}
s^{k, n-1} \omega \left(w_{1}, \cdots, w_{k-1}\right) \left(v_{1}, \cdots, v_{n}\right) 
& := \sum_{i=1}^{n} (-1)^{i-1} \omega \left(v_{i}, w_{1}, \cdots, w_{k-1}\right) \left(v_{1}, \cdots, \widehat{v}_{i}, \cdots, v_{n}\right) \\
& \forall~v_{1}, \cdots, v_{n}, w_{1}, \cdots, w_{k-1} \in \mathbb{R}^{n}.
\end{align*}

Recall that we have reserved $\{ \dd x_i \}$ for a fixed orthonormal basis of $\alt^1(\mathbb R^n)$. We are going to derive more concrete forms of operator $s^{k, n-1}$ in a generic basis $\{ \dd y_i \}$, which may not be orthonormal. 
We expand $\omega\in  \alt^{k,n-1} $ in this basis as
\begin{equation*}
\omega = \sum_{i=1}^n \sum_{\sigma \in \Sigma(1:k,1:n)}a_{\sigma, i} \dd y_{\sigma} \otimes *\dd y_i. 
\end{equation*}
An element in $\alt^{k,n-1}$ can be identified as a matrix $\bs A = (a_{\sigma, i})$ of size $\displaystyle{n \choose k}\times n$ indexed by $(\sigma,i)$ for $i=1,2,\ldots, n$ and $\sigma \in \Sigma(1:k,1:n)$.

\begin{lemma}\label{lem:sexpress}
For $\omega = \sum_{i=1}^n \sum_{\sigma \in \Sigma(1:k,1:n)}a_{\sigma, i} \dd y_{\sigma} \otimes *\dd y_{i},$ we have
\begin{equation*}
s^{k, n-1} \omega = \sum_{\tau \in \Sigma(1:k-1,1:n)} \Big (\sum_{i \in   \tau^c } \epsilon(i, \tau) a_{i + \tau, i} \Big )\dd y_{\tau}\otimes\dd y.
\end{equation*}
\end{lemma}
\begin{proof}
Let $\{\partial y_i\}$ be the basis of the tangent space dual to $\{\dd y_i\}$. That is $\dd y_i(\partial y_j) = \delta_{i,j}$ for $1\leq i,j\leq n$. Given $\sigma\in \Sigma(1:k,1:n)$, we use the notation $\partial y_{\sigma}$ to denote $k$ vectors $(\partial y_{\sigma(1)},\ldots, \partial y_{\sigma(k)})$. Then $\dd y_{\sigma} (\partial y_{\sigma'}) = \delta_{\sigma, \sigma'}$ for $\sigma, \sigma' \in \Sigma(1:k,1:n)$. To get the coefficient of $s^{k, n-1}\omega$ for the component $\dd y_{\tau}\otimes \dd y, \tau \in \Sigma(1:k-1,1:n)$, we check the action 
\begin{align*}
s^{k, n-1} &\omega \left(\partial y_{\tau(1)},\cdots, \partial y_{\tau(k-1)} \right) \left(\partial y_{1}, \cdots, \partial y_{n}\right) \\
&=\sum_{i=1}^{n}(-1)^{i-1} \omega \left(\partial y_{i}, \partial y_{\tau(1)},\cdots, \partial y_{\tau(k-1)}\right) \big(\partial y_{1}, \cdots, \widehat{\partial y}_{i}, \cdots, \partial y_{n}\big )\\
&=\sum_{i \in   \tau^c } \omega \left( \epsilon(i, \tau) \partial y_{i + \tau}\right)\left( (-1)^{i-1}\partial y_{i^c}\right)\\
& =
\sum_{i \in   \tau^c } \epsilon(i, \tau) a_{i + \tau, i} .
\end{align*}
If $i\in  \tau $, then vectors $\partial y_{i}, \partial y_{\tau(1)},\cdots, \partial y_{\tau(k-1)}$ are linearly dependent and thus the term vanishes. So only $i \in   \tau^c $ are left in the summation. 
\end{proof}

For a given $\tau\in \Sigma(1:k-1,1:n)$, we call the sequence of index $\{( i_m + \tau, i_m), i_m \in  \tau^c, m=1,2,\ldots, n-k+1 \}$, the {\em constraint sequence} of $\tau$, which can be also written as $\{(\sigma_{i_m}, i_m), m=1,2,\ldots, n-k+1, \sigma_{i_m} = i_m + \tau\}$. The length of the constraint sequence is $|\tau^c| = n-k+1$. Without loss of generality, we can sort as $i_1 < i_2 < \ldots <i_{n-k+1}$. The first one $(i_1 + \tau, i_1)$ will be called the pair index of the constraint sequence. 

\begin{figure}[htbp]
\subfigure[$s^{1,4}$ for $\tau = \varnothing$]{
\begin{minipage}[t]{0.3\linewidth}
\centering
\includegraphics*[width=3.15cm]{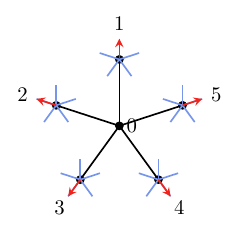}
\end{minipage}}
\subfigure[$s^{3,4}$ for $\tau=(1,2)$]{
\begin{minipage}[t]{0.3\linewidth}
\centering
\includegraphics*[width=3.25cm]{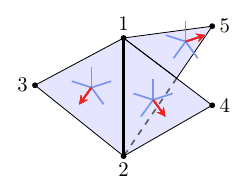}
\end{minipage}}
\subfigure[$s^{4,4}$ for $\tau=(1,2,3)$]
{\begin{minipage}[t]{0.35\linewidth}
\centering
\includegraphics*[width=3cm]{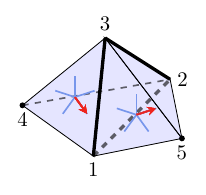}
\end{minipage}}
\caption{Illustration of constraint sequences and operator $s^{k, n-1}$ for $k=1,3$ and $k=n-1$ for a simplex in $\mathbb R^5$. For a $(k-2)$-dimensional sub-simplex $\tau$, $\sigma = i+\tau$ with $i\in \tau^{c}$ are all $(k-1)$-dimensional sub-simplex containing $\tau$. For each $\sigma$, we attach a vector of length $n$ and consider its $i$-th component, which is a representation of  $\dd y_{\sigma} \otimes *\dd y_i$. The constraint sequence of $\tau$ will be formed by all $(\sigma, i)$ surrounding $\tau$.}
\label{fig:s}
\end{figure}

We provide some visualization of the constraint sequence. The tensor product $\dd y_{\sigma} \otimes *\dd y_i$ can be visualized as follows: for each sub-simplex $f_{\sigma}$, we attach a vector of length $n$. The sub-index $i$ in $*\dd y_i$ corresponds to the $i$-th component of this vector. See $5$-edge stars in Fig.~\ref{fig:s}. We can associate the $(k-1)$-form $\dd y_{\tau}$ with the sub-simplex $f_{\tau}$ of dimension $k-2$, then $\{f_{i + \tau}\}_{i \in   \tau^c }$ corresponds to all $(k-1)$-dimensional sub-simplices of $F_0$ (excluding $f_{0+\tau}$ as index $0$ is not used in differential forms) using $\tau$ as a boundary face. See Fig.~\ref{fig:s}.

If we identify entries of the matrix proxy as nodes of a graph, a constraint sequence will define a path of nodes. See Fig.~\ref{fig:graph}. Indices in different constraint sequences are different. Namely for $\tau \neq \tau'$, $(i + \tau, i)\neq (j+ \tau', j)$ as either $i\neq j$ or $i+\tau \neq j+\tau'$. On the graph, different constraint sequences will correspond to disjoint paths. 

Since only the value $a_{\sigma, i}$ on the constraint sequence will contribute to the image $s^{k,n-1}\omega$, we conclude that
\begin{equation*}
s^{k,n-1}(\dd y_{\sigma} \otimes *\dd y_{i}) = 0 \iff i\in \sigma^c.
\end{equation*} 
For each row, i.e., for a fixed $\sigma$, there are $k$ entries $(\sigma, i), i\in \sigma$ on $k$ different constraint sequences and the rest $n-k$ entries are not in any constraint sequence.

\begin{figure}[htbp]
\begin{center}
\includegraphics[width=7.5cm]{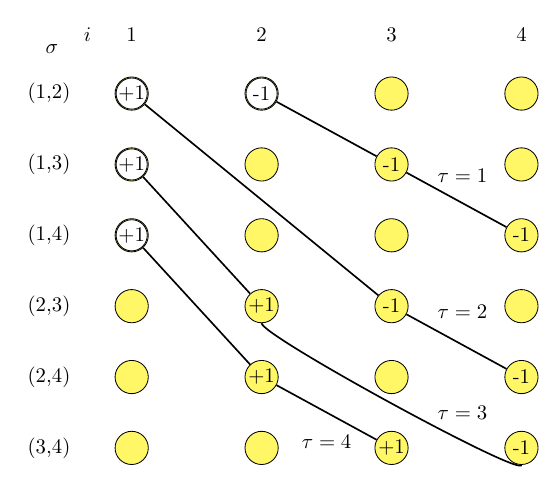}
\caption{Constraint sequences for $n=4, k=2$ with sign $\epsilon(i,\tau)$. A constraint sequence will define a path of nodes. Different constraint sequences will correspond to disjoint paths. The white circle denotes the pair index of each constraint sequence which is a non-free index. Other circles in yellow are free indices. For each row, there are $k$ entries $(\sigma, i), i\in \sigma$ on $k$ different constraint sequences and the rest $n-k$ entries are not in any constraint sequence}.  
\label{fig:graph}
\end{center}
\end{figure}

We can identify $\alt^{k-1, n}$ as a vector in $\mathbb R^{\dim \alt^{k-1} }$.
With the matrix and vector representations, the $s^{k, n-1}$ operator induces an operator from  matrix $\bs A$ to a vector in $\mathbb R^{\dim \alt^{k-1} }$ and will be still denoted by $s^{k, n-1}$. We collect the coefficients of a constraint sequence and denote by $\bs a_{\tau} = (a_{i_m+\tau, i_m})_{m=1,2,\ldots,n-k+1}$. Let $\bs \epsilon_{\tau} = (\epsilon(i_m,\tau))_{m=1,2,\ldots,n-k+1}$ be the corresponding sign vector.  We can write the operator as
\begin{equation*}
(s^{k, n-1} \bs A)_{\tau} =  \bs a_{\tau} \cdot \bs \epsilon_{\tau}, \quad \tau\in \Sigma(1:k-1,1:n).
\end{equation*}
That is the action is along each constraint sequence. 

\begin{lemma}
For $k=1,\ldots, n-1$, the operator $s^{k, n-1}: \alt^{k,n-1} \rightarrow \alt^{k-1, n}$ is onto. And $s^{n, n-1}: \alt^{n, n-1} \rightarrow \alt^{n-1, n}$ is a bijection and its proxy $s^{n,n-1}$ is the transpose operator. 
\end{lemma}
\begin{proof}
First consider $k=1,\ldots, n-1$. By the linearity, it suffices to prove for $\tau=[1,\cdots, k-1]$, there exists $\omega = \sum_{i=1}^n \sum_{\sigma \in \Sigma(1:k,1:n)}a_{\sigma, i} \dd y_{\sigma} \otimes *\dd y_{i} $ such that $s^{k, n-1} \omega =\dd y_{\tau}\otimes\dd y$. 

For a given $\tau$, we just pick up one $\sigma= i + \tau$ from its constraint sequence and set the coefficient be $\epsilon(i, \tau)$. More precisely, take $a_{[1,\cdots, k], k}=\epsilon(k, [1:k-1])=(-1)^{k-1}$, and $a_{\sigma, i}=0$ for the rest. Then
$$
s^{k, n-1} \omega = \sum_{\tilde\tau \in \Sigma(1:k-1,1:n)} \left (\sum_{i \in   \tilde\tau^c } \epsilon(i, \tilde\tau) a_{i+ \tilde\tau, i} \right )\dd y_{\tilde\tau}\otimes\dd y=\dd y_{\tau}\otimes\dd y.
$$ 

Next consider $k=n$. For $\omega = \sum_{i=1}^n a_i \dd y \otimes *\dd y_{i},$ we have
\begin{align*}
s^{n,n-1} \omega &= \sum_{\tau \in \Sigma(1:n-1,1:n)} \left (\sum_{i \in   \tau^c } \epsilon(i, \tau)a_i \right )\dd y_{\tau}\otimes\dd y = \sum_{i=1}^n a_i *\dd y_{i} \otimes \dd y.
\end{align*}
Namely $s^{n,n-1}$ maps the row vector $(a_1, \ldots, a_n)$ to a column vector $(a_1, \ldots, a_n)^{\intercal}$. 
\end{proof}
If we identify the row and column vector by the transpose, we can also say $s^{n,n-1}$ is the identity operator.

\subsection{Constraint tensor spaces}
Now we are ready to introduce the tensor space 
$$
\mathbb X := \ker(s^{k, n-1})\cap \alt^{k,n-1}=\{ \omega \in \alt^{k,n-1} \mid s^{k, n-1}\omega = 0 \}, \quad 1\leq k\leq n-1.
$$
When $k=n$, as $s^{n,n-1}$ is bijection, $\mathbb X = \{ 0 \}$ is trivial. So throughout the rest of the paper, we will consider the non-trivial case $1\leq k\leq n-1$.
For a given basis $\{\dd y_i \}$, it will be more convenient to work on the matrix representation
\begin{align*}
\mathbb X := \bigg\{ \omega =  \sum_{i=1}^n \sum_{\sigma \in \Sigma(1:k,1:n)}& a_{\sigma, i} \dd y_{\sigma} \otimes *\dd y_{i} \mid 
\bs A = (a_{\sigma, i}) \in \mathbb R^{{n \choose k}\times n } : \\
&\sum_{i \in   \tau^c } \epsilon(i, \tau) a_{i + \tau, i} = 0, \quad \forall~\tau \in \Sigma(1:k-1,1:n)\bigg\}.
\end{align*}
As $s^{k, n-1}$ is surjective, 
\begin{equation}\label{eq:dimX}
\dim \mathbb X = \dim \alt^{k,n-1} - \dim \alt^{k-1,n} = n {n \choose k} - {n \choose k-1}=(n-k){n+1 \choose k}.
\end{equation}

For the orthonormal basis $\{ \dd x_i\}$, we introduce 
$$
H(\div, \Omega; \mathbb X): =\Big\{ \bs A = (a_{\sigma, i}) \in L^2(\Omega,\mathbb X): \div\bs A\in L^2\Big(\Omega,\mathbb R^{{n \choose k}}\Big)\Big\} 
$$
with $\div\bs A:=\left(\sum_{i=1}^n\partial_{x_i}a_{\sigma, i}\right)$, i.e., the divergence operator is applied row-wise. Its differential form version is
$$
H(\dd_{n-1},\Omega;\mathbb X): =\Big\{ \omega \in L^2\Lambda^{k,n-1}(\Omega): s^{k, n-1}\omega = 0, \dd_{n-1}\omega \in L^2\Lambda^{k,n}(\Omega)\Big\},
$$
where the exterior derivative $\dd_{n-1}$ is applied to the component $\Lambda^{n-1}$ in $\Lambda^{k,n-1}$. 
In view of the matrix proxy $\bs A$  in the orthonormal basis $\{\dd x_i\}$, the trace on face $F$
$$
{\rm tr}_F^{\div} \bs A =  \bs A\bs n_F 
$$
is a column vector of length $\displaystyle{n \choose k}$, and should be continuous on the $(n-1)$-dimensional faces across simplices. 

\begin{example}\label{exm:knm1}\rm
Consider $k = n-1$. For $\omega = \sum_{i=1}^n\sum_{j=1}^n a_{i, j} *\dd y_{i} \otimes *\dd y_{j},$ we have
\begin{align*}
s^{n-1,n-1} \omega &= \sum_{\tau \in \Sigma(1:n-2,1:n)} \left (\sum_{i \in   \tau^c } (-1)^{([1: n] \backslash [i + \tau]) - 1}\epsilon(i, \tau) a_{[1: n] \backslash [i + \tau], i} \right )\dd y_{\tau}\otimes\dd y \\
&= \sum_{1\leq i<j\leq n} (-1)^{i+j}\left (a_{j, i}-a_{i, j} \right )\dd y_{(i,j)^c}\otimes\dd y,    
\end{align*}
where $\dd y_{(i,j)^c}:=\dd y_{1}\wedge\cdots\wedge\widehat{\dd y}_{i}\wedge\cdots\wedge\widehat{\dd y}_{j}\wedge\cdots\wedge\dd y_{n}$.
In terms of the matrix proxy, it holds
$$
s^{n-1,n-1} (\bs A) = 2{\rm vskw}(\bs A)\quad \textrm{ with }\;\; \bs A=\left(a_{i, j}\right)_{n\times n}\in\mathbb R^{n\times n},
$$
where operator ${\rm vskw}: \mathbb R^{n\times n}\to \mathbb R^{n(n-1)/2}$ is defined by
$$
\big({\rm vskw}(\bs A)\big)_{\sigma}=\frac{1}{2}(-1)^{i+j}\left (a_{j, i}-a_{i, j} \right ) \quad\textrm{ with }\;\; \sigma=[i, j]\in\Sigma(1:2,1:n).
$$
Thus $\mathbb X = \mathbb S$ consists of all symmetric matrices.
\end{example}

\begin{example}\label{exm:k1} \rm
Consider $k=1$. For $\omega = \sum_{i=1}^n\sum_{j=1}^n a_{j, i} \dd y_j \otimes *\dd y_{i},$ we have
$$
s^{1,n-1} \omega = \sum_{i=1}^n a_{i,i}\dd y.
$$ 
In terms of the matrix proxy, 
$$
s^{1,n-1} (\bs A) = {\rm trace}(\bs A).
$$
Thus $\mathbb X = \mathbb T$ is the traceless matrix space.
\end{example}

\subsection{Bases of the constraint tensor space}
Recall that
$$
\alt^{k,n-1} = {\rm span} \{\dd y_{\sigma}\otimes *\dd y_{i}, \sigma \in \Sigma(1:k,1:n), i=1,\ldots, n\}, \quad 1\leq k\leq n.
$$
We shall modify the basis function $\dd y_{\sigma}\otimes *\dd y_{i}$ to get a basis of $\mathbb X = \alt^{k,n-1}\cap \ker(s^{k,n-1})$.

Recall that we consider the non-trivial case: $1\leq k\leq n-1$ so that the length of the constraint sequence $n-k+1$ is greater than or equal to $2$. We shall define an oblique (non-orthogonal) projection operator $\pi_{\mathbb X}$ applied to $\dd y_{\sigma}\otimes *\dd y_{i}$ for index pair $(\sigma, i)$. 

When $i\in  \sigma^c $, $s^{k, n-1}(\dd y_{\sigma}\otimes *\dd y_{i} ) = 0$, and thus we keep it unchanged, i.e., $$\pi_{\mathbb X} ( \dd y_{\sigma}\otimes *\dd y_{i}) = \dd y_{\sigma}\otimes *\dd y_{i}, \quad i\in \sigma^c.$$
For each constraint sequence $(\sigma_{i_m}, i_m)_{m=1,\ldots, n-k+1}$, set $(\sigma_{i_1}, i_1)$ as the pair index and modify basis functions to, for $m=1,\ldots, n-k+1$,
\begin{equation*}
\pi_{\mathbb X}(\dd y_{\sigma_{i_m}}\otimes *\dd y_{i_m}):= \dd y_{\sigma_{i_m}}\otimes *\dd y_{i_m} -  \epsilon(i_m, \tau)  \epsilon(i_1, \tau) \dd y_{\sigma_{i_1}}\otimes*\dd y_{i_1}.
\end{equation*}
In terms of the coefficient vector,  $\pi_{\mathbb X}$ will map the vector $\bs a_{\tau} = (0,\ldots, 1, \ldots, 0)$ to the vector 
$$
\tilde{\bs a}_{\tau} = ( -  \epsilon(i_m, \tau)  \epsilon(i_1, \tau), 0, \ldots, 1, \ldots, 0),
$$ so that the constraint $\tilde{\bs a}_{\tau}\cdot \bs \epsilon_{\tau} = 0$ is satisfied. By the linear combination, we get the mapping $\pi_{\mathbb X}: \alt^{k, n-1} \to \mathbb X$. 

An index $(\sigma, i)$ will be called a {\em free index} if $\pi_{\mathbb X} ( \dd y_{\sigma}\otimes *\dd y_{i})\neq 0$. By definition, only $\pi_{\mathbb X}(\dd y_{\sigma_{i_1}}\otimes*\dd y_{i_1})=0$, i.e., only the pair index of each constraint sequence is not free. Therefore the number of basis functions is reduced by one for each constraint sequence. In total, we remove ${n \choose k-1}$ basis functions of $\alt^{k,n-1}$ and obtain a basis of $\mathbb X$:
\begin{equation*}
\mathbb X = {\rm span} \{  \pi_{\mathbb X} ( \dd y_{\sigma}\otimes *\dd y_{i}), (\sigma, i) \text{ is free } \}.
\end{equation*}
In Section~\ref{sec:geodecompdivtensor}, we will use $\pi_{\mathbb X}$ to define a $t$-$n$ decomposition of $\mathbb X$ and construct finite element subspaces of $H(\div,\Omega;\mathbb X)$.

Next we will present intrinsic bases of $\mathbb X$ using the barycentric coordinate.

\begin{lemma}
For any $\sigma\in\Sigma(0:n-k, 0:n)$ and $i=1,\ldots, n-k$, it holds
$$
\dd\lambda_{\sigma^{*}}\otimes\dd\lambda_{[\sigma(0), \sigma(i)]^*}\in\mathbb X,
$$
where 
$$
\dd\lambda_{[\sigma(0), \sigma(i)]^*} = \dd \lambda_{([0:n]-\sigma(0))- \sigma(i)} = \dd\lambda_{0}\wedge\cdots\wedge\widehat{\dd\lambda}_{\sigma(0)}\wedge\cdots\wedge\widehat{\dd\lambda}_{\sigma(i)}\wedge\cdots\wedge\dd\lambda_{n}. 
$$
\end{lemma}
\begin{proof}
We treat $\sigma(0)$ as the origin. 
Let $y_i=\lambda_{\sigma(i)}$ for $i=1,\ldots, n-k$, and $y_{n-k+i}=\lambda_{\sigma^*(i)}$ for $i=1,\ldots, k$. Then $\dd\lambda_{\sigma^{*}}\otimes\dd\lambda_{[\sigma(0), \sigma(i)]^*}= (-1)^{i-1} (\dd y_{n-k+1}\wedge\cdots\wedge\dd y_{n})\otimes*\dd y_{i} \in\mathbb X$ for $i=1,\ldots, n-k$ as $i\in [n-k+1, \ldots, n]^c$.
\end{proof}

The vector proxy of $\dd\lambda_{[\sigma(0), \sigma(i)]^*}$ is a scaling of the edge vector $\bs t_{\sigma(0)\sigma(i)}$, which is on the tangent plane of $f_{\sigma}\in \Delta_{n-k}(T)$. The $k$-form $\dd \lambda_{\sigma^*}$ is the volume of the normal plane of $f_{\sigma}$, i.e. $\mathscr N^{f_{\sigma}}$ of dimension $k$. Their tensor product for all $f\in \Delta_{n-k}(T)$ forms a basis of $\mathbb X$. 

The inner product of $k$-form $\langle \omega, \eta \rangle$ can be extended to $\alt^{k,n-1}$ by the tensor product. Define $P_{\mathbb X}$ as the orthogonal projection from $\alt^{k,n-1}$ to $\mathbb X$ w.r.t. the inner product $\langle\cdot, \cdot \rangle$. 

\begin{lemma}\label{lem:Xdualbasis}
The set $$\big\{ \dd\lambda_{\sigma^{*}}\otimes\dd\lambda_{[\sigma(0), \sigma(i)]^*} \big\}_{i=1,\ldots, n-k}^{\sigma\in\Sigma(0:n-k, 0:n)}$$ in $\mathbb X$ is dual to the set
$$
\big\{P_{\mathbb X}\big(\star\dd\lambda_{\sigma-\sigma(i)}\otimes\star\dd\lambda_{\sigma(i)}\big) \big\}_{i=1,\ldots, n-k}^{\sigma\in\Sigma(0:n-k, 0:n)}$$
in the sense that: for any $\sigma, \tau\in\Sigma(0:n-k, 0:n)$, and $i,j=1,\ldots, n-k$, 
$$
\left \langle \dd\lambda_{\sigma^{*}}\otimes\dd\lambda_{[\sigma(0), \sigma(i)]^*}, P_{\mathbb X}\big(\star\dd\lambda_{\tau-\tau(i)}\otimes\star\dd\lambda_{\tau(i)}\big) \right \rangle = \delta_{\sigma,\tau}\delta_{i,j}.
$$    
\end{lemma}
\begin{proof}
Since $\dd\lambda_{\sigma^{*}}\otimes\dd\lambda_{[\sigma(0), \sigma(i)]^*}\in\mathbb X$, by definition of $P_{\mathbb X}$ and the inner product $\langle \cdot,\cdot \rangle$, it suffices to prove that: for any $\sigma, \tau\in\Sigma(0:n-k, 0:n)$, and $i,j=1,\ldots, n-k$, 
$$
(\dd\lambda_{\sigma^{*}}\otimes\dd\lambda_{[\sigma(0), \sigma(i)]^*})\wedge (\dd\lambda_{\tau-\tau(j)}\otimes\dd\lambda_{\tau(j)})\neq 0
$$
if and only if 
$$
\sigma=\tau \textrm{ and } i=j.
$$
By definition,
\begin{align*}
&\quad(\dd\lambda_{\sigma^{*}}\otimes\dd\lambda_{[\sigma(0), \sigma(i)]^*})\wedge (\dd\lambda_{\tau-\tau(j)}\otimes\dd\lambda_{\tau(j)})=(\dd\lambda_{\sigma^{*}}\wedge\dd\lambda_{\tau-\tau(j)})\otimes(\dd\lambda_{[\sigma(0), \sigma(i)]^*}\wedge\dd\lambda_{\tau(j)}).
\end{align*}
Then 
$(\dd\lambda_{\sigma^{*}}\otimes\dd\lambda_{[\sigma(0), \sigma(i)]^*})\wedge (\dd\lambda_{\tau-\tau(j)}\otimes\dd\lambda_{\tau(j)})\neq0$ is equivalent to
$$
\tau(j)\in \{\sigma(0), \sigma(i)\} \quad\textrm{ and }\quad \tau-\tau(j)\subset \sigma.
$$
This indicates $\tau\subseteq \sigma$. We finish the proof by the fact $\tau$ and $\sigma$ have the same length.
\end{proof}

We are in the position to present intrinsic bases of $\mathbb X$ using the barycentric coordinates. 
\begin{theorem}[Intrinsic bases of $\mathbb X$]\label{thm:basisofX}
The set
\begin{equation*}
\big\{\dd\lambda_{\sigma^{*}}\otimes\dd\lambda_{[\sigma(0), \sigma(i)]^*} \big\}_{i=1,\ldots, n-k}^{\sigma\in\Sigma(0:n-k, 0:n)}
\end{equation*}
is a basis of $\mathbb X$.
Its dual basis is
\begin{equation}\label{eq:dualbasis}
\big\{P_{\mathbb X}\big(\star\dd\lambda_{\sigma-\sigma(i)}\otimes\star\dd\lambda_{\sigma(i)}\big) \big\}_{i=1,\ldots, n-k}^{\sigma\in\Sigma(0:n-k, 0:n)}.
\end{equation}
\end{theorem}
\begin{proof}

The number of the set $\big\{\dd\lambda_{\sigma^{*}}\otimes\dd\lambda_{[\sigma(0), \sigma(i)]^*} \big\}_{i=1,\ldots, n-k}^{\sigma\in\Sigma(0:n-k, 0:n)}$ is $\displaystyle(n-k){n+1 \choose k}$, which equals to $\dim\mathbb X$, cf.~\eqref{eq:dimX}. Hence it suffices to prove that they are linearly independent. 
Assume there exist $c_{\sigma,i}\in\mathbb R$ for each $\sigma\in\Sigma(0:n-k, 0:n)$ and $i=1,\ldots, n-k$ such that 
$$
\sum_{\sigma\in\Sigma(0:n-k, 0:n)}\sum_{i=1}^{n-k}c_{\sigma,i}\dd\lambda_{\sigma^{*}}\otimes\dd\lambda_{[\sigma(0), \sigma(i)]^*}=0.
$$
Then apply the wedge product with $\dd\lambda_{\tau-\tau(j)}\otimes\dd\lambda_{\tau(j)}$ for $\tau\in\Sigma(0:n-k, 0:n)$ and $0\leq j\leq n-k$, due to Lemma~\ref{lem:Xdualbasis}, we obtain $c_{\tau,j} = 0$. As $(\tau,j)$ runs over the whole set $\Sigma(0:n-k, 0:n)\times\{1,\ldots, n-k\} $, we conclude all $c_{\tau,j}$ vanishes. 
\end{proof}

\begin{example}\rm
When $k=n-1$, $f_{\sigma}$ is an edge and the vector proxy of $\dd \lambda_{\sigma^*}$ is a scaling of the tangent vector $\bs t_{\sigma(0)\sigma(1)}$ of $f_{\sigma}$. A basis of $\mathbb X$ is thus given by $
\big\{\dd\lambda_{\sigma^{*}}\otimes\dd\lambda_{\sigma^{*}} \big\}_{\sigma\in\Sigma(0:1, 0:n)}$, and the dual basis is $
\big\{\sym(\star\dd\lambda_{\sigma(0)}\otimes\star\dd\lambda_{\sigma(1)}) \big\}_{\sigma\in\Sigma(0:1, 0:n)}$. Equivalently, in terms of the vector proxy, a basis of $\mathbb S$ is $
\big\{\bs t^{e}\otimes\bs t^{e} \big\}_{e\in\Delta_1(T)}
$ and the dual basis is $
\big\{\sym(\boldsymbol{n}_{F_i}\otimes\boldsymbol{n}_{F_j}) \big\}_{e={\rm Convex}(\texttt{v}_{i}, \texttt{v}_{j})\in\Delta_1(T)}$, which are crucial in designing the $H(\div;\mathbb S)$ element~\cite{Hu2015,HuZhang2015,ChenHuHuang2018} and useful in the Regge calculus~\cite{christiansenLinearizationReggeCalculus2011}. 
\end{example}

\begin{example}\rm 
When $k=1$, $f_{\sigma}$ is an $(n-1)$-dimensional face $F$ and the vector proxy of $\dd \lambda_{\sigma^*}$ is $\bs n_F$. 
In the matrix proxy, a basis of $\mathbb T$ is $
\big\{\bs n_F\otimes\bs t_{i}^{F} \big\}_{i=1,\ldots, n-1}^{F\in\Delta_{n-1}(T)}
$, which is discovered in~\cite{HuLiang2021} and presented in Lemma \ref{lm:Tbasis}.
\end{example}

\subsection{Formulae on the projections}
We will present an explicit formula on $P_{\mathbb X}$. Recall that the basis $\{\dd \hat{y}_i\}_{i=1}^n$ is dual to $\{\dd y_i\}_{i=1}^n$  in the sense that $\langle \dd \hat{y}_i, \dd y_j\rangle =  \delta_{i,j}$ for $i,j=1,\ldots, n$. The duality also holds for corresponding bases of $\alt^k$; see \eqref{eq:dualkform}. 

\begin{lemma}
It holds
\begin{equation}\label{eq:Xbot}
\mathbb X^{\bot} = {\rm span} \big \{ \hat{\bs \epsilon}_{\tau}^A, \;  \tau \in \Sigma(1:k-1,1:n)\big \},
\end{equation}
where 
\begin{equation}\label{eq:epsilonA}
\hat{\bs \epsilon}_{\tau}^A: = \sum_{i\in \tau^c}\epsilon(i,\tau) \dd \hat{y}_{i + \tau}\otimes *\dd \hat{y}_{i}.
\end{equation}
For $\omega\in \alt^{k,n-1}$ expanded in the basis $\omega = \sum_{\sigma, i} a_{\sigma,i} \dd\hat{y}_{\sigma}\otimes *\dd\hat{y}_i$, it holds 
\begin{equation}\label{eq:projXperpExpress}
P_{\mathbb X^{\bot}} \omega = \sum_{\tau\in \Sigma(1:k-1,1:n)} \frac{ \bs a_{\tau} \cdot \bs \epsilon_{\tau}}{n-k+1} \hat{\bs \epsilon}_{\tau}^A,    
\end{equation}
where $P_{\mathbb X^{\bot}}:= I - P_{\mathbb X}$.
Consequently for $\sigma\in \Sigma(1:k,1:n)$,
\begin{equation}\label{eq:projXExpress}
(P_{\mathbb X}\omega)_{\sigma,i} = 
\begin{cases}
 a_{\sigma,i}, & i\in\sigma^c, \\
 a_{\sigma,i} - \dfrac{ \bs a_{\tau} \cdot \bs \epsilon_{\tau}}{n-k+1} \epsilon(i,\tau), & i\in\sigma$, $\tau=\sigma-i.
\end{cases}
\end{equation}
\end{lemma}
\begin{proof}
Let $\eta=\sum\limits_{i=1}^n \sum\limits_{\sigma \in \Sigma(1:k,1:n)}a_{\sigma, i} \dd y_{\sigma}\otimes *\dd y_i \in \mathbb X$. 
Then 
$$
\langle \hat{\bs \epsilon}_{\tau}^A, \eta \rangle = \sum_{i \in   \tau^c } \epsilon(i, \tau) a_{i + \tau, i} = 0, \quad \quad~\tau \in \Sigma(1:k-1,1:n).
$$
That is $\hat{\bs \epsilon}_{\tau}^A \perp \mathbb X$. 
As any two constraint sequences are disjointed, $\big \{ \hat{\bs \epsilon}_{\tau}^A, \tau \in \Sigma(1:k-1,1:n)\big \}$ is linear independent and~\eqref{eq:Xbot} follows from the dimensions match.
 
Let $\bs \epsilon_{\tau}^A= \sum\limits_{i\in \tau^c}\epsilon(i,\tau) \dd y_{i + \tau}\otimes *\dd y_{i}$, which also forms a basis of $\mathbb X^{\perp}$. Formula~\eqref{eq:projXperpExpress} holds by testing with $\bs \epsilon_{\tau}^A$
$$
\langle \omega, \bs \epsilon_{\tau}^A\rangle=\bs a_{\tau} \cdot \bs \epsilon_{\tau}, \quad \langle \hat{\bs \epsilon}_{\tau}^A, \bs \epsilon_{\tau}^A\rangle=n-k+1.
$$
Combining~\eqref{eq:projXperpExpress} and $P_{\mathbb X} = I - P_{\mathbb X^{\bot}}$ gives~\eqref{eq:projXExpress}.
\end{proof}
The constraint tensor spaces $\mathbb X$ and $\mathbb X^{\perp}$ are defined intrinsically using properties of differential forms, which is independent of choices of the basis. In the proof above, we use different bases $\{\dd y_i \}$ or $\{\dd \hat{y}_i \}$ for the ease of computing the projection.

\section{Geometric Decomposition of $H(\div)$-conforming Tensors with Constraints}\label{sec:geodecompdivtensor}
In this section, we generalize the geometric decomposition of the $H(\div)$-conforming vector finite element to the $H(\div)$-conforming tensor finite element. We decompose $\mathbb P_r(T; \mathbb X)$ into a direct sum of the tangential bubble subspace and a normal subspace. Then we present DoFs and show the $H(\div)$-conformity and the discrete inf-sup condition. 

\subsection{Decomposition of the constraint tensor space}\label{sec:decX}
We start from the tensor product of the Lagrange element with $\mathbb X$:
\begin{align*}
\mathbb P_r(T; \mathbb X)  &= \Oplus_{\ell = 0}^n\Oplus_{f\in \Delta_{\ell}(T)} \left [b_f \mathbb P_{r - (\ell +1)} (f) \otimes \mathbb X \right ].
\end{align*}
For an $\ell$-dimensional face $f\in \Delta_{\ell}(T)$, there is a matrix function $\boldsymbol{A}^f\in \mathbb R^{{n\choose k}\times n}$ satisfying the constraint $s^{k, n-1}(\boldsymbol{A}^f) = 0$. The vector $H(\div)$ element is $k=0$ for which the matrix $\boldsymbol{A}$ is degenerated to a vector of length $n$ and no constraint is imposed. For $1\leq k\leq n-1$, it is the constraint $s^{k, n-1}(\boldsymbol{A}^f) = 0$ that makes the finite element construction difficult as the constraint and the normal continuity should be satisfied simultaneously. 

As before, for a face $f\in \Delta_{\ell}(T)$, we choose a $t$-$n$ basis $\{\bs t_1^f, \ldots, \bs t_{\ell}^f, \bs n_{1}^f, \ldots, \bs n_{n-\ell}^f\}$, where the set of $\ell$ tangential vectors $\{\bs t_1^f, \ldots, \bs t_{\ell}^f\}$ is a basis of the tangent plane $\mathscr T^f$ of $f$ and the set of $n - \ell$ normal vectors $\{ \bs n_1^f, \ldots, \bs n_{n-\ell}^f\}$ forms a basis of the normal plane $\mathscr N^f$ of $f$. All basis vectors are normalized but may not be orthogonal. We write $i\in \mathscr T^f$ and $i\in \mathscr N^f$ to emphasize the range of the index. 

Inside the subspace $\mathscr T^f$, we can find a basis $\{\hat{\bs t}_1^f, \ldots, \hat{\bs t}_{\ell}^f\}$ dual to $\{\bs t_1^f, \ldots, \bs t_{\ell}^f\}$, i.e., $\hat{\bs t}\in \mathscr T^f$ and $(\hat{\bs t}_i, \bs t_j)=\delta_{i,j}$ for $i, j\in \mathscr T^f$. Similarly we have a basis $\{\hat{\bs n}_1^f, \ldots, \hat{\bs n}_{n-\ell}^f\}$ of $\mathscr N^f$ and  $(\hat{\bs n}_i, \bs n_j)=\delta_{i,j}$ for $i, j\in \mathscr N^f$. As $\mathscr T^f \perp \mathscr N^f$, the basis  $\{\hat{\bs t}_1^f, \ldots, \hat{\bs t}_{\ell}^f, \hat{\bs n}_{1}^f, \ldots, \hat{\bs n}_{n-\ell}^f\}$ is also dual to  $\{\bs t_1^f, \ldots, \bs t_{\ell}^f, \bs n_{1}^f, \ldots, \bs n_{n-\ell}^f\}$. Let $V = ( \bs t_1^f, \ldots, \bs t_{\ell}^f, \bs n_{1}^f, \ldots, \bs n_{n-\ell}^f)$ and $\hat{V} = ( \hat{\bs t}_1^f, \ldots, \hat{\bs t}_{\ell}^f, \hat{\bs n}_{1}^f, \ldots, \hat{\bs n}_{n-\ell}^f)$.

We say the basis $\{ \dd y_i^f\}$ is the basis of $\alt^{1}$ corresponding to a $t$-$n$ basis if  
\begin{equation*}
\prox_{1}(\dd y_i^f) =
\begin{cases}
\bs t_i^f & \textrm{ for } i \in \mathscr T^f,\\
\bs n_{i-\ell}^f & \textrm{ for } i\in \mathscr N^f.
\end{cases}
\end{equation*}
Then its dual  $\{ \dd\hat{y}_i^f\}$ has the vector proxy
$$
\prox_{1}(\dd\hat{y}_i^f)  = 
\begin{cases}
\hat{\bs t}_i^f & \textrm{ for } i\in \mathscr T^f,\\
\hat{\bs n}_{i-\ell}^f & \textrm{ for } i\in \mathscr N^f.
\end{cases}
$$

We extend the domain of $\prox$ and $\prox^{-1}$ to subspaces. For example, 
$\prox_{n-1}^{-1}\mathscr T^f = {\rm span}\{ \prox_{n-1}^{-1} (\hat{\bs t}_i^f), i=1,\ldots, \ell\}= {\rm span}\{ \prox_{n-1}^{-1} (\bs t_i^f), i=1,\ldots, \ell\}$. 
Then
\begin{align*}
\alt^{k,n-1} &= {\rm span} \{\dd y^f_{\sigma}\otimes *\dd y^f_{i}, \sigma \in \Sigma(1:k,1:n), i=1,\ldots, n\} \\
& = \big ( \alt^k\otimes \prox_{n-1}^{-1}\mathscr T^f\big ) \oplus \big ( \alt^k\otimes \prox_{n-1}^{-1}\mathscr N^f \big ).
\end{align*}


We introduce the concept {\em normal constraints}. A constraint with the constraint sequence $\{(\sigma_{i_m}, i_m)_{m=1,\ldots, n-k+1}\}$ is called a normal constraint if all $i_m\in \mathscr N^f$. The normal constraints will be imposed inside the normal component. Recall that we sort the constraint sequence s.t. $i_1< i_2 < \ldots < i_{n-k+1}$ and set $(\sigma_{i_1}, i_1)$ as the pair index. So for a non-normal constraint, the pair index $(\sigma_{i_1}, i_1)$, $i_1\in \mathscr T^f$ is in the tangential component for $\dim f\geq 1$. Also recall that non-pair indices are free indices. 

Define, for $f\in \Delta_{\ell}(T)$ with $0\leq \ell \leq n$,
\begin{align*}
\mathscr T^f(\mathbb X) :=& \, {\rm span} \{  \pi_{\mathbb X} ( \dd y^f_{\sigma}\otimes *\dd y^f_{i}), (\sigma, i) \text{ is free}, i\in \mathscr T^f\},\\
\mathscr N^f(\mathbb X) :=& \, {\rm span} \{  \pi_{\mathbb X} ( \dd y^f_{\sigma}\otimes *\dd y^f_{i}),  (\sigma, i) \text{ is free}, i\in \mathscr N^f \}.
\end{align*}
For $\ell = 0$, i.e., at vertex $\texttt{v}\in \Delta_0(T)$, we understand $ \mathscr T^\texttt{v}(\mathbb X) = \{ 0\}$ as no tangent plane and $\mathscr N^\texttt{v}(\mathbb X) = \mathbb X$ as $\mathscr N^f = \mathbb R^n$.

\begin{lemma}
Given a $t$-$n$ basis of a face $f\in \Delta_{\ell}(T)$, we have the following decomposition
\begin{equation*}
\mathbb X =  \mathscr T^f(\mathbb X) \oplus \mathscr N^f(\mathbb X).
\end{equation*}
Their dimensions are
\begin{align*}
\dim \mathscr T^f(\mathbb X) &= \ell {n \choose k} + {n-\ell\choose n-k+1} -  {n \choose k-1},\\
\dim \mathscr N^f(\mathbb X) &= (n - \ell) {n \choose k} - {n-\ell\choose n-k+1} .
\end{align*}
\end{lemma}
\begin{proof}
By construction, the sum is direct. It suffices to count the dimension. 
The number of constraints is $ \dim \alt^{k-1}$. By the proof of the surjectivity of $s^{k, n-1}$, all constraints are linearly independent.

Therefore
\begin{align*}
\dim \mathbb X  = \dim \alt^{n-1} \times \dim \alt^k - \dim \alt^{k-1} = n \dim \alt^k - \dim \alt^{k-1}.
\end{align*}
For each normal constraint, it will remove one index in $\mathscr N^f$. So
\begin{align*}
\dim \mathscr N^f(\mathbb X) =  (n - \ell) \times \dim \alt^k - \# \text{ normal constraints}.
\end{align*}
If a constraint is non-normal, then the pair index is in the tangential component. So 
$$
\dim \mathscr T^f(\mathbb X) =  \ell \times \dim \alt^k - \# \text{ non-normal  constraints}.
$$
Sum these two and use the fact $$\# \text{ normal constraints} + \# \text{ non-normal constraints} = \# \text{ all constraints} = \dim \alt^{k-1}$$ 
to conclude 
$\dim \mathbb X =  \dim \mathscr T^f(\mathbb X) + \dim \mathscr N^f(\mathbb X).$

The number of the normal constraints is $\displaystyle{n-\ell\choose n-k+1}$ (among $n-\ell$ indices of the normal plane, choose $n-k+1$ to form the constraint sequence $\{i_m, m=1,\ldots, n-k+1\}$) and thus the number of the non-normal constraints is $\displaystyle{n\choose k-1}-\displaystyle{n-\ell\choose n-k+1}$.
\end{proof}

\subsection{Geometric decomposition of polynomial constraint tensors}
Define the bubble polynomial space
$$
\mathbb B_r(\div,T; \mathbb X):= \mathbb P_r(T; \mathbb X)\cap \ker(\tr^{\div}).
$$ 
There is no bubble polynomial for lower degree $r = 0,1$. 
\begin{lemma}
 We have $\mathbb B_0(\div,T; \mathbb X)=\mathbb B_1(\div,T; \mathbb X)=0$.
\end{lemma}
\begin{proof}
Take $\omega\in\mathbb B_r(\div, T; \mathbb X)$ with $r=0,1$. Since $\{\bs n_{F_{0}}, \ldots, \bs n_{F_{j-1}}, \bs n_{F_{j+1}}, \ldots, \bs n_{F_{n}}\}$ form a basis of $\mathbb R^n$, and $(\tr_{F_i}\omega)(\texttt{v}_j)=0$ for $0\leq i\neq j\leq n$, we get $\omega(\texttt{v}_j)=0$. Thus, $\omega=0$.
\end{proof}

The tangential component contributes to the bubble. The normal component will contribute to the normal trace. Coupled with the bubble polynomials, we define, for $f\in \Delta_{\ell}(T)$ with $0\leq \ell \leq n$,
\begin{align*}
\mathbb B_r\mathscr T^f(\mathbb X) := b_f\mathbb P_{r-(\ell +1)}(f)\otimes \mathscr T^f(\mathbb X),\quad
\mathbb B_r\mathscr N^f(\mathbb X) := b_f\mathbb P_{r-(\ell +1)}(f)\otimes \mathscr N^f(\mathbb X).
\end{align*}

\begin{theorem}[Characterization of div bubble tensors]\label{thm:BrXdecomp}
For $r\geq 2$, it holds that
$$
\mathbb B_r(\div,T; \mathbb X)=\Oplus_{\ell = 1}^n\Oplus_{f\in \Delta_{\ell}(T)} \mathbb B_r\mathscr T^f(\mathbb X),
$$
and
$$
\tr^{\div}: \Oplus_{\ell = 0}^{n-1}\Oplus_{f\in \Delta_{\ell}(T)} \mathbb B_r\mathscr N^f(\mathbb X) \to \tr^{\div}\mathbb P_r(T; \mathbb X)
$$
is a bijection. Consequently 
$$
\dim \mathbb B_r(\div, T; \mathbb X)  = \sum_{\ell=1}^n{n+1 \choose \ell +1}  {r-1 \choose \ell} \left [  \ell {n \choose k} + {n-\ell\choose n-k+1} -  {n \choose k-1}\right ].
$$
\end{theorem}
\begin{proof}
Notice that the trace operator is applied to the second component in the tensor product $\dd y^f_{\sigma}\otimes *\dd y^f_{i}$ and $\tr_F^{\div}(*\dd y^f_{i}) = \det(V)\, \bs n_F\cdot \hat{\bs t}_i^f\dx_F = 0$ if $i\in \mathscr T^f$ and $f\in \Delta(F)$. The modification in $\pi_{\mathbb X}(\dd y^f_{\sigma}\otimes *\dd y^f_{i})$ will use the pair index in $\mathscr T^f$ and thus remains the normal trace free. 
For $f\not\in\Delta( F)$, the bubble function $b_f|_F = 0$. 

So we have verified $\Oplus_{\ell = 1}^n\Oplus_{f\in \Delta_{\ell}(T)} \mathbb B_r\mathscr T^f(\mathbb X) \subseteq \mathbb B_r(\div,T; \mathbb X)$. 

The rest is the same as Lemma~\ref{lem:divbubbletracespacedecomp}.
\end{proof}

We propose the following finite element for the constraint tensor $\mathbb X$. 
\begin{lemma}\label{cor:freeDof}
For each $f\in \Delta_{\ell}(T)$, choose a $t$-$n$ basis $\{\bs t_{1}^f, \ldots, \bs t_{\ell}^f, \bs n_{1}^f, \ldots, \bs n_{n-\ell}^f\}$ for $\ell = 1,\ldots, n-1$. Let $\{\dd y_i\}$ be the corresponding basis of $\alt^{1}$, and $\{\dd \hat{y}_i\}$ be its dual basis. 
The shape function space $ \mathbb P_r(T; \mathbb X)$ is uniquely determined by the DoFs
\begin{subequations}\label{eq:XdivDoF}
\begin{align}
\label{eq:freeXDof0}
\omega(\texttt{v}_i),  & \quad~i = 0,\ldots, n, \omega\in \mathbb X,\\
\label{eq:freeXDof}
(\omega,\eta)_f,
&\quad \eta\in \mathbb P_{r-(\ell +1)}(f)\otimes \{ \dd\hat{y}^f_{\sigma}\otimes *\dd\hat{y}^f_i \mid (\sigma,i) \text{ is free}, i\in \mathscr N^f \},\\
&\quad f\in \Delta_{\ell}(T), \ell = 1,\ldots, n-1, \notag\\
\label{eq:BXdivDof} 
(\omega, \eta)_T, &\quad \eta \in  \mathbb B_r(\div,T; \mathbb X).
\end{align}
\end{subequations}
\end{lemma}
\begin{proof}
Recall the duality 
$$
\langle \dd y^f_{\sigma}\otimes *\dd y^f_{i}, \dd\hat{y}^f_{\eta}\otimes *\dd\hat{y}^f_j \rangle = \delta_{\sigma,\eta}\delta_{i,j}
$$ for all $1\leq i, j\leq n,$ and $\sigma, \eta\in \Sigma(1:k,1:n)$. 
In $\pi_{\mathbb X}(\dd y^f_{\sigma}\otimes *\dd y^f_{i}) = \dd y^f_{\sigma}\otimes *\dd y^f_{i} - \epsilon(i, \tau)  \epsilon(i_1, \tau) \dd y^f_{\sigma_1}\otimes *\dd y^f_{i_1}$, the pair index $(\sigma_1, i_1)$ is not free and thus the duality still holds for all free indices
$$
\langle \pi_{\mathbb X}(\dd y^f_{\sigma}\otimes *\dd y^f_{i}), \dd \hat{y}^f_{\eta}\otimes *\dd\hat{y}^f_j \rangle = 
\delta_{\sigma,\eta}\delta_{i,j}, \quad (\sigma, i) \text{ and } (\eta, j) \text{ are free}. 
$$

Now assume~\eqref{eq:freeXDof0}-\eqref{eq:freeXDof} vanishes. 
For $\omega = \sum_{\text{ free} (\sigma, i)}c_{(\sigma, i)}b_f  \pi_{\mathbb X}(\dd y^f_{\sigma}\otimes *\dd y^f_{i})$, the DoF $(\omega, c_{(\sigma, i)} \dd \hat{y}^f_{\sigma}\otimes *\dd\hat{y}^f_{i})_f = 0$ will imply $c_{(\sigma, i)} = 0$ for all free indices $(\sigma,i)$ and $i\in \mathscr N^f$. Coupling with the property of the bubble function, we can prove by the forward substitution argument for $\ell = 0, 1, \ldots, n-1$ (see the proof of Lemma~\ref{lem:divbubbletracespacedecomp}), all normal components $\mathbb B_r\mathscr N^f(\mathbb X)$ for all $f\in \Delta_{\ell}(T)$ will vanish and thus only tangential components are left, i.e., $\omega \in \Oplus_{\ell = 1}^n\Oplus_{f\in \Delta_{\ell}(T)} \mathbb B_r\mathscr T^f(\mathbb X)$. 
By Theorem~\ref{thm:BrXdecomp}, $\omega\in  \mathbb B_r(\div,T; \mathbb X)$ and vanishing~\eqref{eq:BXdivDof} will imply $\omega = 0$.
\end{proof}
DoF~\eqref{eq:freeXDof} is in the spirit of the Petrov-Galerkin method, where the test function $\dd\hat{y}^f_{\sigma}\otimes *\dd\hat{y}^f_i$ is different from the trial function $\pi_{\mathbb X}(\dd y^f_{\sigma}\otimes *\dd y^f_i)$. This change is important as $\prox(\dd\hat{y}^f_{\sigma})\otimes \prox(*\dd\hat{y}^f_i) = c\prox(\dd\hat{y}^f_{\sigma}) \otimes \bs n^f_i$ for $i\in \mathscr N^f$ will contain the normal component only, which will determine the normal trace to ensure the $H(\div)$-conformity. 

In view of the vector proxy, usually we can choose an orthonormal basis for $\mathscr T^f$ so that $\bs t_i^f = \hat{\bs t}_i^f$ for $i\in \mathscr T^f$. We use the normal vector $\{\bs n_i^f\}$ to define DoFs while use its dual basis $\{\hat{\bs n}_i^f\}$ to expand the shape function. 

Similar to \cite{ChenChenHuangWei2023}, we can write out an explicit basis function
$$
\phi_{\alpha} \pi_{\mathbb X}(\dd y^f_{\sigma}\otimes *\dd y^f_{i}), \quad \alpha \in \mathbb T_r^{\ell}(\mathring{f}), (\sigma, i) \text{ is free}, f\in \Delta_{\ell}(T), \ell = 0,\ldots, n,
$$
where $\phi_{\alpha}$ is the nodal basis of Lagrange element at lattice point $\alpha$, and $\mathbb T^{\ell}_r(\mathring{f})$ is the set of lattice points whose geometric embedding is in the interior of $f$.

%
%
%

\subsection{$H(\div)$-conforming finite element spaces}
We shall glue local finite element spaces to form an $H(\div)$-conforming subspace of $H(\div,\Omega;\mathbb X)$ by choosing a global $t$-$n$ basis $\{\bs t_{1}^f, \ldots, \bs t_{\ell}^f, \bs n_{1}^f, \ldots, \bs n_{n-\ell}^f\}$, i.e., depending only on $f$ not the element containing $f$. 

\begin{theorem}[$H(\div)$-conforming finite element with global $t$-$n$ bases]\label{thm:divXunisolvence1}
For each $f\in \Delta_{\ell}(\mathcal T_h)$, $\ell = 0,\ldots, n-1$, choose a $t$-$n$ basis $\{\bs t_{1}^f, \ldots, \bs t_{\ell}^f, \bs n_{1}^f, \ldots, \bs n_{n-\ell}^f\}$ depending only on $f$. Let $\{\dd y_i\}$ be the corresponding basis of $\alt^{1}$, and $\{\dd \hat{y}_i\}$ be its dual basis. Then the following DoFs
\begin{subequations}\label{eq:divX}
 \begin{align}
\label{eq:bdXDof0}
\omega(\texttt{v}_i),  & \quad~i = 0,\ldots, n, \omega\in \mathbb X,\\
\label{eq:bdXDof1}
(\omega,\eta)_f,
&\quad \eta\in \mathbb P_{r-(\ell +1)}(f)\otimes \{ \dd\hat{y}^f_{\sigma}\otimes *\dd\hat{y}^f_i \mid(\sigma,i) \text{ is free, } i\in \mathscr N^f\},\\
&\quad f\in \Delta_{\ell}(\mathcal T_h), \ell = 1,\ldots, n-1,\notag \\
\label{eq:bubbleXDof} 
(\omega, \eta)_T, &\quad \eta \in  \mathbb B_r(\div,T; \mathbb X), T\in \mathcal T_h,  
\end{align}
\end{subequations}
will determine a space $V_h \subset H(\div,\Omega; \mathbb X)$. 
\end{theorem}
\begin{proof}
By Lemma~\ref{cor:freeDof}, on each simplex, DoFs~\eqref{eq:divX} will define a function $\omega\in \mathbb P_r(T;\mathbb X)$. We only need to verify the trace is uniquely determined by \eqref{eq:bdXDof0}-\eqref{eq:bdXDof1}. 

For a face $F\in \Delta_{n-1}(\mathcal T_h)$, we have a normal vector $\bs n_F$ depending only on $F$ and the formula 
\begin{equation*}
\langle \tr_F^{\div}\omega,  \dd\hat{y}_{\sigma}^f \rangle = \langle \omega, \dd \hat{y}_{\sigma}^f \otimes \star\dd \bs n_F \rangle. 
\end{equation*}
For $f\in \Delta(F)$, as $\bs n_F\in \mathscr N^f$, we can expand $\bs n_F = \sum_{i=1}^{n-\ell} c_i^f \bs n_i^f$ and $\star\dd \bs n_F = \sum_{i=1}^{n-\ell} c_i^f \star \dd \bs n_i^f =  \sum_{i\in \mathscr N^f} \tilde c_{i}^f * \dd\hat{y}_{i}^f$. 

When $(\sigma, i)$ is free and $i\in \mathscr N^f$, $(\omega, \dd\hat{y}^f_{\sigma}\otimes *\dd\hat{y}^f_i)_f$ is given by the DoFs \eqref{eq:bdXDof0}-\eqref{eq:bdXDof1}. 

Then consider the case $(\sigma, i)$ is not free and $i\in \mathscr N^f$. Namely $(\sigma, i) = (i_1 + \tau, i_1)$ is a pair index for a constraint sequence. As $i_1 = i\in \mathscr N^f$, this constraint is a normal constraint, i.e., all $i_m\in \mathscr N^f$ for $m=1, \ldots, n-k+1$. 
We can express $\epsilon(i_1,\tau)\dd\hat{y}^f_{\sigma}\otimes *\dd\hat{y}^f_i = \hat{\bs \epsilon}_{\tau}^A - \sum_{m=2}^{n-k+1}\epsilon(i_m, \tau) \dd\hat{y}^f_{i_m + \tau}\otimes *\dd\hat{y}^f_{i_m}$, where $\hat{\bs \epsilon}_{\tau}^A\in \mathbb X^{\perp}$ is defined in \eqref{eq:epsilonA}. As $\omega \in \mathbb X$ and $\hat{\bs \epsilon}_{\tau}^A\in \mathbb X^{\perp}$, $\langle \omega, \hat{\bs \epsilon}_{\tau}^A\rangle = 0$ and consequently
$$
(\omega, \epsilon(i_1,\tau) \dd\hat{y}^f_{\sigma}\otimes *\dd\hat{y}^f_i)_f = - \sum_{m=2}^{n-k+1}\epsilon(i_m, \tau) (\omega, \dd\hat{y}^f_{i_m + \tau}\otimes *\dd\hat{y}^f_{i_m})_f.
$$
Notice that the index $(i_m+\tau, i_m)$ is free and $i_m\in \mathscr N^f$ for $m\geq 2$. So $(\omega, \dd\hat{y}^f_{\sigma}\otimes *\dd\hat{y}^f_i)_f$ can be also determined by DoFs \eqref{eq:bdXDof0}-\eqref{eq:bdXDof1} even $(\sigma, i)$ is not free.  

It follows that
$$
(\tr_F^{\div}\omega,\eta)_f,
\quad \quad~f\in \Delta_{\ell}(F),\eta\in \mathbb P_{r-(\ell +1)}(f)\otimes\alt^k,  \ell = 0,\ldots, n-1.
$$
can be determined by DoFs \eqref{eq:bdXDof0}-\eqref{eq:bdXDof1}.  As $\tr_F^{\div}\omega\in \mathbb P_r(F)\otimes \alt^k$, by the uni-solvence of the vector Lagrange element, we conclude $\tr_F^{\div}\omega$ is uniquely determined by DoFs \eqref{eq:bdXDof0}-\eqref{eq:bdXDof1}. 
\end{proof}

DoF~\eqref{eq:bdXDof0} implies the continuity at vertices. %
We argue that the continuity at vertices is also necessary. 
Take a vertex in $\Delta_{0}(T)$, for example $\texttt{v}_0$. Then $(\bs A \bs n_{F_{i}})(\texttt{v}_0)$ is determined by the row vector $(\bs A \bs n_{F_{i}})|_{F_{i}}\in \Lambda^k$ for $i=1,\ldots, n$, where $F_{i}\in\Delta_{n-1}(T)$ is the face opposite to $\texttt{v}_i$. If it is continuous on each face not on vertices, the number of elements in $(\bs A\bs n_{F})(\texttt{v}_0)$ is $\dim \alt^k$ for each face. Running $i$ from $1$ to $n$, $\bs A(\texttt{v}_0)$ is determined by $n\dim \alt^k$ conditions, which is more than $\dim \mathbb X$. In other words, the constraint makes the tensor product of vector DoFs fails and introduce additional smoothness. 

\subsection{Facewise redistribution}
DoFs \eqref{eq:bdXDof0}-\eqref{eq:bdXDof1} implies stronger continuity on the normal plane. We shall further redistribute some DoFs to faces $F\in \Delta_{n-1}(T)$. 

To do so, we introduce the concept of the {\em free block} and the {\em normal constraint block}. For a fixed $\sigma$, the row vector $(\sigma,i), i=\ell + 1:n$ is called a {\em free row} if no index is in a normal constraint. All free rows will form a sub-matrix called the {\em free block}. The rest is called the {\em normal constraint block}, which contains all normal constraints. Do not confuse the free row with the free index. All indices of a free row are free. But a row with all free indices may not be a free row. A free index can be associated to a normal constraint sequence. See the second and third rows in Fig.~\ref{fig:constraintblock}. 

\begin{figure}[htbp]
\begin{center}
\includegraphics[width=8.75cm]{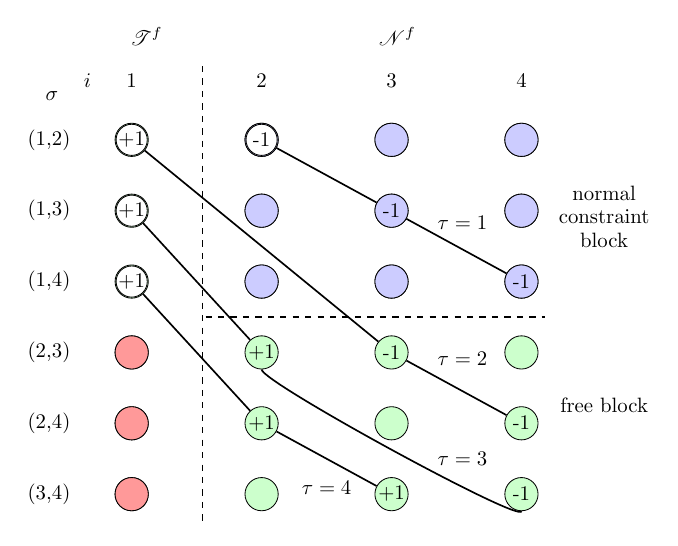}
\caption{The normal constraint block (in blue) and the free block (in green) for an $f\in \Delta_1(T)$ with $n=4, k=2$. In the free block, no index is in a normal constraint. In the normal constraint block, each row contains at least one index in a normal constraint. The red circles will contribute to the bubble spaces. The white circle denotes the pair index of each constraint sequence which is not a free index.}
\label{fig:constraintblock}
\end{center}
\end{figure}

For indices in the free block, the pair index is in the tangential component, which will not change the normal trace. For a normal constraint, the pair index is still in the normal constraint block and changing a basis of $\mathscr N^f$ may destroy the $H(\div)$-conformity. Therefore a global basis $\{\bs n_i^f\}$ of $\mathscr N^f$ is chosen to impose the normal constraints for the normal constraint block. For example, DoFs~\eqref{eq:divfemSkdof2} for $\bs n_i^{\intercal}\bs\sigma \bs n_j$ in the $H(\div,\mathbb S)$ element are normal constraints.

For a free row, we have a vector in $\mathbb R^n$ consists of $[a_{\sigma,1}, \ldots, a_{\sigma,\ell}, a_{\sigma,\ell + 1}, \ldots, a_{\sigma,n}]^{\intercal}$. The first $\ell$ components are in the tangential component and their values are determined locally as the element-wise div bubble polynomials. The part $[a_{\sigma,\ell + 1}, \ldots, a_{\sigma,n}]^{\intercal}$ is in the normal component and the corresponding DoFs can be redistributed facewisely by using the face normal basis $\{\bs n_F, f\subseteq F, F\in \Delta_{n-1}(T)\}$ of $\mathscr N^f$. In short, a free row is just like a vector $H(\div)$ element. 

\begin{theorem}[$H(\div)$-conforming finite element with face redistribution]\label{thm:PrXmorelocal}
For each $f\in \Delta_{\ell}(\mathcal T_h)$, $\ell = 0,\ldots, n-1$, choose a global  $t$-$n$ basis $\{\bs t_{1}^f, \ldots, \bs t_{\ell}^f, \bs n_{1}^f, \ldots, \bs n_{n-\ell}^f\}$ depending only on $f$. Let $\{\dd y_i\}$ be the corresponding basis of $\alt^{1}$, and $\{\dd \hat{y}_i\}$ be its dual basis. The following DoFs
\begin{subequations}\label{eq:divXredis}
\begin{align}
\label{eq:bdXXXDof0}
\omega(\texttt{v}_i),  & \quad~i = 0,\ldots, n, \omega\in \mathbb X,\\
\label{eq:bdXXXDof1}
(\omega,\eta)_f,
&\quad \eta\in \mathbb P_{r-(\ell +1)}(f)\otimes\{ \dd\hat{y}^f_{\sigma}\otimes *\dd\hat{y}^f_i \mid (\sigma, i) \text{ is free and in the normal constraint block} \}, \\
&\quad   f\in \Delta_{\ell}(\mathcal T_h), \ell = 1,\ldots, n-1, \notag\\ 
\label{eq:bdXXXDof2}
(\tr_F^{\div}\omega,\eta)_f,
&\quad \eta\in \mathbb P_{r-(\ell +1)}(f)\otimes \{ \dd\hat{y}^f_{\sigma} \mid \sigma \text{ is in the free block} \},\\
&\quad   F\in \Delta_{n-1}(\mathcal T_h), f\in \Delta_{\ell}(F),\ell = 1,\ldots, n-1, \notag\\
 \label{eq:bdXXXDof3}
(\omega, \eta)_T, &\quad \eta \in  \mathbb B_r(\div,T; \mathbb X), T\in \mathcal T_h,
\end{align}
\end{subequations}
will determine a space $V_h \subset H(\div,\Omega; \mathbb X)$.
\end{theorem}
\begin{proof}
\step{1} Local unisolvence on an element. 
We write DoF~\eqref{eq:bdXXXDof2} as
\begin{equation}\label{eq:tromegadxsigma}
(\tr_F^{\div}\omega,\dd\hat{y}^f_{\sigma})_f=(\omega, \dd\hat{y}^f_{\sigma}\otimes\star\dd \bs n_F^{T})_f, \quad f\subseteq F,
\end{equation}
where $\{\bs n_F^T, f\subseteq F, F\in \Delta_{n-1}(T)\}$ is the face normal basis of $\mathscr N^f$ in element $T$. As $\{\bs n_{F_i}^{T}, i\in f^*\}$ and $\{\bs n^f_j, j=1,\ldots, n-\ell\}$ are different bases of the same space $\mathscr N^f$, DoF~\eqref{eq:tromegadxsigma} will  also determine 
\begin{equation}\label{eq:dyif}
(\omega, \dd\hat{y}^f_{\sigma}\otimes*\dd \hat{y}_i^f)_f, \quad i = \ell +1, \ldots, n.
\end{equation}
Together with~\eqref{eq:bdXXXDof1}, we obtain DoFs~\eqref{eq:bdXDof1}. The number of DoFs remains the same as in a free row all indices are free. 
Then we conclude the unisolvence from Theorem~\ref{thm:divXunisolvence1}. 

\step{2} Global conformity across elements. The continuity 
$$
(\tr_F^{\div}\omega,\eta)_f, \quad \eta\in \mathbb P_{r-(\ell +1)}(f)\otimes \{ \dd \hat{y}^f_{\sigma} \mid \sigma \text{ is in the normal constraint block} \}
$$
is implied by DoFs \eqref{eq:bdXXXDof0}-\eqref{eq:bdXXXDof1} as the global  $t$-$n$ basis $\{\bs t_{1}^f, \ldots, \bs t_{\ell}^f, \bs n_{1}^f, \ldots, \bs n_{n-\ell}^f\}$ depending only on $f$. Together with DoF \eqref{eq:bdXXXDof2}, we conclude $\tr_F^{\div}\omega$ is continuous. 
\end{proof}

\begin{remark}\rm 
It seems that we can also try to redistribute the rows in the normal constraint block as DoF \eqref{eq:dyif} can be derived from DoF \eqref{eq:tromegadxsigma}. The problem comes from the fact the basis $\{\bs n_F^T, f\subseteq F, F\in \Delta_{n-1}(T)\}$ is element-dependent. Expand a global basis vector $\bs n_i^f = \sum_{F}c_F^T \bs n_F^T$ will let \eqref{eq:dyif} be element dependent. For example, for $H(\div,\mathbb S)$ element, $\bs n_i^{\intercal}\bs\tau\bs n_j$ cannot be redistributed to faces. In the above proof, we transfer \eqref{eq:bdXXXDof2} to  \eqref{eq:dyif} only for the ease of uni-solvence.
\end{remark}


Following the management of DoFs presented in \cite{ChenChenHuangWei2023}, we need to set global and local indexing rules for all DoFs. The global numbering rule is similar to the Lagrange interpolation points.
Globally, we can divide the DoFs into those that are shared among simplices and those that are not. The DoFs shared among simplices can be further allocated to the respective sub-simplices. For the DoFs situated on a sub-simplex \( f \), we can choose global normal vectors \( \bs n_{i}^f \) which share a global DoF labeling, and local tangential vectors \( \bs t_i^f \) which have different labeling in different elements.

We count the size of the normal constraint block. The normal constraint block disappears when $\ell\geq k$ as the length of the constraint $n-k+1$ will be greater than $n - \ell$ the dimension of the normal plane. That is when $\ell \geq k$, all rows are free and corresponding DoFs can be redistributed to faces. 

Consider the case $\ell< k$. 
If $\sigma\in\Sigma(1:k,1:n)$ is in the constraint block, there exists some $i>\ell$ such that $\tau = \sigma-i\in\Sigma(1:k-1,1:n)$ satisfies $\tau^c\subseteq[\ell+1:n]$, which is equivalent to $\sigma^c\subseteq[\ell+1:n]$.  Hence
the number of rows in the constraint block is $\displaystyle{n-\ell\choose n-k}$: among all $n-\ell$ indices of the normal plane, choose $n-k$ indices to form $\sigma^c$. When $\ell = 0$, $\displaystyle{n\choose n-k} = {n \choose k}$, i.e., all rows belong to the constraint block. Consequently DoFs at vertices cannot be redistributed facewisely. We thus give another justification of the continuity at vertices.

\subsection{Discrete inf-sup condition}

%
For a smooth tensor $u=(u_{\sigma})$ with index $\sigma \in \Sigma(1:k,1:n)$, let $\grad u$ be a tensor with size ${{n \choose k}\times n }$ give by 
$$
(\grad u)_{\sigma,i}:=\partial_{x_i}u_{\sigma}.
$$
\begin{lemma}
It holds that
$$
\ker(P_{\mathbb X}\grad)=\mathbb P_0(T;\mathbb R^{{n \choose k} })+ \mathbb X^{\bot}x.
$$
\end{lemma}
\begin{proof}
Noting that $\grad(\mathbb P_0(T;\mathbb R^{{n \choose k} })+ \mathbb X^{\bot}x)=\mathbb X^{\bot}$, hence 
$$
\mathbb P_0(T;\mathbb R^{{n \choose k} })+ \mathbb X^{\bot}x\subseteq \ker(P_{\mathbb X}\grad).
$$
By (31) in~\cite{arnoldComplexesComplexes2021}, $\ker(P_{\mathbb X}\grad)\subseteq\mathbb P_1(T;\mathbb R^{{n \choose k} })$. Take $c+Ax\in\ker(P_{\mathbb X}\grad)$. By $\grad(c+Ax)=A$, we have $P_{\mathbb X} A=0$, i.e. $A\in\mathbb X^{\bot}$. Therefore
$\ker(P_{\mathbb X}\grad)\subseteq\mathbb P_0(T;\mathbb R^{{n \choose k} })+ \mathbb X^{\bot}x.$
\end{proof}

We introduce notation ${\rm RX}:= \ker (P_{\mathbb X}\grad)$. Examples are ${\rm RX} = {\rm RM}$ the rigid motion for $\mathbb X = \mathbb S$, and ${\rm RX} = {\rm RT}$ for $\mathbb X = \mathbb T$, where ${\rm RT}:=\mathbb P_0(T;\mathbb R^n)+x\mathbb P_0(T)$ is the lowest Raviart-Thomas element. In general, ${\rm RX}$ is the Whitney form $\mathbb P_0\Lambda^k + \kappa_{k+1} \mathbb P_0\Lambda^{k+1}$, which is another characterization of $\ker (P_{\mathbb X}\grad)$.

Operator $P_{\mathbb X}\grad$ is the proxy of $P_{\mathbb X} \dd :=(-1)^{n-1}\star P_{\mathbb X}\star\dd: \Lambda^{k,0}\to \Lambda^{k,1}$. Indeed $\int_T \omega \wedge P_{\mathbb X} \dd \eta=(-1)^{n-1}\int_T \omega \wedge \star\star\dd \eta=\int_T \omega \wedge\dd \eta$, then the integration by parts holds
\begin{equation}\label{eq:ibp}
\int_T \dd \omega \wedge \eta = (-1)^{n}\int_T \omega \wedge P_{\mathbb X} \dd \eta + \int_{\partial T} \tr^{\div}\omega \wedge \eta
\end{equation}
for any $\omega \in H(\div, T;\mathbb X)$ and $\eta\in \Lambda^{k,0}$.
In the matrix and vector proxy, we have
$$
\int_T (\div \bs A) \cdot \bs u \dx = -\int_T \bs A: P_{\mathbb X}\grad \bs u\dx + \int_{\partial T} (\bs A\bs n)\cdot \bs u \dd S.
$$

We consider the finite elements defined in Theorem~\ref{thm:PrXmorelocal}. 
Define the global finite element space
\begin{align}
\notag
V_h:=\{&\omega_h\in L^2(\Omega;\mathbb X):\omega_h|_T\in \mathbb P_r(T; \mathbb X)\quad\forall~T\in\mathcal T_h,  \\
\notag
&\textrm{ the DoFs~\eqref{eq:bdXXXDof0}-\eqref{eq:bdXXXDof1} is single-valued across } f\in\Delta_{\ell}(\mathcal T_h) \textrm{ for } \ell=0,\ldots, n-1,\\
\notag
&\quad\quad\qquad\qquad\quad\quad\quad\quad\;\textrm{ the DoF~\eqref{eq:bdXXXDof2} is single-valued across } F\in\Delta_{n-1}(\mathcal T_h) \},    \\
\label{eq:QhLambdak}
Q_h:=&\{q_h\in L^2(\Omega;\Lambda^{k}):q_h|_T\in \mathbb P_{r-1}(T; \Lambda^{k})\quad\forall~T\in\mathcal T_h\}.
\end{align}
Thanks to Theorem~\ref{thm:PrXmorelocal}, $V_h\subset H(\div,\Omega;\mathbb X)$.
We are going to verify the discrete inf-sup condition $\div V_h=Q_h$ if $r\geq n+1$. The following characterization of the range of the div operator on the bubble polynomial space is an abstract version of results \eqref{eq:TdivBrsurjection} and \eqref{eq:SdivBrsurjection} established in~\cite{Hu2015,HuZhang2015,HuLiang2021}.

\begin{lemma}\label{lm:divbubbleX}
For each $T\in\mathcal T_h$,
it holds
\begin{equation}\label{eq:divBrsurjection}
\div \mathbb B_r(\div, T; \mathbb X) = \mathbb P_{r-1}(T;\Lambda^{k})\cap{\rm RX}^{\perp}.
\end{equation}
\end{lemma}
\begin{proof}
When $r = 0, 1$, \eqref{eq:divBrsurjection} is obviously true as both sides are zero. We thus consider $r\geq 2$. 
 
Apply the integration by parts~\eqref{eq:ibp}
to get $$\div \mathbb B_r(\div, T; \mathbb X)\subseteq (\mathbb P_{r-1}(T;\Lambda^{k})\cap{\rm RX}^{\perp}).$$
Next we focus on the proof of the other side 
$$(\mathbb P_{r-1}(T;\Lambda^{k})\cap{\rm RX}^{\perp})\subseteq\div \mathbb B_r(\div, T; \mathbb X).$$
For simplicity, write $\dd\lambda_{\sigma^{*}}\otimes\dd\lambda_{[\sigma(0), \sigma(i)]^*}$ as $\phi_{\sigma,i}$ for each $\sigma\in\Sigma(0:n-k, 0:n)$ and $i=1,\ldots, n-k$.
By Theorem~\ref{thm:basisofX}, $\{\phi_{\sigma,i}\}_{i=1,\ldots, n-k}^{\sigma\in\Sigma(0:n-k, 0:n)}$
is a basis of $\mathbb X$, whose dual basis (appropriate rescaling of~\eqref{eq:dualbasis}) is denoted by $\{\psi_{\tau,j}\}_{j=1,\ldots, n-k}^{\tau\in\Sigma(0:n-k, 0:n)}$, that is $\psi_{\tau,j}\in\mathbb X$, $\langle \phi_{\sigma,i}, \psi_{\tau,j} \rangle = 1$ for $\sigma=\tau$ and $i=j$, otherwise it vanishes. 

Consider the edge $e = \texttt{e}_{\sigma(0), \sigma(i)}$. 
The vector proxy of $\dd\lambda_{[\sigma(0), \sigma(i)]^*}$ is proportional to $\bs t^{e}$. Coupled with the edge bubble function $b_e = \lambda_{\sigma(0)}\lambda_{\sigma(i)}$, the vector function $b_e\bs t^{e}$ satisfies 
$$
\bs n_F\cdot b_e\bs t^{e}|_F = 0, \quad \quad~F\in \Delta_{n-1}(T),
$$
as if the edge $e\not\subseteq F$, then $b_e|_F = 0$; otherwise $\bs n_F\cdot \bs t^{e} = 0$. Therefore $\lambda_{\sigma(0)}\lambda_{\sigma(i)}\phi_{\sigma,i}\in \mathbb B_2(\div,T;\mathbb X)$. 

If $(\mathbb P_{r-1}(T;\Lambda^{k})\cap{\rm RX}^{\perp})\not\subseteq\div \mathbb B_r(\div, T; \mathbb X)$, then there exists $u\in \mathbb P_{r-1}(T;\Lambda^{k})\cap{\rm RX}^{\perp}$ satisfying $(u, \div\omega)_T=0$ for any $\omega\in\mathbb B_r(\div, T; \mathbb X)$. Equivalently 
$$
(P_{\mathbb X}\grad u, \omega)_T=0\quad\forall~\omega\in\mathbb B_r(\div, T; \mathbb X).
$$
By expressing $P_{\mathbb X}\grad u=\sum\limits_{\sigma\in\Sigma(0:n-k, 0:n)}\sum\limits_{i=1}^{n-k}q^{\sigma,i}\psi_{\sigma,i}$ with $q^{\sigma,i}\in \mathbb P_{r-2}(T)$, we choose 
$$
\omega=\sum\limits_{\sigma\in\Sigma(0:n-k, 0:n)}\sum\limits_{i=1}^{n-k}\lambda_{\sigma(0)}\lambda_{\sigma(i)}q^{\sigma,i}\phi_{\sigma,i}\in\mathbb B_r(\div, T; \mathbb X).
$$
Then we have
$$
\sum_{\sigma\in\Sigma(0:n-k, 0:n)}\sum_{i=1}^{n-k}(\lambda_{\sigma(0)}\lambda_{\sigma(i)}q^{\sigma,i}, q^{\sigma,i})_T=0.
$$
Therefore $q^{\sigma,i}=0$ for all $i$ and $\sigma$ and consequently $u=0$.
\end{proof}

Employing the same argument as the proof of~\eqref{eq:discretedivinfsup}, the discrete inf-sup condition follows from~\eqref{eq:divBrsurjection}.
\begin{lemma}\label{lem:divXdiscreteinfsup}
Let $r\geq n+1$. Let $V_h \subset H(\div,\Omega; \mathbb X)$ be the finite element space defined in Theorem~\ref{thm:divXunisolvence1} and $Q_h$ be the space defined in~\eqref{eq:QhLambdak}. It holds 
$$
\div V_h=Q_h.
$$
\end{lemma}
\begin{proof}
Take $q_h\in Q_h$. By (37) in~\cite{arnoldComplexesComplexes2021}, $\div H^1(\Omega;\mathbb X)=L^2(\Omega;\Lambda^{k})$. Then there exists $\omega\in H^1(\Omega;\mathbb X)$ satisfying $\div\omega=q_h$.
Let $\omega_1$ be the Scott-Zhang interpolation~\cite{ScottZhang1990} of $\omega$ in the tensor $r$th order Lagrange element space $S_h^r\otimes\mathbb X$ such that
$$
(\tr_F^{\div}(\omega-\omega_1),\eta)_F=0,
\quad \quad~\eta\in \mathbb P_{r-n}(F)\otimes\alt^k, F\in \Delta_{n-1}(\mathcal T_h).
$$
Since $r\geq n+1$, ${\rm RX}\subset\mathbb P_1(T; \alt^{k})\subseteq \mathbb P_{r-n}(T;\alt^{k})$, the last equation together with the integration by parts implies
$\div(\omega-\omega_1)|_T\in\mathbb P_{r-1}(T;\alt^{k})\cap{\rm RX}^{\perp}$. 
By~\eqref{eq:divBrsurjection}, there exists $\omega_2\in V_h$ such that $\omega_2|_T\in \mathbb B_r(\div, T; \mathbb X)$ for each $T\in\mathcal T_h$ and $\div\omega_2=\div(\omega-\omega_1)$.
Finally take $\omega_h=\omega_1+\omega_2\in V_h$ to get $\div\omega_h=q_h$.
\end{proof}

For $\ell \geq k$, all DoFs can be redistributed facewisely as no normal constraint block exists. 
We can further modify the DoFs to get the discrete inf-sup condition with degree $r\geq k+1$ relaxing the requirement $r\geq n+1$ for $k=1,\ldots, n-2$.
\begin{theorem}[$H(\div)$-conforming finite element with a better inf-sup condition]
\label{th:CHspace2}
Let $1\leq k\leq n-2$
and $r\geq k+1$.
For each $f\in \Delta_{\ell}(\mathcal T_h)$, $\ell = 0,\ldots, n-1$, choose a global  $t$-$n$ basis $\{\bs t_{1}^f, \ldots, \bs t_{\ell}^f, \bs n_{1}^f, \ldots, \bs n_{n-\ell}^f\}$ depending only on $f$. Let $\{\dd y_i\}$ be the corresponding basis of $\alt^{1}$, and $\{\dd \hat{y}_i\}$ be its dual basis. The {\rm DoFs}
\begin{subequations}\label{eq:newX}
\begin{align}
\label{eq:newbdXDof0}
\omega(\texttt{v}_i),  & \quad~i = 0,\ldots, n, \omega\in \mathbb X,\\
\label{eq:newbdXDof1}
(\omega,\eta)_f,
&\quad \eta\in \mathbb P_{r-(\ell +1)}(f)\otimes\{ \dd\hat{y}^f_{\sigma}\otimes *\dd\hat{y}^f_i \mid (\sigma,i) \text{ is free and in the constraint block} \}, \\
&\quad  f\in \Delta_{\ell}(\mathcal T_h), \ell = 1,\ldots, k-1,  \notag\\
\label{eq:newbdXDof2}
(\tr_F^{\div}\omega,\eta)_f,
&\quad \eta\in \mathbb P_{r-(\ell +1)}(f)\otimes \{ \dd\hat{y}^f_{\sigma} \mid \sigma \text{ is in the free block} \},\\
&\quad F\in \Delta_{n-1}(\mathcal T_h), f\in \Delta_{\ell}(F),\ell = 1,\ldots, k-1, \notag\\
\label{eq:newbdXDof3}
(\tr_F^{\div}\omega,\eta)_F,
&\quad \eta\in [\mathbb P_{1}(F)\oplus (\mathbb B_{r,k}(F)\cap \mathbb P_1^{\perp}(F))]\otimes \alt^k, F\in \Delta_{n-1}(\mathcal T_h), \\
\notag
(\omega, \eta)_T, &\quad \eta \in  \mathbb B_r(\div,T; \mathbb X), T\in\mathcal T_h, 
\end{align}
\end{subequations}
will determine a space $V_h \subset H(\div,\Omega; \mathbb X)$.
\end{theorem}
\begin{proof}
The condition $k\leq n-2$ is to ensure 
$$
\dim \mathbb B_{k+1,k}(F) = |\Delta_k(F)| = {n\choose k+1}\geq n = \dim\mathbb P_{1}(F)\quad\; \textrm{ if } k\leq n-2.
$$
So that we can modify the face DoF to \eqref{eq:newbdXDof3}. 
Vanishing DoFs \eqref{eq:newbdXDof0}-\eqref{eq:newbdXDof2} will imply $\tr_F^{\div}\omega\in \mathbb B_{r,k}(F)\otimes \alt^k$, which can be decomposed into \eqref{eq:newbdXDof3}.
\end{proof}

As $\eta \in \mathbb P_1(F)\otimes \alt^k$ is included in DoF~\eqref{eq:newbdXDof3}, 
we acquire the following discrete inf-sup condition by applying the same argument as the proof of Lemma~\ref{lem:divXdiscreteinfsup}.
\begin{corollary}
Let $1\leq k\leq n-2$
and $r\geq k+1$. Let $V_h \subset H(\div,\Omega; \mathbb X)$ be the finite element space defined in Theorem~\ref{th:CHspace2} and $Q_h$ be the space defined in~\eqref{eq:QhLambdak}. It holds 
$$
\div V_h=Q_h.
$$
\end{corollary}

\bibliographystyle{abbrv}
\bibliography{FEEC}
\end{document}